\documentclass{amsart}
\usepackage[margin=1in]{geometry}

\usepackage{eurosym}
\usepackage{amsmath,amssymb,amsthm}
\usepackage{amsfonts,amsrefs}
\usepackage{url,hyperref}
\usepackage{graphicx,caption}

\theoremstyle{plain} 
\newtheorem{theorem}{Theorem}[section]
\newtheorem{corollary}[theorem]{Corollary}
\newtheorem{lemma}[theorem]{Lemma}

\theoremstyle{definition}
\newtheorem{definition}[theorem]{Definition}
\newtheorem{remark}[theorem]{Remark}
\newtheorem{example}[theorem]{Example}

\usepackage{chngcntr}
\counterwithin{figure}{section}

\usepackage[T3,T1]{fontenc}
\DeclareSymbolFont{tipa}{T3}{cmr}{m}{n}
\DeclareMathAccent{\invbreve}{\mathalpha}{tipa}{16}

\title[Coefficients of Catalan States of Lattice Crossing II]{Coefficients of Catalan States of Lattice Crossing II: \\ Applications of $\Theta_{A}$-state Expansions}

\author{Mieczyslaw K. Dabkowski}
\address{Department of Mathematical Sciences, The University of Texas at Dallas, Richardson, TX 75080}
\email{mdab@utdallas.edu}

\author{Cheyu Wu}
\address{Department of Mathematical Sciences, The University of Texas at Dallas, Richardson, TX 75080}
\email{cheyu.wu@utdallas.edu}

\begin{document}

\maketitle

\begin{abstract}
Plucking polynomial of a plane rooted tree with a delay function $\alpha$ was introduced in 2014 by J.H.~Przytycki. As shown in this paper, plucking polynomial factors when $\alpha$ satisfies additional conditions. We use this result and $\Theta_{A}$-state expansion introduced in our previous work to derive new properties of coefficients $C(A)$ of Catalan states $C$ resulting from an $m \times n$-lattice crossing $L(m,n)$. In particular, we show that $C(A)$ factors when $C$ has arcs with some special properties. In many instances, this yields a more efficient way for computing $C(A)$. As an application, we give closed-form formulas for coefficients of Catalan states of $L(m,3)$.
\end{abstract}

%\tableofcontents

\section{Introduction}
\label{s:intro}

Skein modules were introduced by J.H.~Przytycki \cites{Prz1999,Prz1991} in 1987 and later by V.~Turaev \cite{Tur1990} in 1988. For $M^{3} = \Sigma_{g,n} \times I$, where $\Sigma_{g,n}$ is an oriented surface of genus $g$ with $n$ boundary components and $I = [0,1]$, the Kauffman Bracket Skein Module of $M^{3}$ is also an algebra called the Kauffman Bracket Skein Algebra (KBSA) of $\Sigma_{g,n}$. In 1999, D.~Bullock proved that these algebras are finitely generated \cite{Bul1999} and later in 2000, D.~Bullock and J.H.~Przytycki studied KBSAs of $\Sigma_{g,n}$ for $(g,n)\in \{(1,0),(1,1),(1,2),(0,4)\}$ in \cites{BP2000}. The recent growing interest in KBSAs resulted in several publications (see \cites{PS2019,FK2018,Le2018,FKL2019}) including also an earlier important result called \emph{``product-to-sum formula''} found by C.~Frohman and R.~Gelca in \cite{FG2000} for multiplication of curves on $\Sigma_{1,0}$. The first result concerning product of curves on $\Sigma_{0,4}$ (a presentation for KBSA of $\Sigma_{0,4}$) was obtained by D.~Bullock and J.H.~Przytycki in \cites{BP2000}. In 2018, R.P.~Bakshi et al. gave an algorithm in \cite{BMPSW2018} for computing product of curves on $\Sigma_{0,4}$ and, in particular, closed-form formulas for multiplication of two special families of curves. As these formulas are rather involved, it is worth to understand product of curves on $\Sigma_{0,4}$ locally. Such an approach was originated in \cite{DLP2015}, where the problem of finding coefficients $C(A)$ of Catalan states $C$ of an $m \times n$-lattice crossing $L(m,n)$ was formulated. This led to several new developments, e.g., plucking polynomial of plane rooted trees (see \cite{Prz2016} and \cites{CMPWY2017,CMPWY2018,CMPWY2019}), formula for coefficients of Catalan states with no returns on one side (see \cite{DP2019}), closed-form formulas for coefficients of some infinite families of Catalan states (see \cite{DM2021}), and coefficients of some B-type Catalan states (see \cite{DR2021}). Finally, a new method that allows us to compute $C(A)$ for any Catalan state $C$, called a $\Theta_{A}$-state expansion, was introduced in \cite{DW2022}. As we show in this paper, $\Theta_{A}$-state expansions can also be used to study properties of $C(A)$. In particular, we show that if $C$ has a removable arc $c$ (see Definition~\ref{def:removable_arc}), then $C(A)$ up to a power of $A$ can be found by computing $C'(A)$, where $C' = C\smallsetminus c$ is the Catalan state obtained from $C$ after removing $c$. Furthermore, if $C$ has a local family of arcs $\Lambda$ that vertically factorizes $C$ (see Definition~\ref{def:vertical_factor}), then $C(A)$ factors into a product of coefficients of two ``simpler" Catalan states determined by $\Lambda$ and its complement in the set of arcs of $C$, respectively. Proofs for both results use $\Theta_{A}$-state expansion, a new property of plucking polynomial obtained in this paper, and a connection between coefficients of Catalan states with no bottom returns and plucking polynomial found in \cite{DP2019}.

The paper is organized as follows. In Section~\ref{s:pre}, a brief summary of necessary notions and results is given. In Section~\ref{s:factor_plucking_poly}, we prove in Theorem~\ref{thm:prod_formula_plucking} that plucking polynomial of a plane rooted tree with delay can be factored when its rooted subtree satisfies conditions of Definition~\ref{def:splitting_subtree}. In Section~\ref{s:removable_arc_thm}, we prove the \emph{Removable Arc Theorem} (see Theorem~\ref{thm:remove_an_arc}). Its proof is derived in several steps, which include a new formula given in Corollary~\ref{cor:lr_relations} for the sum of terms of the maximal sequence $\mathbf{b}$ for $C$ and, as a by-product, a formula given in Corollary~\ref{cor:b_seq} for each term of $\mathbf{b}$. As an application of Theorem~\ref{thm:remove_an_arc}, we give closed-form formulas for coefficients of Catalan states of $L(m,3)$. In Section~\ref{s:vertical_factor_thm}, we prove the \emph{Vertical Factorization Theorem} (see Theorem~\ref{thm:vertical_factor_thm}). This result is established after existence of a $\Theta_{A}$-state expansion with additional properties is shown in Lemma~\ref{lem:theta_state_expn_new}.

\section{Preliminaries}
\label{s:pre}

Given a rectangle $\mathrm{R}^{2}_{m,n,2k-n}$ with $2(m+k)$ points fixed on its boundary (see Figure~\ref{fig:LC_intro}(a)), a \emph{crossingless connection} $C$ in $\mathrm{R}^{2}_{m,n,2k-n}$ consists of $(m+k)$ arcs embedded in $\mathrm{R}^{2}_{m,n,2k-n}$ that join points on its boundary. For a crossingless connection $C$, its points on the top side (respectively, bottom, left, and right side) of $\mathrm{R}^{2}_{m,n,2k-n}$ are called top-boundary points (respectively, bottom-, left-, and right-boundary points). We denote by $\mathrm{ht}(C), n_{t}(C)$, $n_{b}(C)$ the number of left-, top-, bottom-boundary points of $C$ and we label boundary points of $C$ by $x_{i}, x'_{i}, y_{j}$, and $y'_{j}$ (see Figure~\ref{fig:LC_intro}(a)). An arc of $C$ joining a pair of top-boundary (bottom-, left-, or right-boundary) points is called a \emph{top} (\emph{bottom, left, or right}) \emph{return}.

\begin{figure}[ht] 
\centering
\includegraphics[scale=1]{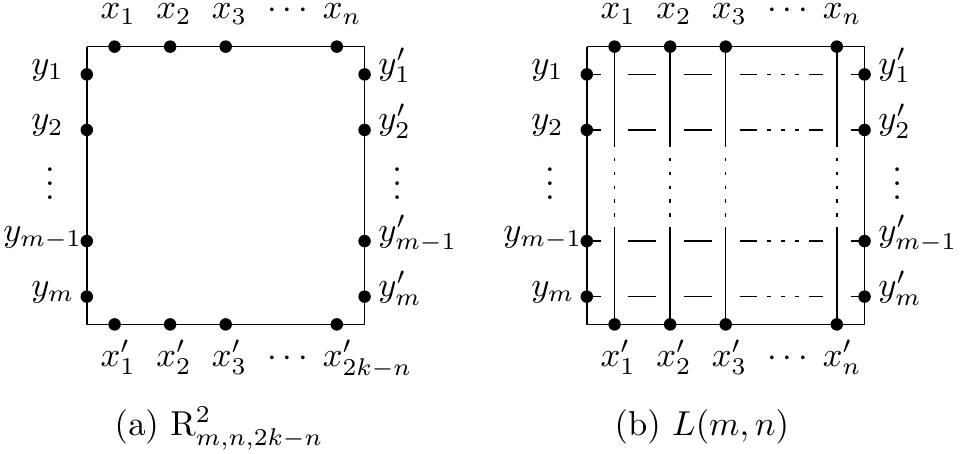}
\caption{Rectangle $\mathrm{R}^{2}_{m,n,2k-n}$ and lattice crossing $L(m,n)$}
\label{fig:LC_intro}
\end{figure}

\begin{definition}[\cite{DW2022}, Definition~2.1]
\label{def:states_intro}
A \emph{roof state} (\emph{floor state}) is a crossingless connection with no bottom (top) returns. A roof state with no top returns is called a \emph{middle state}. A roof state $R$ with $\mathrm{ht}(R) = 0$ is called a \emph{top state}, likewise, a floor state $F$ with $\mathrm{ht}(F) = 0$ is called a \emph{bottom state}. 
\end{definition}

A crossingless connection $C$ with $n_{t}(C) = n_{b}(C)$ is called a \emph{Catalan state} and we denote by $\mathrm{Cat}(m,n)$ the set of all such states in $\mathrm{R}^{2}_{m,n,n}$. A relative framed link $L(m,n)$ in cylinder $\mathrm{R}^{2}_{m,n,n} \times I$ with $2(m+n)$ points fixed on its boundary whose projection is shown in Figure~\ref{fig:LC_intro}(b) is called \emph{\textup{(}A-type\textup{)} lattice crossing}. As shown in \cite{Prz1999}, $\mathrm{Cat}(m,n)$ is a basis of the Relative Kauffman Bracket Skein Module of $\mathrm{R}^{2}_{m,n,n} \times I$, so
\begin{equation*}
L(m,n) = \sum_{C \in \mathrm{Cat}(m,n)} C(A) \, C,
\end{equation*}
where each $C(A) \in \mathbb{Z}[A^{\pm 1}]$ is called the \emph{coefficient} of Catalan state $C$.

A Kauffman state $s$ of $L(m,n)$ is an assignment of positive and negative markers to crossings of $L(m,n)$. For such $s$, let $D_{s}$ be the diagram obtained from $L(m,n)$ after applying rules shown in Figure~\ref{fig:markers_lhlv}(a) to smooth its crossings. Denote by $C_{s}$ the Catalan state resulting from $D_{s}$ after removing its closed components. We say that $C \in \mathrm{Cat}(m,n)$ is \emph{realizable} if $C = C_{s}$ for some Kauffman state $s$ of $L(m,n)$. 

\begin{figure}[ht] 
\centering
\includegraphics[scale=1]{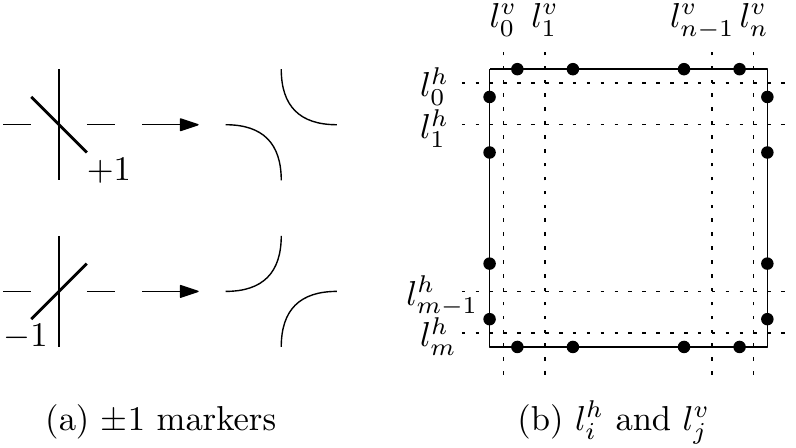}
\caption{$\pm 1$ markers, horizontal lines $l^{h}_{i}$ and vertical lines $l^{v}_{j}$}
\label{fig:markers_lhlv}
\end{figure}

Let $l^{h}_{i}$ and $l^{v}_{j}$ be horizontal and vertical lines in Figure~\ref{fig:markers_lhlv}(b). Denote by $\#(C \cap l)$ the geometric intersection number of a crossingless connection $C$ and a line $l$. Realizable Catalan states are characterized by the following results.

\begin{theorem}[\protect\cite{DLP2015}, Theorem~2.5] 
\label{thm:vh_line_condi}
A Catalan state $C \in \mathrm{Cat}(m,n)$ is realizable if and only if $\#(C \cap l^{h}_{i}) \leq n$ for $i = 1,2,\ldots,m-1$ and $\#(C \cap l^{v}_{j}) \leq m$ for $j = 1,2,\ldots,n-1$.
\end{theorem}

\begin{theorem}[\protect\cite{DW2022}, Theorem~3.13] 
\label{thm:coef_nonzero}
A Catalan state $C$ is realizable if and only if $C(A) \neq 0$.
\end{theorem}

For $C \in \mathrm{Cat}(m,n)$ with no bottom returns, let $(T(C),v_{0},\alpha)$ be the plane rooted tree with root $v_{0}$ and delay $\alpha$ defined from the set of leaves of $T(C)$ (not including $v_{0}$) to $\mathbb{N}$ by $\alpha(v) = k$ if $v$ corresponds to a left or a right return with its lower end $y_{k}$ or $y'_{k}$ and $\alpha(v) = 1$ otherwise (see Figure~\ref{fig:rooted_tree} and \cite{DP2019} for more details).

\begin{figure}[ht] 
\centering
\includegraphics[scale=1]{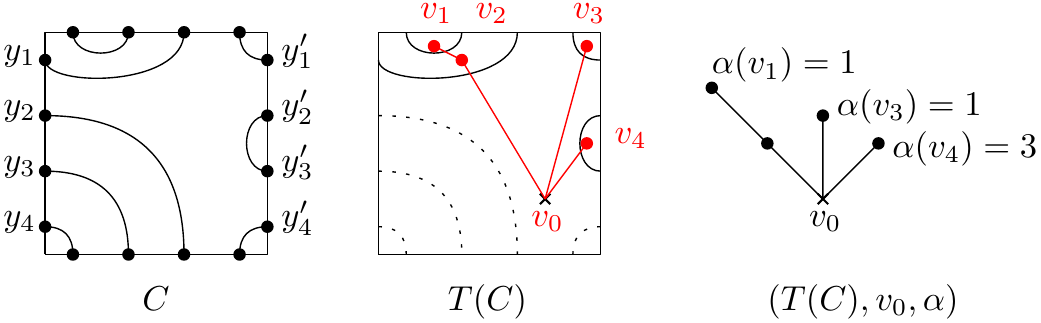}
\caption{Plane rooted tree $(T(C),v_{0})$ with delay $\alpha$}
\label{fig:rooted_tree}
\end{figure}

\begin{definition}[\cite{DP2019}, Definition~3.1] 
\label{def:q_poly}
Let $(T,v_{0},\alpha)$ be a plane rooted tree $T$ with root $v_{0}$ and weight function $\alpha$ defined on leaves of $T$ (not including $v_{0}$) to $\mathbb{N}$. Denote by $L_{1}(T)$ the set of all leaves $v$ of $T$ with $\alpha(v) = 1$. The \emph{plucking polynomial} $Q_{q}$ of $(T,v_{0},\alpha)$ is a polynomial in variable $q$ which we set $Q_{q}(T,v_{0},\alpha) = 1$ for $T$ with no edges, and otherwise,
\begin{equation*}
Q_{q}(T,v_{0},\alpha) = \sum_{v \in L_{1}(T)} q^{r(T,v_{0},v)} Q_{q}(T-v,v_{0},\alpha_{v}),
\end{equation*}
where $r(T,v_{0},v)$ is the number of vertices of $T$ to the right of the unique path from $v_{0}$ to $v$, and $\alpha_{v}$ is defined by $\alpha_{v}(u) = \max\{1,\alpha(u)-1\}$ if $u$ is a leaf of $T$ and $\alpha_{v}(u) = 1$ if $u$ is a new leaf of $T-v$.
\end{definition}

For a realizable Catalan state $C \in \mathrm{Cat}(m,n)$ with no bottom returns, $C(A)$ can be found using Kauffman states of $L(m,n)$ with markers
\begin{equation*}
(\underset{b_{j}}{\underbrace{1,1,\ldots,1}},\underset{n-b_{j}}{\underbrace{-1,-1,\ldots,-1}})
\end{equation*}
in each $j$-th row ($1 \leq j \leq m$), i.e., using sequences $\mathbf{b} = (b_{1},b_{2},\ldots,b_{m})$, where $0 \leq b_{j} \leq n$ (see also \cite{DP2019}). We say that $\mathbf{b}$ \emph{realizes} $C$ if the Kauffman state corresponding to $\mathbf{b}$ realizes $C$. Let $\mathfrak{b}(C)$ be the set of $\mathbf{b}$'s that realize $C$ and let $\Vert\mathbf{b}\Vert = b_{1} + b_{2} + \cdots + b_{m}$. Define
\begin{equation*}
\beta(C) = \max\{\Vert\mathbf{b}\Vert: \mathbf{b} \in \mathfrak{b}(C)\}
\end{equation*}
and note that there is a unique $\mathbf{b} \in \mathfrak{b}(C)$, called the \emph{maximal sequence} of $C$, such that $\beta(C) = \Vert\mathbf{b}\Vert$.

\begin{theorem}[\cite{DP2019}, Theorem~3.4] 
\label{thm:coef_no_bot_rtn}
Let $C \in \mathrm{Cat}(m,n)$ be a realizable Catalan state with no bottom returns, and let $T = T(C)$ be its corresponding plane rooted tree with root $v_{0}$ and delay $\alpha$. Then 
\begin{equation*}
C(A) = A^{2\beta(C)-mn} \, Q_{A^{-4}}^{*}(T,v_{0},\alpha),
\end{equation*}
where $Q_{q}^{*}(T,v_{0},\alpha) = q^{-\min\deg_{q} Q_{q}(T,v_{0},\alpha)} \, Q_{q}(T,v_{0},\alpha)$.
\end{theorem}

Let $\varphi_{n}$ be a map from the set $\mathcal{B}_{n}$ of all bottom states $F$ with $n_{b}(F) = n$ to the set $\mathrm{Fin}(\mathbb{N})$ of all finite subsets of $\mathbb{N}$ defined by 
\begin{equation*}
\varphi_{n}(F) = \{i_{1},i_{2},\ldots,i_{k}\},
\end{equation*}
where $i_{1} < i_{2} < \cdots < i_{k}$ are indices of the left ends of bottom returns of $F$ (see \cite{DW2022} for more details). Denote by $\mathcal{L}_{n}$ the image $\varphi_{n}(\mathcal{B}_{n})$ of $\varphi_{n}$ and note that, since $\varphi_{n}$ is injective, for each $I \in \mathcal{L}_{n}$ there is a unique $\varphi_{n}^{-1}(I) \in \mathcal{B}_{n}$ (see example in Figure~\ref{fig:phi_inverse} for $I = \{i\}$).

\begin{figure}[ht] 
\centering
\includegraphics[scale=1]{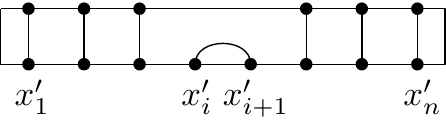}
\caption{$\varphi_{n}^{-1}(\{i\})$ for $1 \leq i \leq n-1$}
\label{fig:phi_inverse}
\end{figure}

For a pair of crossingless connections $C_{1}$ and $C_{2}$, their \emph{vertical product} $C_{1} * C_{2}$ is a crossingless connection in Figure~\ref{fig:I_oplus_J}(a) if $n_{b}(C_{1}) = n_{t}(C_{2})$ and no closed components are created. Otherwise, $C_{1} * C_{2}$ is not a crossingless connection and we set $C_{1} * C_{2} = K_{0}$ and, in addition, for $K_{0}$ we put
\begin{equation*}
K_{0} * C_{2} = C_{1} * K_{0} = K_{0} * K_{0} = K_{0}.
\end{equation*}
If $I \neq \emptyset$ and $I \in \mathrm{Fin}(\mathbb{N})$, it is clear that $I \in \mathcal{L}_{\max I+|I|}$. Hence, there is a non-negative integer $n_{I}$ such that, $I \in \mathcal{L}_{n}$ for $n \geq n_{I}$ and $I \notin \mathcal{L}_{n}$ for $n < n_{I}$. As in \cite{DW2022}, we define operation $\oplus$ on $\mathrm{Fin}(\mathbb{N})$ as follows. For $I,J \in \mathrm{Fin}(\mathbb{N})$,
\begin{equation*}
I \oplus J = \varphi_{n^{*}}(\varphi_{n^{*}-2|J|}^{-1}(I) * \varphi_{n^{*}}^{-1}(J)),
\end{equation*}
where $n^{*} = \max\{n_{I}+2|J|,n_{J}\}$. The geometric interpretation of $I \oplus J$ is shown in Figure~\ref{fig:I_oplus_J}(b).

\begin{figure}[ht] 
\centering
\includegraphics[scale=1]{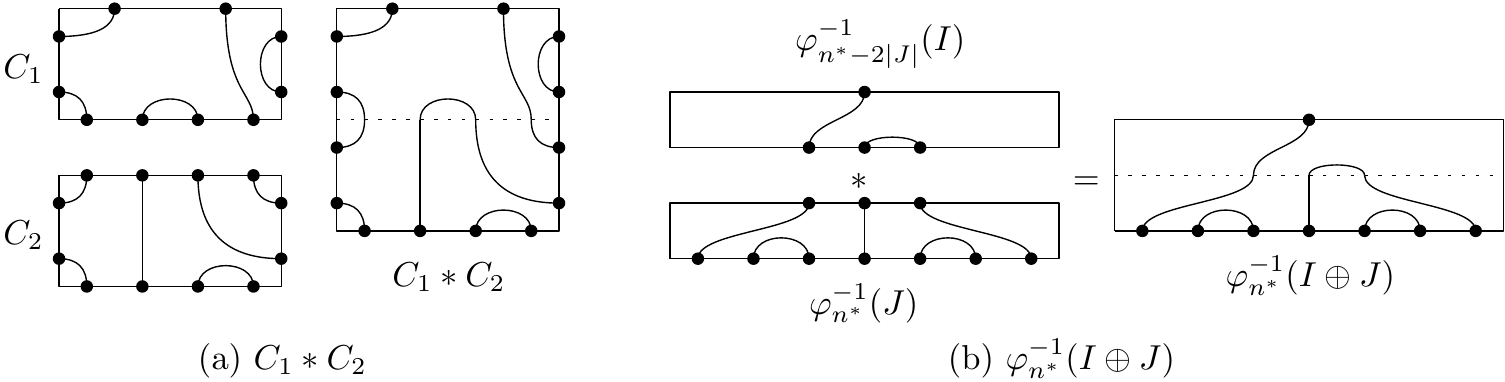}
\caption{Vertical product of $C_{1}$ and $C_{2}$, geometric interpretation of $I \oplus J$}
\label{fig:I_oplus_J}
\end{figure}

Given a crossingless connection $C$ with $n_{b}(C) = n$ or $C = K_{0}$ and $I \in \mathrm{Fin}(\mathbb{N})$, define $C_{I} = C'$ if $C$ is a crossingless connection, $I \in \mathcal{L}_{n}$, and there is a crossingless connection $C'$ such that
\begin{equation*}
C = C' * \varphi_{n}^{-1}(I),
\end{equation*}
and in all other cases we set $C_{I} = K_{0}$ (see Figure~\ref{fig:C_I}). 

\begin{figure}[ht] 
\centering
\includegraphics[scale=1]{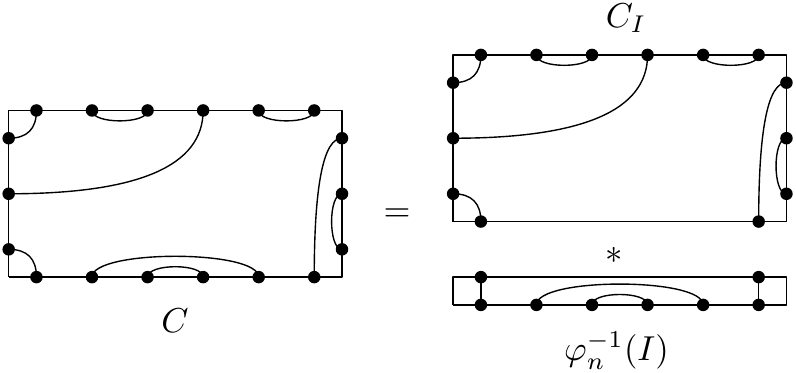}
\caption{$C_{I}$}
\label{fig:C_I}
\end{figure}

Let $\mathcal{W}$ be the set consisting of all pairs $(R,I)$, where $R$ is a roof state and $I \in \mathrm{Fin}(\mathbb{N})$. Given $(R,I) \in \mathcal{W}$, as in \cite{DW2022} we let $\Theta_{A}(R,I;\cdot)$ to be a function of variable $F$ defined by
\begin{equation*}
\Theta_{A}(R,I;F) = \begin{cases}
C(A), & \text{if}\ C = R * F_{I}\ \text{is a Catalan state}, \\
0,    & \text{otherwise},
\end{cases}
\end{equation*}
where $F$ is a floor state or $F = K_{0}$. In this paper, we use an equivalent version of $\Theta_{A}$-state expansion defined as below.

\begin{definition}[\protect\cite{DW2022}, Definition~4.1] 
\label{def:Theta_state_expansion}
A \emph{$\Theta_{A}$-state expansion} for $(R,I) \in \mathcal{W}$ is a relation given by
\begin{equation*}
\Theta_{A}(R,I;\cdot) = \sum_{(R',I') \in \mathcal{P}'} Q_{R',I'}(A) \, \Theta_{A}(R',I' \oplus I;\cdot),
\end{equation*}
where $\mathcal{P}'$ is a finite subset of $\mathcal{W}$ such that each $(R',I') \in \mathcal{P}'$ satisfies conditions
\begin{enumerate}
\item[i)] $I' \in \mathcal{L}_{n_{t}(R)}$,
\item[ii)] $R'$ is a middle state with $n_{t}(R') = n_{t}(R)-2|I'|$ and $n_{b}(R') = n_{b}(R)$,
\end{enumerate}
and $0 \neq Q_{R',I'}(A) \in \mathbb{Q}(A)$ for every $(R',I') \in \mathcal{P}'$.
\end{definition}

Since each pair $(R',I') \in \mathcal{P}'$ consists of a middle state $R'$ and $I' \in \mathrm{Fin}(\mathbb{N})$, for every floor state $F$ either $R'*F_{I' \oplus I}$ is a Catalan state with no top returns or it is not a Catalan state. Consequently, if a $\Theta_{A}$-state expansion for $(R,I)$ exists, $\Theta_{A}(R,I;F)$ can be expressed as a sum of $Q_{R',I'}(A) \, \Theta_{A}(R',I' \oplus I;F)$, where either $\Theta_{A}(R',I' \oplus I;F) \neq 0$ can be found using Theorem~\ref{thm:coef_no_bot_rtn} or $\Theta_{A}(R',I' \oplus I;F) = 0$.

\begin{theorem}[\protect\cite{DW2022}, Theorem~4.7] 
\label{thm:main}
Every $(R,I) \in \mathcal{W}$ has a $\Theta_{A}$-state expansion.
\end{theorem}

Since any Catalan state $C$ can be written as a vertical product of a roof state $R$ and a floor state $F$, i.e., $C = R*F$. Therefore, $C(A) = \Theta_{A}(R,\emptyset;F)$ can be found using $\Theta_{A}$-state expansion for $(R,\emptyset)$. In particular, $C(A)$ is a linear combination over $\mathbb{Q}(A)$ of coefficients of Catalan states with no top returns.

\section{Factorization of Plucking Polynomial}
\label{s:factor_plucking_poly}

In this section, we prove that plucking polynomial of a plane rooted tree $(T,v_{0},\alpha)$ factors (see Theorem~\ref{thm:prod_formula_plucking}) into a product of plucking polynomials of a ``splitting subtree'' (see Definition~\ref{def:splitting_subtree}) and its ``complementary tree'' (see Definition~\ref{def:complementary_tree}). A special version of this property for plucking polynomial was known earlier by J.H.~Przytycki (see Remark~\ref{rem:prod_formula_plucking2}).
%\footnote{As we learned from J.H.~Przytycki, a version of Theorem~3.5 (stated without a proof) is given in \cite{Prz2016-3} on page 129.}

Let $V(T)$ be the set of vertices of a plane rooted tree $(T,v_{0})$ and $v \in V(T)$. Denote by $T_{v}$ the \emph{subtree of $T$ above $v$}, i.e., $T_{v}$ is the plane rooted subtree of $T$ induced by all vertices $u \in V(T)$ which can be connected to $v_{0}$ via a simple path that includes $v$ as its vertex. The \emph{ordered rooted sum} of plane rooted trees $(T_{1},v_{1})$ and $(T_{2},v_{2})$ is the plane rooted tree $(T_{1} \vee T_{2},v_{0})$ shown in Figure~\ref{fig:ordered_rooted_sum}.

\begin{figure}[ht]
\centering
\includegraphics[scale=1]{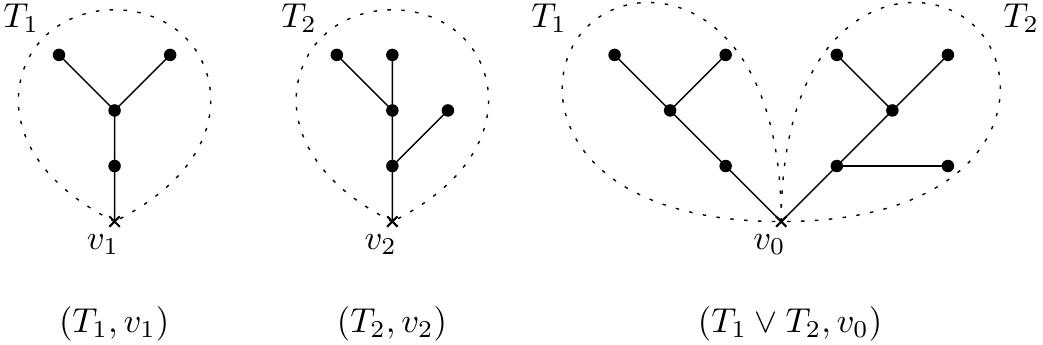}
\caption{Ordered rooted sum $(T_{1} \vee T_{2},v_{0})$ of $(T_{1},v_{1})$ and $(T_{2},v_{2})$}
\label{fig:ordered_rooted_sum}
\end{figure}

\begin{definition}
\label{def:splitting_subtree}
Let $(T,v_{0},\alpha)$ be a plane rooted tree with a weight function $\alpha$ from the set of its leaves $L(T)$ (not including root) to $\mathbb{N}$. A plane rooted tree $(T',v',\alpha')$ is called a \emph{splitting subtree} of $(T,v_{0},\alpha)$ if
\begin{enumerate}
\item[i)] $v' \in V(T)$, $T_{v'} = T_{l} \vee T' \vee T_{r}$ for some trees $T_{l}$ and $T_{r}$, $\alpha'$ is the restriction of $\alpha$ to $L(T')$, and
\item[ii)] $\alpha'(v) \leq \alpha(u)$ for every $v \in L(T')$ and $u \in L(T) \setminus L(T')$.
\end{enumerate}
\end{definition}

\begin{definition}
\label{def:complementary_tree}
Let $(T',v',\alpha')$ be a splitting subtree of $(T,v_{0},\alpha)$ and let $P_{k}$ be a plane rooted path of length $k = |V(T')|-1$ with one of its leaves fixed as the root. A plane rooted tree $(T'',v_{0},\alpha'')$ is called a \emph{complementary tree} for $(T',v',\alpha')$ if
\begin{enumerate}
\item[i)] $T''$ is obtained from $T$ by replacing $T_{v'} = T_{l} \vee T' \vee T_{r}$ by the tree $T_{l} \vee P_{k} \vee T_{r}$ (see Figure~\ref{fig:prod_formula_plucking}), and
\item[ii)] $\alpha''$ is defined by $\alpha''(v) = \alpha(v)$ if $v \in L(T'') \cap L(T)$ and $\alpha''(v) = 1$ if $v \in L(T'') \setminus L(T)$.
\end{enumerate}
\end{definition}

\begin{figure}[ht]
\centering
\includegraphics[scale=1]{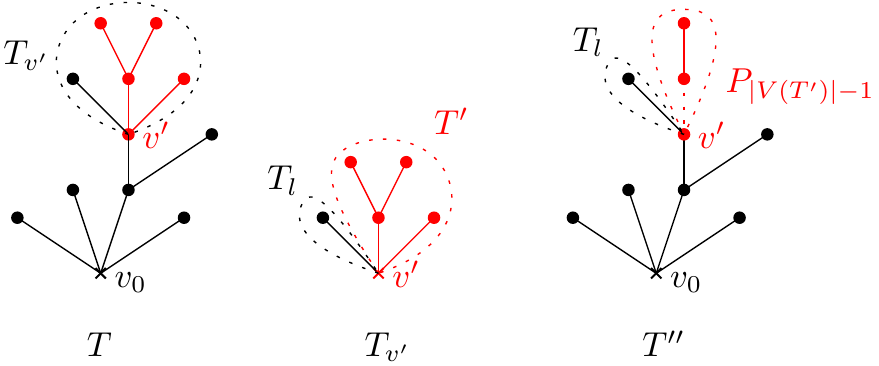}
\caption{Trees $T$, $T_{v'}$, and $T''$}
\label{fig:prod_formula_plucking}
\end{figure}

\begin{lemma}
\label{lem:prod_formula_plucking_1}
Let $(T',v',\alpha')$ be a splitting subtree of $(T,v_{0},\alpha)$ and let $(T'',v_{0},\alpha'')$ be the corresponding complementary tree. Assume that $L_{1}(T') \neq \emptyset$ and denote by $u$ the unique leaf in $L(T'') \setminus L(T)$. Then, for every $v \in L_{1}(T')$
\begin{enumerate}
\item[i)] $(T'-v,v',\alpha'_{v})$ is a splitting subtree of $(T-v,v_{0},\alpha_{v})$ and
\item[ii)] $(T''-u,v_{0},\alpha''_{u})$ is the complementary tree for $(T'-v,v',\alpha'_{v})$,
\end{enumerate}
where $\alpha_{v}$, $\alpha'_{v} = (\alpha')_{v}$, and $\alpha''_{u} = (\alpha'')_{u}$ are as in Definition~\ref{def:q_poly}.
\end{lemma}

\begin{proof}
For i), we notice that, for every $v \in L_{1}(T')$ either both $T-v$ and $T'-v$ have the same new leaf $u'$ or they both do not have such a leaf. As the latter case is analogous, we only prove the former case. Since $T_{v'} = T_{l} \vee T' \vee T_{r}$ for some trees $T_{l}$ and $T_{r}$, it follows that
\begin{equation*}
(T-v)_{v'} = T_{l} \vee (T'-v) \vee T_{r}.    
\end{equation*}
Furthermore, for all $w \in L(T'-v) \setminus \{u'\}$
\begin{equation*}
\alpha'_{v}(w) = \max\{\alpha'(w)-1,1\} = \max\{\alpha(w)-1,1\} = \alpha_{v}(w),
\end{equation*}
and $\alpha'_{v}(u') = 1 = \alpha_{v}(u')$, so $\alpha'_{v}$ is a restriction of $\alpha_{v}$ to $L(T'-v)$. Finally,
\begin{equation*}
\alpha'_{v}(u') = 1 \leq \alpha'_{v}(w) = \max\{\alpha'(w)-1,1\} \leq \max\{\alpha(w')-1,1\} = \alpha_{v}(w')
\end{equation*}
for all $w \in L(T'-v) \setminus \{u'\}$ and $w' \in L(T-v) \setminus L(T'-v)$. Therefore, i) follows.

For ii), let $k = |V(T')|-1$ and, since $T''_{v'} = T_{l} \vee P_{k} \vee T_{r}$, one notices that
\begin{equation*}
(T''-u)_{v'} =  T_{l} \vee P_{k-1} \vee T_{r} \quad \text{and} \quad  L(T''-u) \cap L(T-v) = L(T'') \setminus \{u\}.  
\end{equation*}
Moreover, by the definition of $\alpha''_{u}$, 
\begin{equation*}
\alpha''_{u}(w) = \max\{\alpha''(w)-1,1\} = \max\{\alpha(w)-1,1\} = \alpha_{v}(w)
\end{equation*} 
if $w \in L(T'') \setminus \{u\}$ and $\alpha''_{u}(w) = 1$ if $w \in L(T''-u) \setminus L(T-v)$. Consequently, ii) follows.
\end{proof}

\begin{lemma}
\label{lem:prod_formula_plucking_2}
Let $(T',v',\alpha')$ be a splitting subtree of $(T,v_{0},\alpha)$ and let $(T'',v_{0},\alpha'')$ be the corresponding complementary tree. Then for every $v \in L_{1}(T) \setminus L_{1}(T')$,
\begin{enumerate}
\item[i)] $(T',v',\alpha')$ is a splitting subtree of $(T-v,v_{0},\alpha_{v})$ and
\item[ii)] $(T''-v,v_{0},\alpha''_{v})$ is the complementary tree for $(T',v',\alpha')$,
\end{enumerate}
where $\alpha_{v}$ and $\alpha''_{v} = (\alpha'')_{v}$ are as in Definition~\ref{def:q_poly}.
\end{lemma}

\begin{proof}
Notice that if $L_{1}(T) \setminus L_{1}(T') \neq \emptyset$ then $\alpha(w) = \alpha'(w) = 1$ for all $w \in L(T')$ by condition ii) of Definition~\ref{def:splitting_subtree}. Therefore, $\alpha' \equiv (\alpha')_{v}$ is a restriction of $\alpha_{v}$ to $L(T')$ and, as one may see, i) follows. Using arguments analogous to those in our proof of Lemma~\ref{lem:prod_formula_plucking_1}, one shows ii).
\end{proof}

\begin{theorem}
\label{thm:prod_formula_plucking}
If $(T',v',\alpha')$ is a splitting subtree of $(T,v_{0},\alpha)$ then
\begin{equation}
Q_{q}(T,v_{0},\alpha) = Q_{q}(T',v',\alpha') \cdot Q_{q}(T'',v_{0},\alpha''),
\label{eqn:prod_formula_plucking}
\end{equation}
where $(T'',v_{0},\alpha'')$ is the complementary tree for $(T',v',\alpha')$.
\end{theorem}

\begin{proof}
We prove \eqref{eqn:prod_formula_plucking} using induction on $n = |V(T)|$. For $n = 1$, \eqref{eqn:prod_formula_plucking} holds since
\begin{equation*}
Q_{q}(T,v_{0},\alpha) = Q_{q}(T',v',\alpha') = Q_{q}(T'',v_{0},\alpha'') = 1.  
\end{equation*}
Assume that \eqref{eqn:prod_formula_plucking} is true for trees with $n-1$ vertices. Let $T$ be a tree with $n$ vertices. If $L_{1}(T') = \emptyset$ then, as one may check, \eqref{eqn:prod_formula_plucking} clearly holds. If $L_{1}(T') \neq \emptyset$, let $u$ be the unique leaf in $L(T'') \setminus L(T)$. Then, by Lemma~\ref{lem:prod_formula_plucking_1} and the induction hypothesis, 
\begin{equation}
\label{eqn:pf_prop_product_qpoly_1}
Q_{q}(T-v,v_{0},\alpha_{v}) = Q_{q}(T'-v,v',\alpha'_{v}) \cdot Q_{q}(T''-u,v_{0},\alpha''_{u})
\end{equation}
for every $v \in L_{1}(T')$. Analogously, if $L_{1}(T) \setminus L_{1}(T') \neq \emptyset$ then by Lemma~\ref{lem:prod_formula_plucking_2} and the induction hypothesis,
\begin{equation}
\label{eqn:pf_prop_product_qpoly_2}
Q_{q}(T-v,v_{0},\alpha_{v}) = Q_{q}(T',v',\alpha') \cdot Q_{q}(T''-v,v_{0},\alpha''_{v})
\end{equation}
for every $v \in L_{1}(T) \setminus L_{1}(T')$.

We show that
\begin{equation}
\label{eqn:pf_prop_product_qpoly_3}
\sum_{v \in L_{1}(T')} q^{r(T,v_{0},v)} \, Q_{q}(T-v,v_{0},\alpha_{v}) = q^{r_{*}} \, Q_{q}(T',v',\alpha') \cdot Q_{q}(T''-u,v_{0},\alpha''_{u}), 
\end{equation}
where $r_{*} = \min\{r(T,v_{0},v) \mid v \in L(T')\}$ and $r(T,v_{0},v)$ is as in Definition~\ref{def:q_poly}. Let $v_{*}$ be a leaf of $T'$ with $r(T,v_{0},v_{*}) = r_{*}$ and suppose that there is $w \in V(T')$ to the right of the path $P_{*}$ joining $v_{*}$ and $v_{0}$ in $T$. Then, each leaf $v$ of $T'$ for which $w$ is on the path joining $v_{0}$ and $v$ must be to the right of $P_{*}$. For such leaves $r(T,v_{0},v) < r(T,v_{0},v_{*}) = r_{*}$ which contradicts the definition of $r_{*}$, so no $w \in V(T')$ is to the right of $P_{*}$. Let $P$ be a path joining $v \in L(T')$ and $v_{0}$. Notice that both $P_{*}$ and $P$ include $v'$ as a vertex, thus, $P_{*}$ and $P$ have a common subpath between $v'$ and $v_{0}$. Since there are no vertices of $T'$ to the right of $P_{*}$, it follows that
\begin{equation*}
r(T,v_{0},v) = r(T,v_{0},v_{*}) + r(T',v',v) = r_{*} + r(T',v',v)
\end{equation*}
and consequently
\begin{equation}
\label{eqn:pf_prop_product_qpoly_4}
\sum_{v \in L_{1}(T')} q^{r(T,v_{0},v)} \, Q_{q}(T'-v,v',\alpha'_{v}) = \sum_{v \in L_{1}(T')} q^{r_{*}+r(T',v',v)} \, Q_{q}(T'-v,v',\alpha'_{v})
= q^{r_{*}} \, Q_{q}(T',v',\alpha').
\end{equation}
Therefore, \eqref{eqn:pf_prop_product_qpoly_3} follows from \eqref{eqn:pf_prop_product_qpoly_1} and \eqref{eqn:pf_prop_product_qpoly_4}.

Notice that $r(T,v_{0},v) = r(T'',v_{0},v)$ for any $v \in L_{1}(T) \setminus L_{1}(T')$, so by \eqref{eqn:pf_prop_product_qpoly_2},
\begin{equation}
\label{eqn:pf_prop_product_qpoly_5}
\sum_{v \in L_{1}(T) \setminus L_{1}(T')} q^{r(T,v_{0},v)} \, Q_{q}(T-v,v_{0},\alpha_{v}) = \sum_{v \in L_{1}(T) \setminus L_{1}(T')} q^{r(T'',v_{0},v)} \, Q_{q}(T',v',\alpha') \cdot Q_{q}(T''-v,v_{0},\alpha''_{v}).
\end{equation}
Moreover, by Definition~\ref{def:q_poly},
\begin{equation}
\label{eqn:pf_prop_product_qpoly_6}
Q_{q}(T,v_{0},\alpha) = \sum_{v \in L_{1}(T')} q^{r(T,v_{0},v)} \, Q_{q}(T-v,v_{0},\alpha_{v}) + \sum_{v \in L_{1}(T) \setminus L_{1}(T')} q^{r(T,v_{0},v)} \, Q_{q}(T-v,v_{0},\alpha_{v})
\end{equation}
and
\begin{equation}
\label{eqn:pf_prop_product_qpoly_7}
Q_{q}(T'',v_{0},\alpha'') = q^{r(T'',v_{0},u)} \, Q_{q}(T''-u,v_{0},\alpha''_{u}) + \sum_{v \in L_{1}(T'') \setminus \{u\}} q^{r(T'',v_{0},v)} \, Q_{q}(T''-v,v_{0},\alpha''_{v}).
\end{equation}
Since $r(T'',v_{0},u) = r_{*}$ and $L_{1}(T'') \setminus \{u\} = L_{1}(T) \setminus L_{1}(T')$, it follows from \eqref{eqn:pf_prop_product_qpoly_6}, \eqref{eqn:pf_prop_product_qpoly_3}, \eqref{eqn:pf_prop_product_qpoly_5}, 
and \eqref{eqn:pf_prop_product_qpoly_7} that
\begin{eqnarray*}
& & Q_{q}(T,v_{0},\alpha)
= \sum_{v \in L_{1}(T')} q^{r(T,v_{0},v)} \, Q_{q}(T-v,v_{0},\alpha_{v}) + \sum_{v \in L_{1}(T) \setminus L_{1}(T')} q^{r(T,v_{0},v)} \, Q_{q}(T-v,v_{0},\alpha_{v}) \\
& & = \quad q^{r_{*}} \, Q_{q}(T',v',\alpha') \cdot Q_{q}(T''-u,v_{0},\alpha''_{u})
+ \sum_{v \in L_{1}(T) \setminus L_{1}(T')} q^{r(T'',v_{0},v)} \, Q_{q}(T',v',\alpha') \cdot Q_{q}(T''-v,v_{0},\alpha''_{v}) \\
& & = \quad Q_{q}(T',v',\alpha') \cdot Q_{q}(T'',v_{0},\alpha'').
\end{eqnarray*}
\end{proof}

\begin{remark}
\label{rem:prod_formula_plucking2}
Note that Theorem~\ref{thm:prod_formula_plucking} implies also the formula obtained by J.H. Przytycki in \cite{Prz2016-3} on page 129 for the ordered rooted sum of trees 
\begin{equation*}
T = T_{k} \vee T_{k-1} \vee \cdots \vee T_{1}
\end{equation*}
with root $v_{0}$ and weight function
\begin{equation*}
\alpha = \alpha_{k} \vee \alpha_{k-1} \vee \cdots \vee \alpha_{1},    
\end{equation*}
where $\alpha|_{T_{j}} = \alpha_{j} \equiv s_{j}$ satisfies conditions $s_{1} = 1$ and
\begin{equation*}
s_{j-1} \leq s_{j} \leq \sum_{i=1}^{j-1} n_{i} + 1
\end{equation*}
for $j = 2,3,\ldots,k$ and $n_{i} = |V(T_{i})|-1$. Indeed, let 
\begin{equation*}
T' = T_{k-1} \vee T_{k-2} \vee \ldots \vee T_{1} \quad \text{and} \quad \alpha' = \alpha_{k-1} \vee \alpha_{k-2} \vee \cdots \vee \alpha_{1},    
\end{equation*}
then $(T',v_{0},\alpha')$ is a splitting subtree of $(T,v_{0},\alpha)$ with complementary tree $(T_{k} \vee P_{n'_{k}},v_{0},\alpha'_{k})$, where $n'_{k} = n_{1}+n_{2}+\cdots+n_{k-1}$ and $\alpha'_{k}(v) = s_{k}$ if $v \in L(T_{k})$ and $\alpha'_{k}(v) = 1$ if $v$ is a new leaf. By Theorem~\ref{thm:prod_formula_plucking},
\begin{equation*}
Q_{q}(T,v_{0},\alpha) = Q_{q}(T',v_{0},\alpha') \cdot Q_{q}(T_{k} \vee P_{n'_{k}},v_{0},\alpha'_{k}).
\end{equation*}
For convenience, we simply write $(T_{0},v_{0}) = (T_{0},v_{0},\alpha_{0})$ for any tree $T_{0}$ when $\alpha_{0} \equiv 1$. By Definition~\ref{def:q_poly},
\begin{equation*}
Q_{q}(T_{k} \vee P_{n'_{k}},v_{0},\alpha'_{k}) = Q_{q}(T_{k} \vee P_{n'_{k}-s_{k}+1},v_{0}),
\end{equation*}
and since $(T_{k},v_{0})$ is a splitting subtree of $(T_{k} \vee P_{n'_{k}-s_{k}+1},v_{0})$ with  complementary tree $(P_{n_{k}} \vee P_{n'_{k}-s_{k}+1},v_{0})$ for a constant weight function $\alpha''_{k} \equiv 1$, by Theorem~\ref{thm:prod_formula_plucking},
\begin{eqnarray*}
Q_{q}(T_{k} \vee P_{n'_{k}-s_{k}+1},v_{0}) &=& Q_{q}(T_{k},v_{0}) \cdot Q_{q}(P_{n_{k}} \vee P_{n'_{k}-s_{k}+1},v_{0}) \\
&=& Q_{q}(T_{k},v_{0}) \cdot \binom{n_{1}+n_{2}+\cdots+n_{k}-s_{k}+1}{n_{k}}_{q}.
\end{eqnarray*}
Therefore, we showed that
\begin{equation*}
Q_{q}(T,v_{0},\alpha) = Q_{q}(T',v_{0},\alpha') \cdot Q_{q}(T_{k},v_{0}) \cdot \binom{n_{1}+n_{2}+\cdots+n_{k}-s_{k}+1}{n_{k}}_{q}.
\end{equation*}
Applying recursively equation above, we see that
\begin{equation*}
Q_{q}(T,v_{0},\alpha) = \prod_{i=1}^{k} Q_{q}(T_{i},v_{0}) \cdot \prod_{i=2}^{k} \binom{n_{1}+n_{2}+\cdots+n_{i}-s_{i}+1}{n_{i}}_{q}.
\end{equation*}
We also would like to point it out that this result generalizes Theorem~2.2 in \cite{Prz2016} by taking $k = 2$ and $s_{1} = s_{2} = 1$.
\end{remark}

\begin{example}
\label{ex:prod_formula_plucking}
Let $(T,v_{0},\alpha)$ be a plane rooted tree with the weight function shown in Figure~\ref{fig:ex_prod_formula_plucking}(a). Consider a (red) rooted subtree $T'$ of $T$ with root $v_{0}$ that includes all leaves $v \in L(T)$ with $\alpha(v) = 1$. Then $(T',v_{0},\alpha')$ is a splitting subtree of $(T,v_{0},\alpha)$, where $\alpha' = \alpha |_{L(T')}$. Clearly, $(T_{1},v_{0},\alpha_{1})$ shown in Figure~\ref{fig:ex_prod_formula_plucking}(b) is the corresponding complementary tree for $(T',v_{0},\alpha')$. Therefore, by Theorem~\ref{thm:prod_formula_plucking},
\begin{equation*}
Q_{q}(T,v_{0},\alpha) = Q_{q}(T',v_{0},\alpha') \cdot Q_{q}(T_{1},v_{0},\alpha_{1}).
\end{equation*}
Let $T'_{1}$ be a (red) rooted subtree of $T_{1}$ with root $v_{0}$ that includes all leaves $v \in L(T_{1})$ with $\alpha_{1}(v) \leq 2$. If we let $\alpha'_{1} = \alpha_{1} |_{L(T'_{1})}$, then $(T'_{1},v_{0},\alpha'_{1})$ is a splitting subtree of $(T_{1},v_{0},\alpha_{1})$ and $(T_{2},v_{0},\alpha_{2})$ shown in Figure~\ref{fig:ex_prod_formula_plucking}(c) is its complementary tree. Thus, by Theorem~\ref{thm:prod_formula_plucking}
\begin{equation*}
Q_{q}(T_{1},v_{0},\alpha_{1}) = Q_{q}(T'_{1},v_{0},\alpha'_{1}) \cdot Q_{q}(T_{2},v_{0},\alpha_{2}).
\end{equation*}
Analogously, for $i = 1,2,\ldots,k-1$, let $(T_{i},v_{0},\alpha_{i})$ be the plane rooted tree shown in Figure~\ref{fig:ex_prod_formula_plucking}(d) for $1 \leq i \leq k$ and let $T'_{i}$ be the (red) rooted subtree of $T_{i}$ with root $v_{0}$ that includes all leaves $v \in L(T_{i})$ with $\alpha_{i}(v) \leq 2i$. We put $\alpha'_{i} = \alpha_{i} |_{L(T'_{i})}$ and notice that $(T'_{i},v_{0},\alpha'_{i})$ is a splitting subtree of $(T_{i},v_{0},\alpha_{i})$ with its corresponding complementary tree $(T_{i+1},v_{0},\alpha_{i+1})$. Therefore, by Theorem~\ref{thm:prod_formula_plucking}
\begin{equation*}
Q_{q}(T_{i},v_{0},\alpha_{i}) = Q_{q}(T'_{i},v_{0},\alpha'_{i}) \cdot Q_{q}(T_{i+1},v_{0},\alpha_{i+1})
\end{equation*}
Since by Definition~\ref{def:q_poly}, $Q_{q}(T',v_{0},\alpha') = 1+q$,
\begin{equation*}
Q_{q}(T'_{i},v_{0},\alpha'_{i}) = q^{2i-1} (1+q)(1+q+q^{2})    
\end{equation*}
for $i=1,2,\ldots,k-1$, and
\begin{equation*}
Q_{q}(T_{k},v_{0},\alpha_{k}) = q^{2k-1} (1+q)(1+q+q^{2}),
\end{equation*}
it follows that
\begin{equation*}
Q_{q}(T,v_{0},\alpha) = Q_{q}(T',v_{0},\alpha') \cdot \bigg( \prod_{i=1}^{k-1} Q_{q}(T'_{i},v_{0},\alpha'_{i}) \bigg) \cdot Q_{q}(T_{k},v_{0},\alpha_{k}) = q^{k^{2}} (1+q)^{k+1} (1+q+q^{2})^{k}.
\end{equation*} 
\end{example}

\begin{figure}[ht]
\centering
\includegraphics[scale=1]{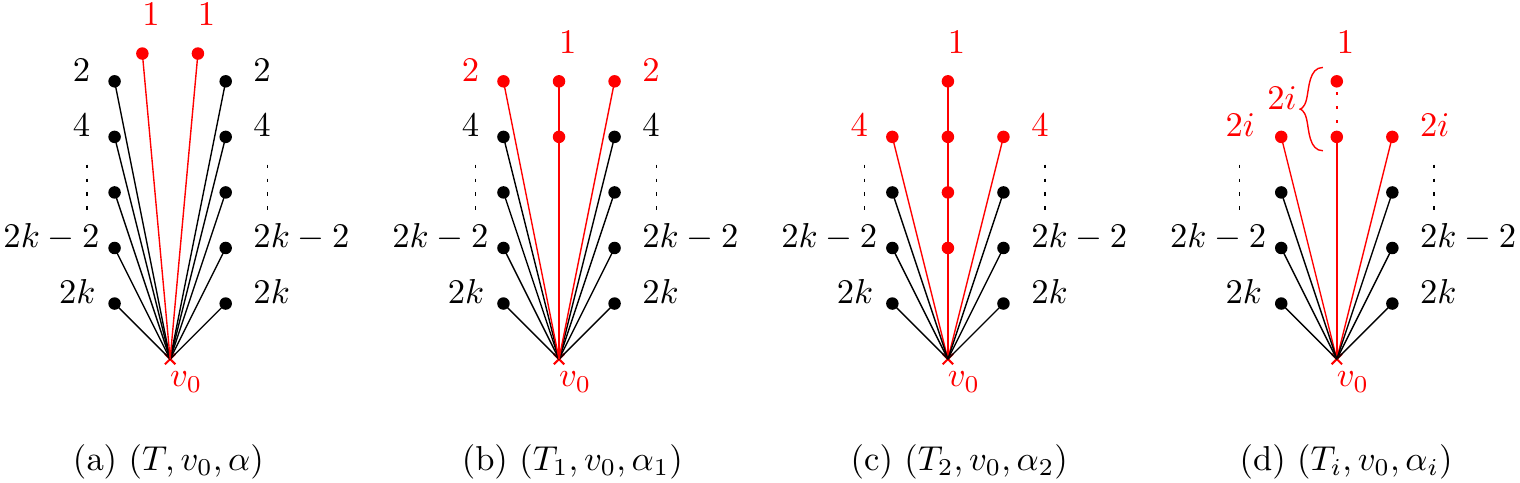}
\caption{Trees $(T,v_{0},\alpha)$, $(T_{1},v_{0},\alpha_{1})$, $(T_{2},v_{0},\alpha_{2})$, and $(T_{i},v_{0},\alpha_{i})$}
\label{fig:ex_prod_formula_plucking}
\end{figure}

\section{Removable Arc Theorem}
\label{s:removable_arc_thm}

In this section, we use $\Theta_{A}$-state expansion to show that $C(A)$ for a Catalan state $C$ with ``removable arcs'' (see Definition~\ref{def:removable_arc}) can be found by computing $C'(A)$ for the Catalan state $C'$ with those arcs removed. Thus, in particular, one may use Theorem~\ref{thm:remove_an_arc} to reduce complexity of computations for $C(A)$ in such a case.

For a crossingless connection $C$ in $\mathrm{R}^{2}_{m,n,2k-n}$ with vertices $v_{1},v_{2},v_{3},v_{4}$ and $-\lfloor\frac{n}{2}\rfloor \leq t \leq m$, let $\tau_{t}(C)$ be the crossingless connection in $\mathrm{R}^{2}_{m-t,n+2t,2k-n}$ obtained from $C$ by shifting $v_{1}$ and $v_{2}$ by $t$ units down if $t \geq 0$ and $-t$ units up if $t \leq 0$ (see example in Figure~\ref{fig:tau_C} when $t = 3$). 

\begin{figure}[htb]
\centering
\includegraphics[scale=1]{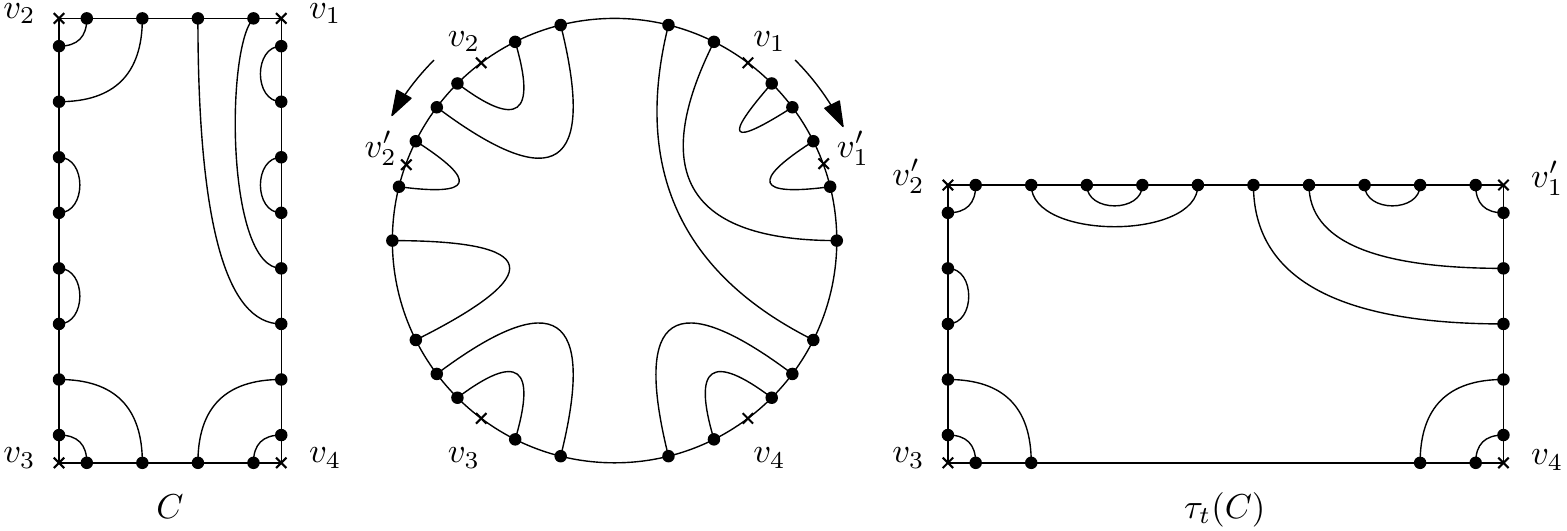}
\caption{Crossingless connections $C$ and $\tau_{t}(C)$}
\label{fig:tau_C}
\end{figure}

An arc $c$ of a crossingless connection $C$ which neither connects bottom- and top-boundary points nor it is a left or a right return will be called a \emph{proper arc}. Assume that $c$ is a proper arc of $C$ in $\mathrm{R}^{2}_{m,n,2k-n}$, where $m \geq 1$. Define a crossingless connection $C \smallsetminus c$ as follows:
\begin{enumerate}
\item[i)] If $c$ has no ends on the bottom boundary then
\begin{equation*}
C \smallsetminus c = \tau_{1-m}(\tau_{m}(C)-\{c\}), 
\end{equation*}
where $\tau_{m}(C)-\{c\}$ is the crossingless connection with $c$ and its ends removed from $\tau_{m}(C)$ (see Figure~\ref{fig:C_remove_c}).
\item[ii)] If $c$ has an end on the bottom boundary then
\begin{equation*}
C \smallsetminus c = (C^{*} \smallsetminus c^{*})^{*}    
\end{equation*}
where $c^{*}$ is the image of $c$ after a $\pi$-rotation $C^{*}$ of $C$ (see Figure~\ref{fig:C_star_remove_c_star}).
\end{enumerate}

\begin{figure}[ht]
\centering
\includegraphics[scale=1]{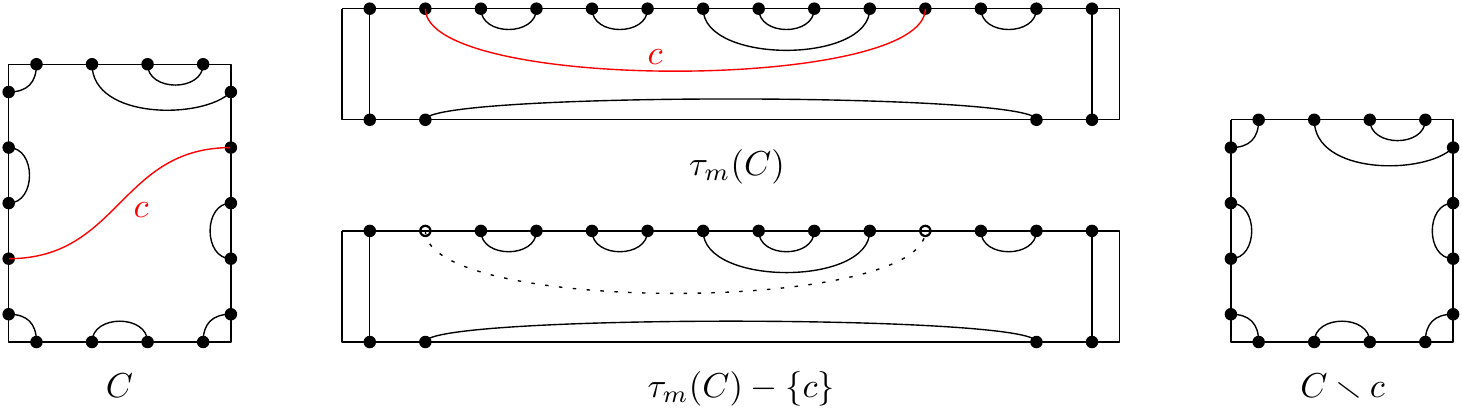}
\caption{Catalan state $C \smallsetminus c$ for $c$ with no ends on the bottom boundary}
\label{fig:C_remove_c}
\end{figure}

\begin{figure}[ht]
\centering
\includegraphics[scale=1]{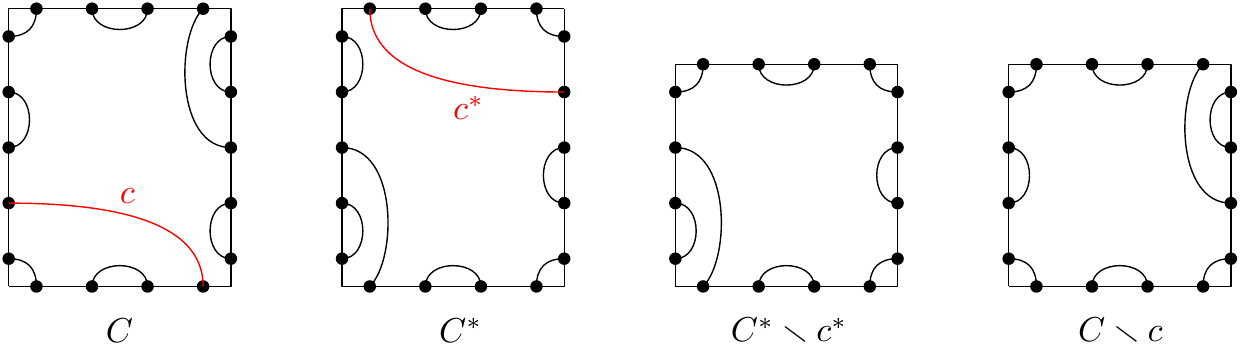}
\caption{Catalan state $C \smallsetminus c$ for $c$ with at least one end on the bottom boundary}
\label{fig:C_star_remove_c_star}
\end{figure}

We relabel points on the top and bottom boundaries of $\mathrm{R}^{2}_{m,n,n}$ as follows:
\begin{equation*}
x_{i} = y_{1-i} = y'_{i-n} \ \text{for} \ 1 \leq i \leq n
\end{equation*}
and
\begin{equation*}
x'_{i} = y_{m+i} = y'_{m+(n-i)+1} \ \text{for} \ 1 \leq i \leq n.
\end{equation*}

\begin{definition}
\label{def:removable_arc}
Let $C \in \mathrm{Cat}(m,n)$, where $m \geq 1$. An arc $c$ of $C$ is called \emph{removable} if $c$ is a proper arc and there is a non-negative integer $j_{0} \leq m-1$ such that all other arcs of $C$ with ends $y_{j},y_{j+1}$ or $y'_{j},y'_{j+1}$ are in $A_{1}$ if $j \leq j_{0}$ and in $A_{2}$ if $j > j_{0}$, where $A_{1}$ and $A_{2}$ are regions into which $c$ splits $\mathrm{R}^{2}_{m,n,n}$ with $A_{1}$ touching the top boundary and $A_{2}$ touching the bottom boundary (see Figure~\ref{fig:regions_a1_a2}).
\end{definition}

\begin{figure}[ht]
\centering
\includegraphics[scale=1]{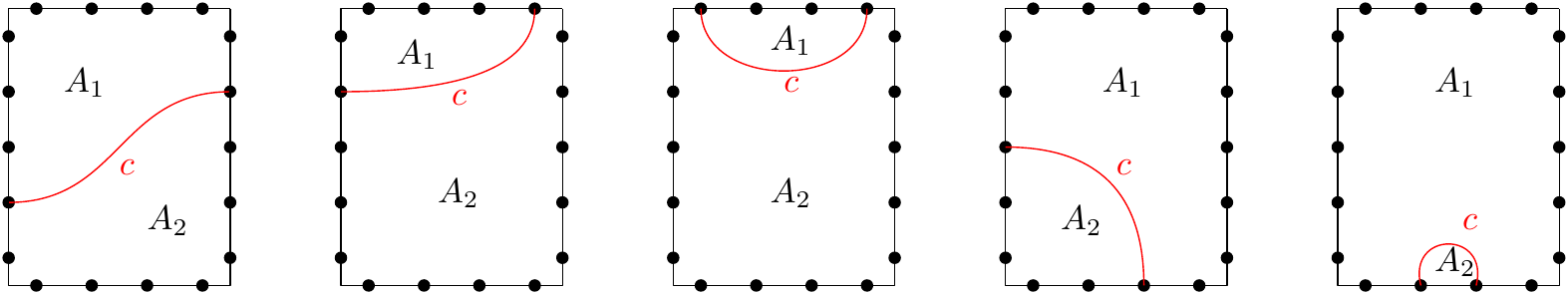}
\caption{Regions $A_{1}$ and $A_{2}$ determined by $c$}
\label{fig:regions_a1_a2}
\end{figure}

For a Catalan state with no left (or right) returns, arcs joining its left- and right-boundary points are clearly examples of removable arcs. We say that arcs $\{a_{1},a_{2},\ldots,a_{k}\}$ of a Catalan state $C$ constitute a family of \emph{parallel arcs} of $C$ if $\{p_{1},p_{2},\ldots,p_{k}\}$ and $\{q_{1},q_{2},\ldots,q_{k}\}$ form distinct sets of consecutive boundary points of $C$, where $p_{i},q_{i}$ are ends of $a_{i}$ for $i = 1,2,\ldots,k$. Arcs $c$ and $c'$ of $C$ are \emph{parallel} if they belong to a family of parallel arcs of $C$.

\begin{remark}
\label{rem:removable_arc}
If $c$ is a removable arc of $C \in \mathrm{Cat}(m,n)$, then all top returns of $C$ in $A_{2}$ are parallel to $c$ and $C$ has no \emph{innermost top corners} (i.e., arcs with ends $x_{1},y_{1}$ or $x_{n},y'_{1}$) in $A_{2}$, and all bottom returns of $C$ in $A_{1}$ are parallel to $c$ and $C$ has no \emph{innermost bottom corners} (i.e., arcs with ends $x'_{1},y_{m}$ or $x'_{n},y'_{m}$) in $A_{1}$. Indeed, if there is a top return that is not parallel to $c$ or an innermost top corner in $A_{2}$, then there must be an arc with ends $y_{i},y_{i+1}$ or $y'_{i},y'_{i+1}$ in $A_{2}$, where $i \leq 0$. Since $c$ is a removable arc, by Definition~\ref{def:removable_arc}, $i > j_{0}$ for some non-negative integer $j_{0}$, which is impossible. Analogous argument applies to the latter case.
\end{remark}

\begin{theorem}[Removable Arc Theorem]
\label{thm:remove_an_arc}
Let $c$ be a removable arc of $C \in \mathrm{Cat}(m,n)$ with the left end $y_{a}$ and the right end $y'_{b}$, where $m \geq 1$ and $1-n \leq a,b \leq m+n$. Then 
\begin{equation*}
C(A) = A^{b-a} \, C'(A),
\end{equation*}
where $C' = C \smallsetminus c$. In particular, $C$ is realizable if and only if $C'$ is realizable.
\end{theorem}

We start by showing that Theorem~\ref{thm:remove_an_arc} holds for Catalan states $C$ with no bottom returns which have a removable arc $c$ with no ends on the bottom boundary (see Lemma~\ref{lem:remove_an_arc_no_bot_rtn}). This will be achieved in two steps, i.e., we first establish a relation between plucking polynomials of plane rooted trees corresponding to $C$ and $C \smallsetminus c$ (see Lemma~\ref{lem:removable_arc_tree}) and then we find a relation between $\beta(C)$ and $\beta(C \smallsetminus c)$ (see Lemma~\ref{lem:removable_arc_beta}). For a Laurent polynomial $Q_{q}$ in variable $q$, define
\begin{equation*}
Q_{q}^{*} = \begin{cases}
q^{-\min\deg_{q} Q_{q}} \, Q_{q}, & \text{if} \ Q_{q} \neq 0,\\
0, & \text{if} \ Q_{q} = 0.
\end{cases}
\end{equation*}

\begin{lemma}
\label{lem:removable_arc_tree}
Let $C$ be a Catalan state with no bottom returns which has a removable arc $c$ whose both ends are not bottom-boundary points. Denote by $(T,v_{0},\alpha)$ and $(\tilde{T},v_{0},\tilde{\alpha})$ the plane rooted trees with delay for $C$ and $C \setminus c$, respectively. Then
\begin{equation*}
Q_{q}^{*}(T,v_{0},\alpha) = Q_{q}^{*}(\tilde{T},v_{0},\tilde{\alpha}).
\end{equation*}
In particular, $C$ is realizable if and only if $C \setminus c$ is realizable.
\end{lemma}

\begin{proof}
Let $e_{c}$ be the edge of $T$ corresponding to $c$, and let $T/e_{c}$ be contraction of $T$ by $e_{c}$. Define a weight function $\alpha_{0}$ on $L(T/e_{c})$ by 
\begin{equation*}
\alpha_{0}(v) =
\begin{cases}
\max\{\alpha(v)-1,1\}, & \text{if} \  v \in L(T/e_{c}) \cap (L(T) \setminus L(T_{v'})), \\ 
\alpha(v), & \text{if} \ v \in L(T/e_{c}) \cap L(T_{v'}), \\
1, & \text{if} \ v \in L(T/e_{c}) \setminus L(T),
\end{cases}
\end{equation*}
where $v'$ is the vertex of $T$ incident to $e_{c}$ such that $e_{c}$ is not an edge of $T_{v'}$. As it could easily be seen $\tilde{T} = T/e_{c}$ and $\tilde{\alpha} = \alpha_{0}$.

Since $c$ is a removable arc, there is a non-negative integer $j_{0}$ such that $\alpha(v) \leq j_{0}+1 < \alpha(w)$ for all $v \in L(T_{v'})$ and $w \in L(T) \setminus L(T_{v'})$. Thus, $(T',v',\alpha')$ is a splitting subtree of $(T,v_{0},\alpha)$, where $T' = T_{v'}$ and $\alpha' = \alpha |_{L(T')}$. 
Moreover, if $\tilde{v}'$ denotes the new vertex of $\tilde{T}$ obtained after identifying ends of $e_{c}$ in $T$, then 
\begin{equation*}
\tilde{\alpha}(v) \leq j_{0}+1 \leq \tilde{\alpha}(w)
\end{equation*}
for all $v \in L(\tilde{T}_{\tilde{v}'})$ and $w \in L(\tilde{T}) \setminus L(\tilde{T}_{\tilde{v}'})$. Consequently, $(\tilde{T}',\tilde{v}',\tilde{\alpha}')$ is a splitting subtree of $(\tilde{T},v_{0},\tilde{\alpha})$, where $\tilde{T}' = \tilde{T}_{\tilde{v}'}$ and $\tilde{\alpha}' = \tilde{\alpha} |_{L(\tilde{T}')}$. 
Clearly, 
\begin{equation*}
Q_{q}(T',v',\alpha') = Q_{q}(\tilde{T}',\tilde{v}',\tilde{\alpha}').
\end{equation*}
Let $(T'',v_{0},\alpha'')$ and $(\tilde{T}'',v_{0},\tilde{\alpha}'')$ be complementary trees for $(T',v',\alpha')$ and $(\tilde{T}',\tilde{v}',\tilde{\alpha}')$, respectively. Using Definition~\ref{def:q_poly} and Definition~\ref{def:complementary_tree}, one shows that
\begin{equation*}
Q_{q}(T'',v_{0},\alpha'') = q^{r_{*}} \, Q_{q}(T''-u,v_{0},\alpha''_{u}) = q^{r_{*}} \, Q_{q}(\tilde{T}'',v_{0},\tilde{\alpha}''),
\end{equation*}
where $u$ is the unique leaf of $T''$ with $\alpha''(u) = 1$ and $r_{*} = r(T'',v_{0},u)$. Therefore, by Theorem~\ref{thm:prod_formula_plucking}
\begin{equation*}
Q_{q}(T,v_{0},\alpha) = Q_{q}(T',v',\alpha') \cdot Q_{q}(T'',v_{0},\alpha'') = Q_{q}(\tilde{T}',\tilde{v}',\tilde{\alpha}') \cdot q^{r_{*}} \, Q_{q}(\tilde{T}'',v_{0},\tilde{\alpha}'') = q^{r_{*}} \, Q_{q}(\tilde{T},v_{0},\tilde{\alpha}),
\end{equation*}
and consequently $Q_{q}^{*}(T,v_{0},\alpha) = Q_{q}^{*}(\tilde{T},v_{0},\tilde{\alpha})$. Since $Q_{q}(T,v_{0},\alpha) \neq 0$ if and only if $Q_{q}(\tilde{T},v_{0},\tilde{\alpha}) \neq 0$, by Proposition~3.14 of \cite{DW2022}, $C$ is realizable if and only if $C\smallsetminus c$ is realizable.
\end{proof}

To prove Lemma~\ref{lem:removable_arc_beta}, we first establish relations between ends of arcs of a Catalan state $C$ and its maximal sequence $\mathbf{b}$ (see Lemma~\ref{lem:condi_bk_equals_n} and Lemma~\ref{lem:lr_relations}).

Given a crossingless connection $C$, let
\begin{equation*}
\mathcal{A}(C) = \{a_{1},a_{2},\ldots,a_{k}\}    
\end{equation*}
be the set of all its top returns, where the left ends $x_{i_{j}}$ of arcs $a_{j}$ satisfy condition $i_{1} > i_{2} > \cdots > i_{k}$. For $0 \leq j \leq k$, let
\begin{equation*}
\mathcal{A}(C,j) = \{a_{1},a_{2},\ldots,a_{j}\} \subseteq \mathcal{A}(C)    
\end{equation*}
be the set of $j$ rightmost top returns of $C$. If $\mathcal{A}(\tau_{j}(C))$ has at least $j$ elements, we define the crossingless connection 
\begin{equation*}
C_{(j)} = \tau_{j}(C) - \mathcal{A}(\tau_{j}(C),j)    
\end{equation*}
obtained from $\tau_{j}(C)$ after removing arcs $c \in \mathcal{A}(\tau_{j}(C),j)$ together with their ends. 

\begin{remark}
\label{rem:C_j}
We note that, if $C$ is a Catalan state of $L(m,n)$ with no bottom returns then $C_{(j)} \in \mathrm{Cat}(m-j,n)$ has no bottom returns for $j = 0,1,\ldots,m$. Indeed, since $C$ is realizable with no bottom returns, there is $\mathbf{b} \in \mathfrak{b}(C)$. Each subsequence $(b_{1},b_{2},\ldots,b_{j})$ of $\mathbf{b}$ corresponds to $j$ top returns of $\tau_{j}(C)$, so $\mathcal{A}(\tau_{j}(C))$ has at least $j$ elements and consequently $C_{(j)}$ is defined. Since $\tau_{j}(C)$ has $(n+2j)$ top-boundary points, the crossingless connection $C_{(j)}$ has $(n+2j)-2j = n$ top-boundary points. Moreover, $\tau_{j}(C)$ has $(m-j)$ left- and right-boundary points. Therefore, $C_{(j)}$ is as claimed.
\end{remark}

Let $C$ be a Catalan state of $L(m,n)$ with no bottom returns realized by $\mathbf{b} = (b_{1},b_{2},\ldots,b_{m})$ and let $C^{(j)} \in \mathrm{Cat}(j,n)$ be the Catalan state realized by its subsequence $\mathbf{b}^{(j)} = (b_{1},b_{2},\ldots,b_{j})$ for $j = 1,2,\ldots,m$. We say that an arc $c$ of $C$ has \emph{index $j$ relative to $\mathbf{b}$} if $j$ is the minimal integer such that $c = c'$ for some arc $c'$ of $C^{(j)}$ with none of its ends on the bottom boundary.

Using conventions and notations introduced above we state and prove the following lemma.

\begin{lemma}
\label{lem:catalan_state_subseq_b}
Let $C$ be a Catalan state with no bottom returns realized by the maximal sequence $\mathbf{b} = (b_{1},b_{2},\ldots,b_{m})$. Then, for all $0 \leq j \leq m$, Catalan state $C_{(j)}$ is realized by the maximal sequence $\mathbf{b}_{(j+1)} = (b_{j+1},b_{j+2},\ldots,b_{m})$.
\end{lemma}

\begin{proof}
First, we notice that for any $m \geq 0$ and $j = 0$, clearly $C_{(0)} = C$ is realized by the maximal sequence $\mathbf{b}_{(1)} = \mathbf{b}$. For $m \geq 1$ and $j = 1$, let $C' \in \mathrm{Cat}(1,n)$ be realized by the maximal sequence $(b_{1})$. Then $C = C' * C''$, where $C''$ is a Catalan state realized by $\mathbf{b}_{(2)}$. After considering cases for $C'$, one concludes that
\begin{equation*}
C'' = \tau_{1}(C) - \mathcal{A}(\tau_{1}(C),1),  
\end{equation*}
so $C_{(1)} = C''$ is realized by the maximal sequence $\mathbf{b}_{(2)}$.

We prove lemma by induction on $m = \mathrm{ht}(C)$. As it was shown above, statement of the lemma is true when $m = 0,1$. Let $m \geq 2$ and assume that the lemma holds for all realizable Catalan states $C'$ with no bottom returns and $\mathrm{ht}(C') < m$. Let $C \in \mathrm{Cat}(m,n)$ be a Catalan state realized by $\mathbf{b}$. As it has already been shown above $C_{(j)}$ is realized by the maximal sequence $\mathbf{b}_{(j+1)}$ for $j = 0,1$. 

For $2 \leq j \leq m$, let $C = C' * C''$, where $C' \in \mathrm{Cat}(1,n)$ is realized by $\mathbf{b}' = (b_{1})$ and $C'' = C_{(1)} \in \mathrm{Cat}(m-1,n)$ is realized by $\mathbf{b}'' = \mathbf{b}_{(2)}$. By the induction hypothesis, $C''_{(j-1)}$ is realized by the maximal sequence $\mathbf{b}''_{(j)} = (b_{j+1},b_{j+2},\ldots,b_{m})$. So, it suffices to show that $C''_{(j-1)} = C_{(j)}$, i.e.,
\begin{equation}
\label{eqn:pf_catalan_state_subseq_b_1}
\tau_{j-1}(C'') - \mathcal{A}(\tau_{j-1}(C''),j-1) = \tau_{j}(C) - \mathcal{A}(\tau_{j}(C),j).
\end{equation}
Since $C'' = \tau_{1}(C) - \mathcal{A}(\tau_{1}(C),1)$, 
\begin{equation}
\label{eqn:pf_catalan_state_subseq_b_2}
\tau_{j-1}(C'') = \tau_{j-1}(\tau_{1}(C) - \mathcal{A}(\tau_{1}(C),1)) = \tau_{j}(C) - \{a\}
\end{equation}
for some $a \in \mathcal{A}(\tau_{j}(C))$ that corresponds the unique top return in $\mathcal{A}(\tau_{1}(C),1)$. Since
\begin{equation*}
\tau_{j}(C) = \tau_{j-1}(\tau_{1}(C)),
\end{equation*}
the numbers of top-boundary points of $\tau_{j}(C)$ and $\tau_{1}(C)$ differ by $2(j-1)$ with $(j-1)$ of them are to the right of the right end $q$ of $a$. As each of these $(j-1)$ points is an end of an arc of $\tau_{j}(C)$, it follows that $\tau_{j}(C)$ has at most $(j-1)$ top returns with their left ends to the right of $q$. Therefore, $a \in \mathcal{A}(\tau_{j}(C),j)$ and, using analogous arguments for the top returns in $\mathcal{A}(\tau_{j-1}(C''),j-1)$, one also shows that
\begin{equation*}
\mathcal{A}(\tau_{j-1}(C''),j-1) \subset \mathcal{A}(\tau_{j}(C),j).
\end{equation*}
Since $a \notin \mathcal{A}(\tau_{j-1}(C''),j-1)$, it follows that
\begin{equation}
\label{eqn:pf_catalan_state_subseq_b_3}
\mathcal{A}(\tau_{j-1}(C''),j-1) \cup \{a\} = \mathcal{A}(\tau_{j}(C),j).
\end{equation}
Hence, \eqref{eqn:pf_catalan_state_subseq_b_1} is a consequence of \eqref{eqn:pf_catalan_state_subseq_b_2} and \eqref{eqn:pf_catalan_state_subseq_b_3}. Since $\mathbf{b}''_{(j)} = \mathbf{b}_{(j+1)}$, it follows that $C_{(j)}$ is realized by the maximal sequence $\mathbf{b}_{(j+1)}$ for $2 \leq j \leq m$.
\end{proof}

Define a \emph{coordinate function} $\iota_{m,n}$ on the set of top-, left-, and right-boundary points of $\mathrm{R}^{2}_{m,n,n}$ by
\begin{equation*}
\iota_{m,n}(p) =
\begin{cases}
i, & \text{if}\ p = y'_{i} \ \text{for some} \ 1-n \leq i \leq m, \\
-n+1-i, & \text{if}\ p = y_{i} \ \text{for some} \ 1 \leq i \leq m.
\end{cases}
\end{equation*}
Given $C \in \mathrm{Cat}(m,n)$ with no bottom returns which has an arc $c$ with none of its ends $p,q$ on the bottom boundary, we say that $p$ is the \emph{left end} and $q$ is the \emph{right end} of $c$ if
\begin{equation*}
\iota_{m,n}(p) < \iota_{m,n}(q)
\end{equation*}
and we write $c = (p,q)$. Let $\mathbf{b} = (b_{1},\ldots,b_{m})$ be a sequence that realizes $C$. For $1 \leq j \leq m$ define
\begin{equation*}
l_{j,n}(\mathbf{b}) = \iota_{m,n}(p) \quad \text{and} \quad r_{j,n}(\mathbf{b}) = \iota_{m,n}(q),
\end{equation*}
where $(p,q)$ is an arc of $C$ with index $j$ relative to $\mathbf{b}$. 

\begin{lemma}
\label{lem:condi_bk_equals_n}
Let $\mathbf{b} = (b_{1},\ldots,b_{m})$ be the maximal sequence for a realizable $C \in \mathrm{Cat}(m,n)$ with no bottom returns. For every $1 \leq j \leq m$,
\begin{equation*}
l_{j,n}(\mathbf{b}) + r_{j,n}(\mathbf{b}) \geq 1 \ \text{if and only if} \ b_{j} = n.
\end{equation*}
\end{lemma}

\begin{proof}
For $m = 0$ or $n = 0$ the statement is clearly true. Thus, we assume that $n \geq 1$ and we prove lemma by induction on $m \geq 1$. For $m = 1$, let $(p,q)$ be the unique arc of $C$ with neither $p$ nor $q$ on the bottom boundary. Then $\iota_{1,n}(q) = \iota_{1,n}(p)+1$, hence 
\begin{equation*}
l_{1,n}(\mathbf{b}) + r_{1,n}(\mathbf{b}) = \iota_{1,n}(p) + \iota_{1,n}(q) = 2 \iota_{1,n}(p) + 1 \geq 1
\end{equation*}
if and only if $\iota_{1,n}(p) \geq 0$, i.e., $p = x_{n}$ and $q = y'_{1}$. Consequently, $l_{1,n}(\mathbf{b}) + r_{1,n}(\mathbf{b}) \geq 1$ if and only if $b_{1} = n$.

Assume that $m > 1$ and the statement is true for all realizable $C \in \mathrm{Cat}(m-1,n)$ with no bottom returns. Consider a realizable $C \in \mathrm{Cat}(m,n)$ with no bottom returns realized by the maximal sequence $\mathbf{b} = (b_{1},\ldots,b_{m})$. By Lemma~\ref{lem:catalan_state_subseq_b}, $C' = C_{(1)} \in \mathrm{Cat}(m-1,n)$ has no bottom returns and it is realized by $\mathbf{b}' = (b'_{1},\ldots,b'_{m-1}) = (b_{2},\ldots,b_{m})$. 

For $j = 1$, an analogous argument as for case $m = 1$ shows that
\begin{equation*}
l_{1,n}(\mathbf{b}) + r_{1,n}(\mathbf{b}) \geq 1 \ \text{if and only if} \ b_{1} = n.
\end{equation*}
For $2 \leq j \leq m$, let $(p,q),(p',q')$, and $(p'',q'')$ be arcs of index $j,j-1$, and $1$ relative to $\mathbf{b},\mathbf{b}'$, and $\mathbf{b}$, respectively. For convenience, we also let $\iota(x) = \iota_{m,n}(x)$ and $\iota'(x) = \iota_{m-1,n}(x)$ for a boundary point $x$. There are three cases shown in Figure~\ref{fig:pf_lem_condi_bk_equals_n}(a)-(c) for the position of $(p'',q'')$ relative to $(p,q)$, i.e.,
\begin{itemize}
\item[i)] $\iota(q'') < \iota(p)$,
\item[ii)] $\iota(p) < \iota(p'') < \iota(q'') < \iota(q)$, or
\item[iii)] $\iota(q) < \iota(p'')$.
\end{itemize}

\begin{figure}[htb]
\centering
\includegraphics[scale=1]{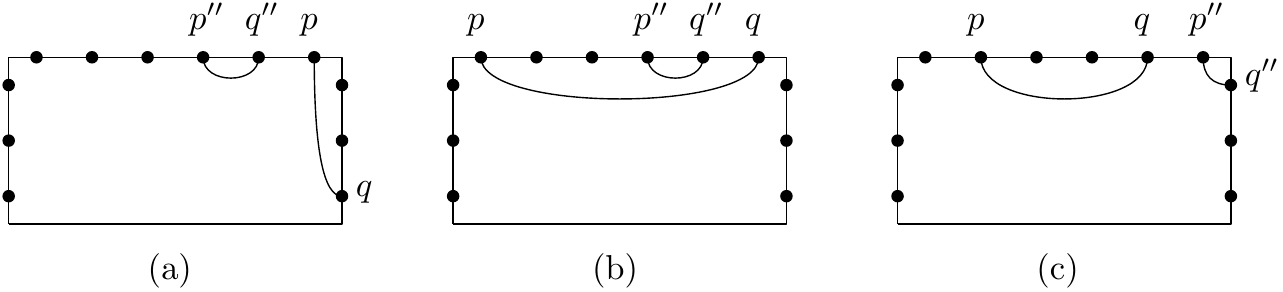}
\caption{Cases for relative positions of $(p,q)$ and $(p'',q'')$}
\label{fig:pf_lem_condi_bk_equals_n}
\end{figure}

\medskip

In case i), as it is easy to see
\begin{equation*}
\iota'(p') = \iota(p)-1 \ \text{and} \ \iota'(q') = \iota(q)-1.
\end{equation*}
We show that $\iota(q) > 0$. Suppose that $\iota(q) \leq 0$, then there is a top return $(x_{i},x_{i+1})$ of $C$ with $i$ satisfying
\begin{equation*}
i > n+\iota(p) > n+\iota(q'') > b_{1}.
\end{equation*}
However, this is impossible since $\mathbf{b}$ is the maximal sequence. Furthermore, we claim that
\begin{equation*}
\iota(p)+\iota(q) \geq 3.
\end{equation*}
Indeed, this is true when $\iota(p) \geq 1$. For $\iota(p) \leq 0$, if $\iota(q) \leq 1-\iota(p)$ then $C$ has either an arc $(x_{i},x_{i+1})$ with $i > b_{1}$ or an arc $(x_{n},y'_{1})$. This is impossible since $\mathbf{b}$ is the maximal sequence. Therefore,
\begin{equation*}
\iota(q) > 1-\iota(p)
\end{equation*}
and, since the number of boundary points between $p$ and $q$ is even, it must also be $\iota(q) > 2-\iota(p)$, i.e., $\iota(p)+\iota(q) \geq 3$ as claimed. Consequently,
\begin{equation*}
l_{j-1,n}(\mathbf{b}') + r_{j-1,n}(\mathbf{b}') = \iota'(p')+\iota'(q') = \iota(p)+\iota(q)-2 \geq 1, 
\end{equation*}
and, using the induction hypothesis for $C'$, we see that $b'_{j-1} = n$. Thus,
\begin{equation*}
l_{j,n}(\mathbf{b}) + r_{j,n}(\mathbf{b}) = \iota(p)+\iota(q) \geq 3 > 1
\end{equation*}
and $b_{j} = b'_{j-1} = n$.

\medskip

In case ii), as it is easy to see
\begin{equation*}
\iota'(p') = \iota(p)+1 \ \text{and} \ \iota'(q') = \iota(q)-1.
\end{equation*}
Therefore, by the induction hypothesis for $C'$,
\begin{equation*}
l_{j,n}(\mathbf{b}) + r_{j,n}(\mathbf{b}) = \iota(p)+\iota(q) = \iota'(p')+\iota'(q') = l_{j-1,n}(\mathbf{b}') + r_{j-1,n}(\mathbf{b}') \geq 1
\end{equation*}
if and only if $b_{j} = b'_{j-1} = n$.

\medskip

Finally, for the case iii), as it is easy to see
\begin{equation*}
\iota'(p') = \iota(p)+1 \ \text{and} \ \iota'(q') = \iota(q)+1.
\end{equation*} 
Clearly, $\iota(q) < 0$ and consequently $\iota(p) \leq \iota(q)-1 < -1$. Therefore,
\begin{equation*}
l_{j-1,n}(\mathbf{b}') + r_{j-1,n}(\mathbf{b}') = \iota'(p')+\iota'(q') = \iota(p)+\iota(q)+2 < 1, 
\end{equation*}
and using the induction hypothesis for $C'$, we see that $b'_{j-1} < n$. Thus,
\begin{equation*}
l_{j,n}(\mathbf{b}) + r_{j,n}(\mathbf{b}) = \iota(p)+\iota(q) < -1 < 1
\end{equation*}
and $b_{j} = b'_{j-1} < n$.

\medskip

Therefore, we showed that for any $m \geq 0$ and a Catalan state $C \in \mathrm{Cat}(m,n)$ realized by its maximal sequence $\mathbf{b} = (b_{1},b_{2},\ldots,b_{m})$,
\begin{equation*}
l_{j,n}(\mathbf{b}) + r_{j,n}(\mathbf{b}) \geq 1 \ \text{if and only if} \ b_{j} = n \quad \text{for} \ j = 1,2,\ldots,m.
\end{equation*}
\end{proof}

\begin{lemma}
\label{lem:lr_relations}
Let $\mathbf{b} = (b_{1},\ldots,b_{m})$ be the maximal sequence for a realizable $C \in \mathrm{Cat}(m,n)$ with no bottom returns.
\begin{enumerate}
\item[i)] If $b_{j} < n$, then $l_{j,n}(\mathbf{b}) = b_{j}-n-(j-1)$.
\item[ii)] If $b_{j} = n$, then $r_{j,n}(\mathbf{b}) = j$.
\end{enumerate}
\end{lemma}

\begin{proof}
If $b_{j} < n$ then clearly the left end $p$ of the arc $c$ of index $j$ relative to $\mathbf{b}$ is on the left or top boundary by Lemma~\ref{lem:condi_bk_equals_n}. Since $\tau_{j-1}(C)$ has $j-1$ additional top-boundary points to the right of $p$ and all arcs in $\mathcal{A}(\tau_{j-1}(C),j-1)$ have their left ends to the right of $p$. Hence, 
\begin{equation*}
\iota_{m-j+1,n}(p) = l_{j,n}(\mathbf{b})-(j-1)+2(j-1)    
\end{equation*}
in $C_{(j-1)} = \tau_{j-1}(C)-\mathcal{A}(\tau_{j-1}(C),j-1)$, which is the Catalan state realized by $\mathbf{b}_{(j)} = (b_{j},\ldots,b_{m})$ by Lemma~\ref{lem:catalan_state_subseq_b}. Since the index of $c$ relative to $\mathbf{b}_{(j)}$ is $1$, 
\begin{equation*}
l_{1,n}(\mathbf{b}_{(j)}) = \iota_{m-j+1,n}(p) = l_{j,n}(\mathbf{b})+(j-1).    
\end{equation*}
Finally, after we notice that $l_{1,n}(\mathbf{b}_{(j)}) = b_{j}-n$, i) follows.

If $b_{j} = n$ then clearly the right end $q$ of the arc $c$ of index $j$ relative to $\mathbf{b}$ is on the right boundary by Lemma~\ref{lem:condi_bk_equals_n}. Comparing to $C$, $\tau_{j-1}(C)$ has $j-1$ points less on its right boundary above $q$. Hence, 
\begin{equation*}
\iota_{m-j+1,n}(q) = r_{j,n}(\mathbf{b})-(j-1)
\end{equation*}
in $C_{(j-1)} = \tau_{j-1}(C)-\mathcal{A}(\tau_{j-1}(C),j-1)$. By the definition, index of $c$ relative to $\mathbf{b}_{(j)}$ is $1$, so 
\begin{equation*}
r_{1,n}(\mathbf{b}_{(j)}) = \iota_{m-j+1,n}(q) = r_{j,n}(\mathbf{b})-(j-1). 
\end{equation*}
Since $r_{1,n}(\mathbf{b}_{(j)}) = 1$, statement ii) follows.
\end{proof}

As a consequence of Lemma~\ref{lem:condi_bk_equals_n} and Lemma~\ref{lem:lr_relations}, one can determine $\beta(C)$ and the maximal sequence $\mathbf{b}$ in terms of ends of arcs of a realizable Catalan state $C$ with no bottom returns. Let $S(C)$ be the set of all arcs $(p,q)$ of $C$ with neither $p$ nor $q$ on the bottom boundary.

\begin{corollary}
\label{cor:lr_relations}
Let $C \in \mathrm{Cat}(m,n)$ be a realizable Catalan state with no bottom returns. Then
\begin{equation*}
\beta(C) = mn + \frac{m(m-1)}{2} + \sum_{(p,q) \in S(C)} \min\{\iota_{m,n}(p),1-\iota_{m,n}(q)\}.
\end{equation*}
\end{corollary}

\begin{proof}
By the definition of $\beta(C)$ and Lemma~\ref{lem:lr_relations}(i),
\begin{equation*}
\beta(C) = \sum_{\{j: b_{j} = n\}} b_{j} + \sum_{\{j: b_{j} < n\}} b_{j} = \sum_{\{j: b_{j} = n\}} n + \sum_{\{j: b_{j} < n\}} (l_{j,n}(\mathbf{b}) + j-1 + n).
\end{equation*}
It follows that
\begin{equation*}
\beta(C) = mn + \sum_{\{j: b_{j} < n\}} (l_{j,n}(\mathbf{b}) + j-1) = mn + \frac{m(m-1)}{2} + \sum_{\{j: b_{j} < n\}} l_{j,n}(\mathbf{b}) - \sum_{\{j: b_{j} = n\}} (j-1)
\end{equation*}
and by Lemma~\ref{lem:lr_relations}(ii),
\begin{equation*}
\beta(C) = mn + \frac{m(m-1)}{2} + \sum_{\{j: b_{j} < n\}} l_{j,n}(\mathbf{b}) + \sum_{\{j: b_{j} = n\}} (1-r_{j,n}(\mathbf{b})).
\end{equation*}
We can rewrite the above equation using Lemma~\ref{lem:condi_bk_equals_n} as
\begin{equation*}
\beta(C) = mn + \frac{m(m-1)}{2} + \sum_{j=1}^{m} \min\{l_{j,n}(\mathbf{b}),1-r_{j,n}(\mathbf{b})\},
\end{equation*}
so the statement of lemma follows.
\end{proof}

\begin{figure}[ht]
\centering
\includegraphics[scale=1]{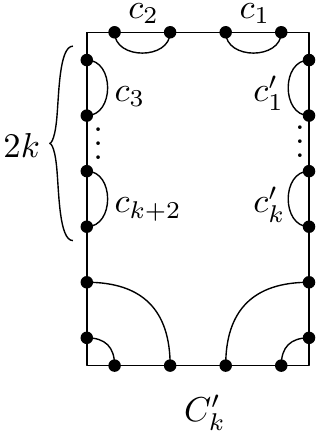}
\caption{Catalan state $C'_{k}$}
\label{fig:lr_relations}
\end{figure}

\begin{example}
\label{ex:lr_relations}
We illustrate Corollary~\ref{cor:lr_relations} by finding $\beta(C_{k}')$ for Catalan state $C'_{k} \in \mathrm{Cat}(2k+2,4)$ shown in Figure~\ref{fig:lr_relations}. 
Let $c_{i} = (p_{i},q_{i})$ and $c'_{j} = (p'_{j},q'_{j})$ be arcs as shown in Figure~\ref{fig:lr_relations}, where $i = 1,2,\ldots,k+2$ and $j = 1,2,\ldots,k$. For an arc $c = (p,q) \in S(C'_{k})$ define
\begin{equation*}
m(c) = \min\{\iota_{m',4}(p),1-\iota_{m',4}(q)\}, 
\end{equation*}
where $m' = 2k+2$. Then
\begin{eqnarray*}
m(c_{1}) &=& \min\{\iota_{m',4}(y'_{-1}),1-\iota_{m',4}(y'_{0})\} = \min\{-1,1-0\} = -1, \\
m(c_{2}) &=& \min\{\iota_{m',4}(y'_{-3}),1-\iota_{m',4}(y'_{-2})\} = \min\{-3,1-(-2)\} = -3,
\end{eqnarray*}
and for $i = 3,4,\ldots,k+2$,
\begin{equation*}
m(c_{i}) = \min\{\iota_{m',4}(y_{2(i-2)}),1-\iota_{m',4}(y_{2(i-2)-1})\} = \min\{1-2i,1-(2-2i)\} = 1-2i,   
\end{equation*}
and for $j = 1,2,\ldots,k$,
\begin{equation*}
m(c'_{j}) = \min\{\iota_{m',4}(y'_{2j-1}),1-\iota_{m',4}(y'_{2j})\} = \min\{2j-1,1-2j\} = 1-2j.   
\end{equation*}   
Therefore, by Corollary~\ref{cor:lr_relations},
\begin{eqnarray*}
\beta(C'_{k}) &=& 4m' + \frac{m'(m'-1)}{2} + \sum_{i=1}^{k+2} m(c_{i}) + \sum_{j=1}^{k} m(c'_{j})\\
&=& 4(2k+2) + \frac{(2k+2)(2k+2-1)}{2} - (k+2)^{2}-k^{2} = 7k+5.
\end{eqnarray*}
\end{example}

\begin{corollary}
\label{cor:b_seq}
Given a realizable $C \in \mathrm{Cat}(m,n)$ with no bottom returns, let
\begin{equation*}
S_{1} = \{(p,q) \in S(C) \mid \iota_{m,n}(p)+\iota_{m,n}(q) \geq 1\},\ S_{2} = S(C) \setminus S_{1}
\end{equation*}
\begin{equation*}
I_{1} = \{\iota_{m,n}(q) \mid (p,q) \in S_{1}\},\ I_{2} = \{1,2,\ldots,m\} \setminus I_{1},\ \text{and} \ I_{3} = \{\iota_{m,n}(p) \mid (p,q) \in S_{2}\}.
\end{equation*}
If $\mathbf{b} = (b_{1},b_{2},\ldots,b_{m})$ is the maximum sequence for $C$ and elements of $I_{2}, I_{3}$ are listed in an ascending and descending order, i.e., $I_{2} = \{i_{1} < i_{2} < \cdots < i_{s}\}$ and $I_{3} = \{l_{1} > l_{2} > \cdots > l_{s}\}$, where $s = |S_{2}|$. Then
\begin{equation*}
b_{j} = n\ \text{for}\ j \in I_{1} \quad \text{and} \quad b_{i_{j}} = n+(i_{j}-1)+l_{j} \ \text{for}\ j = 1,2,\ldots,s.
\end{equation*}
\end{corollary}

\begin{proof}
It follows from Lemma~\ref{lem:condi_bk_equals_n} and Lemma~\ref{lem:lr_relations}(ii) that $I_{1}$ consists of all indices $j \in \{1,2,\ldots,m\}$ for which $b_{j} = n$. Moreover, since for $(p,q),(p',q') \in S_{2}$ if $\iota_{m,n}(p) > \iota_{m,n}(p')$ then index of $(p,q)$ is strictly less than index of $(p',q')$ relative to $\mathbf{b}$, it follows from Lemma~\ref{lem:lr_relations}(i) that 
\begin{equation*}
l_{j} = b_{i_{j}}-n-(i_{j}-1)   
\end{equation*}
for $j = 1,2,\ldots,s$. Therefore, $b_{i_{j}}$ is given by the formula above.
\end{proof}

\begin{example}
\label{ex:lr_relations_b_sequence}
Using the same notations as in Example~\ref{ex:lr_relations}, we find the maximal sequence $\mathbf{b} = (b_{1},b_{2},\ldots,b_{m'})$ for Catalan state $C'_{k}$ in Figure~\ref{fig:lr_relations}. As one may check,
\begin{equation*}
S_{1} = \{(y'_{1},y'_{2}),\ldots,(y'_{2k-1},y'_{2k})\} \quad \text{and} \quad
S_{2} = \{(x_{3},x_{4}),(x_{1},x_{2}),(y_{2},y_{1}),\ldots,(y_{2k},y_{2k-1})\},
\end{equation*}
so
\begin{eqnarray*}
I_{1} &=& \{\iota_{m',4}(y'_{2j}) \mid 1\leq j \leq k\} = \{2j \mid 1\leq j \leq k\}, \\
I_{2} &=& \{1,2,\ldots,2k+2\} \setminus I_{1} = \{1,3,\ldots,2k-1,2k+1,2k+2\},
\end{eqnarray*}
and
\begin{equation*}
I_{3} = \{\iota_{m',4}(x_{3}),\iota_{m',4}(x_{1})\} \cup \{\iota_{m',4}(y_{2(j-2)}) \mid 3\leq j \leq k+2\} = \{-1,-3,\ldots,1-2(k+2)\}.
\end{equation*}
Therefore, $b_{i} = 4$ for all $i \in I_{1}$ by Corollary~\ref{cor:b_seq}, i.e., $b_{2j} = 4$ for $j = 1,2,\ldots,k$. Since
\begin{equation*}
i_{j} = \begin{cases}
2j-1 & \text{if}\ 1 \leq j \leq k+1, \\
2k+2 & \text{if}\ j = k+2, \\
\end{cases}
\quad \text{and} \quad l_{j} = 1-2j\ \text{for}\ 1 \leq j \leq k+2,  
\end{equation*}
it follows from Corollary~\ref{cor:b_seq} that for $1 \leq j \leq k+1$,
\begin{equation*}
b_{2j-1} = b_{i_{j}} = 4 + ((2j-1) -1) + (1-2j) = 3,
\end{equation*}
and 
\begin{equation*}
b_{2k+2} = b_{i_{k+2}} = 4 + ((2k+2) -1) + (1-2(k+2)) = 2.
\end{equation*}
\end{example}

To prove Lemma~\ref{lem:removable_arc_beta}, we need to introduce additional notations. Let $\mathbf{b} = (b_{1},\ldots,b_{m})$ be the maximal sequence for a realizable Catalan state $C \in \mathrm{Cat}(m,n)$ with no bottom returns. For $1 \leq k \leq m$ and $l,r \in \mathbb{Z}$, define sets
\begin{equation*}
U_{n}(\mathbf{b},l,r,k) = \{1 \leq j \leq k-1 \mid l < l_{j,n}(\mathbf{b}) < r_{j,n}(\mathbf{b}) < r\}
\end{equation*}
and
\begin{equation*}
L_{n}(\mathbf{b},l,r,k) = \{1,2,\ldots,k-1\} \setminus U_{n}(\mathbf{b},l,r,k).
\end{equation*}
We split the above sets further into four sets according to the property $b_{j} < n$ or $b_{j} = n$, i.e.,
\begin{eqnarray*}
U_{n,1}(\mathbf{b},l,r,k) &=& \{j \in U_{n}(\mathbf{b},l,r,k) \mid b_{j} < n\}, \\
U_{n,2}(\mathbf{b},l,r,k) &=& \{j \in U_{n}(\mathbf{b},l,r,k) \mid b_{j} = n\}, \\
L_{n,1}(\mathbf{b},l,r,k) &=& \{j \in L_{n}(\mathbf{b},l,r,k) \mid b_{j} < n\}, \\
L_{n,2}(\mathbf{b},l,r,k) &=& \{j \in L_{n}(\mathbf{b},l,r,k) \mid b_{j} = n\}.
\end{eqnarray*}

\begin{lemma}
\label{lem:lr_relations_with_c1}
Assume that a Catalan state $C$ of $L(m,n)$ with no bottom returns has a removable arc $c = (y_{a},y'_{b})$, where $a,b \leq m$. Let $k$ be the index of $c$ relative to the maximal sequence $\mathbf{b}$ for $C$ and let $l = l_{k,n}(\mathbf{b})$ and $r = r_{k,n}(\mathbf{b})$. Then
\begin{enumerate}
\item[i)] $L_{n,1}(\mathbf{b},l,r,k) = \emptyset$,
\item[ii)] $b_{k} = k-a$ and $|L_{n,2}(\mathbf{b},l,r,k)| = k - \frac{n+a+b}{2}$.
\end{enumerate}
\end{lemma}

\begin{proof}
We start by showing that
\begin{equation}
\label{eqn:pf_lem_lr_relations_with_c1_L1}
L_{n,1}(\mathbf{b},l,r,k) = \{j \in L_{n}(\mathbf{b},l,r,k) \mid l_{j,n}(\mathbf{b}) < r_{j,n}(\mathbf{b}) < l\}
\end{equation}
and
\begin{equation}
\label{eqn:pf_lem_lr_relations_with_c1_L2}
L_{n,2}(\mathbf{b},l,r,k) = \{j \in L_{n}(\mathbf{b},l,r,k) \mid r_{j,n}(\mathbf{b}) > l_{j,n}(\mathbf{b}) > r\}.
\end{equation}
For any $j \in L_{n}(\mathbf{b},l,r,k)$, either
\begin{equation*}
l_{j,n}(\mathbf{b}) \leq l \quad \text{or} \quad r_{j,n}(\mathbf{b}) \geq r.  
\end{equation*}
Since $j < k$, the arc with index $j$ relative to $\mathbf{b}$ is realized before $c$ and consequently
\begin{equation*}
l_{j,n}(\mathbf{b}) \leq l < r \leq r_{j,n}(\mathbf{b})   
\end{equation*}
is impossible. Therefore, as arcs of $C$ do not intersect, it must be either
\begin{equation}
\label{eqn:pf_lem_lr_relations_with_c1_1l}
l_{j,n}(\mathbf{b}) < r_{j,n}(\mathbf{b}) < l   
\end{equation}
or 
\begin{equation}
\label{eqn:pf_lem_lr_relations_with_c1_1r}
r_{j,n}(\mathbf{b}) > l_{j,n}(\mathbf{b}) > r,    
\end{equation}
i.e., each arc of $C$ with index $j \in L_{n}(\mathbf{b},l,r,k)$ relative to $\mathbf{b}$ must have its both ends either to the left of $y_{a}$ or to the right of $y'_{b}$, respectively.

We show that there is no $j \in L_{n,1}(\mathbf{b},l,r,k)$ which satisfies \eqref{eqn:pf_lem_lr_relations_with_c1_1r}. Suppose that there is such $j$ and let $c_{j} = (p_{j},q_{j})$ be the arc of $C$ with index $j$ relative to $\mathbf{b}$. Then $b_{j} < n$ and by Lemma~\ref{lem:condi_bk_equals_n},
\begin{equation}
\label{eqn:pf_lem_lr_relations_with_c1_L1new_ineq}
l_{j,n}(\mathbf{b})+r_{j,n}(\mathbf{b}) < 1.
\end{equation}
Using \eqref{eqn:pf_lem_lr_relations_with_c1_1r} and \eqref{eqn:pf_lem_lr_relations_with_c1_L1new_ineq}, we see that $r < 0$ and consequently $y'_{b}$ is a top-boundary point. It follows from \eqref{eqn:pf_lem_lr_relations_with_c1_L1new_ineq} that $c_{j}$ is not a right return, so either $c_{j}$ is an arc joining top- and right-boundary points or it is a top return. In the former case, since
\begin{equation*}
r_{j,n}(\mathbf{b}) < 1-l_{j,n}(\mathbf{b})   
\end{equation*}
(i.e., the number of top-boundary points to the right of $p_{j}$ is greater than the number of right-boundary points to the left of $q_{j}$), $C$ has a top return with both ends to the right of $p_{j}$. Thus, in both cases, $C$ has a top return not parallel to $c$ in the region $A_{2}$ (see Definition~\ref{def:removable_arc}), which is impossible by Remark~\ref{rem:removable_arc}. Hence, we showed that for every $j \in L_{n,1}(\mathbf{b},l,r,k)$ only \eqref{eqn:pf_lem_lr_relations_with_c1_1l} can hold. Consequently,
\begin{equation}
\label{eqn:pf_lem_lr_relations_with_c1_L1new}
L_{n,1}(\mathbf{b},l,r,k) \subseteq \{j \in L_{n}(\mathbf{b},l,r,k) \mid l_{j,n}(\mathbf{b}) < r_{j,n}(\mathbf{b}) < l\}. 
\end{equation}

We show that there is no $j \in L_{n,2}(\mathbf{b},l,r,k)$ that satisfies \eqref{eqn:pf_lem_lr_relations_with_c1_1l}. Suppose that there is such $j$. Since $b_{j} = n$,
\begin{equation}
\label{eqn:pf_lem_lr_relations_with_c1_L2new_ineq}
l_{j,n}(\mathbf{b})+r_{j,n}(\mathbf{b}) \geq 1
\end{equation}
by Lemma~\ref{lem:condi_bk_equals_n}.
However, $l \leq 0$ as $c$ is a proper arc, so by \eqref{eqn:pf_lem_lr_relations_with_c1_1l}
\begin{equation*}
l_{j,n}(\mathbf{b})+r_{j,n}(\mathbf{b}) < 2l \leq 0,
\end{equation*}
which contradicts to \eqref{eqn:pf_lem_lr_relations_with_c1_L2new_ineq}. Hence, we showed that, for every $j \in L_{n,2}(\mathbf{b},l,r,k)$ only \eqref{eqn:pf_lem_lr_relations_with_c1_1r} can hold. Consequently,
\begin{equation}
\label{eqn:pf_lem_lr_relations_with_c1_L2new}
L_{n,2}(\mathbf{b},l,r,k) \subseteq \{j \in L_{n}(\mathbf{b},l,r,k) \mid r_{j,n}(\mathbf{b}) > l_{j,n}(\mathbf{b}) > r\}. 
\end{equation}

Clearly,
\begin{equation*}
L_{n}(\mathbf{b},l,r,k) = L_{n,1}(\mathbf{b},l,r,k) \cup L_{n,2}(\mathbf{b},l,r,k)
\end{equation*}
and
\begin{equation*}
L_{n}(\mathbf{b},l,r,k) = \{j \in L_{n}(\mathbf{b},l,r,k) \mid l_{j,n}(\mathbf{b}) < r_{j,n}(\mathbf{b}) < l\} \cup \{j \in L_{n}(\mathbf{b},l,r,k) \mid r_{j,n}(\mathbf{b}) > l_{j,n}(\mathbf{b}) > l\},
\end{equation*}
so by \eqref{eqn:pf_lem_lr_relations_with_c1_L1new} and \eqref{eqn:pf_lem_lr_relations_with_c1_L2new} we conclude that \eqref{eqn:pf_lem_lr_relations_with_c1_L1} and \eqref{eqn:pf_lem_lr_relations_with_c1_L2} hold.

Now we find relations that express $b_{k}$ in terms of $|L_{n,1}(\mathbf{b},l,r,k)|$ and $|L_{n,2}(\mathbf{b},l,r,k)|$, respectively. Let $\tilde{c}$ be the arc of $\tau_{k-1}(C)$ that corresponds to $c$, and let $\tilde{p}$ be its left end. Then the number of top-boundary points to the right of $\tilde{p}$ in $\tau_{k-1}(C)$ and the number of top-boundary points to the right of $y_{a}$ in $C$ differ by $(k-1)$. By the definition, Catalan state $C_{(k-1)}$ is obtained from $\tau_{k-1}(C)$ by removing the set of arcs $\mathcal{A}(\tau_{k-1}(C),k-1)$ and this set obviously includes
\begin{equation*}
|U_{n}(\mathbf{b},l,r,k)|+|L_{n,2}(\mathbf{b},l,r,k)|    
\end{equation*}
arcs with their left ends to the right of $\tilde{p}$. Since by Lemma~\ref{lem:catalan_state_subseq_b} the Catalan state $C_{(k-1)}$ is realized by $\mathbf{b}_{(k)} = (b_{k+1},b_{k+2},\ldots,b_{m})$, it follows that
\begin{equation}
\label{eqn:pf_lem_lr_relations_with_c1_l1n}
l_{1,n}(\mathbf{b}_{(k)}) = l_{k,n}(\mathbf{b})-(k-1)+2|U_{n}(\mathbf{b},l,r,k)|+2|L_{n,2}(\mathbf{b},l,r,k)|.
\end{equation}

Analogously, one can show that
\begin{equation}
\label{eqn:pf_lem_lr_relations_with_c1_r1n}
r_{1,n}(\mathbf{b}_{(k)}) = r_{k,n}(\mathbf{b})-(k-1)+2|L_{n,2}(\mathbf{b},l,r,k)|.
\end{equation}

Furthermore,
\begin{equation}
\label{eqn:pf_lem_lr_relations_with_c1_l}
l = \iota_{m,n}(y_{a}) =
\begin{cases}
\iota_{m,n}(y_{a}) &  \text{if} \ a \geq 1 \\
\iota_{m,n}(y'_{1-n-a}) &  \text{if} \ a < 1
\end{cases}
= -n+1-a
\end{equation}
and 
\begin{equation}
\label{eqn:pf_lem_lr_relations_with_c1_r}
r = \iota_{m,n}(y'_{b}) = b.
\end{equation}
Notice that $\tilde{c}$ has index $1$ relative to $\mathbf{b}_{(k)}$, so its ends are consecutive boundary points. Therefore,
\begin{equation}
\label{eqn:pf_lem_lr_relations_with_c1_l1nr1n}
r_{1,n}(\mathbf{b}_{(k)}) = l_{1,n}(\mathbf{b}_{(k)})+1 = b_{k}-n+1.
\end{equation}
By the definition of $U_{n}(\mathbf{b},l,r,k)$, $L_{n,1}(\mathbf{b},l,r,k)$, and $L_{n,2}(\mathbf{b},l,r,k)$, we see that
\begin{equation}
\label{eqn:pf_lem_lr_relations_with_c1_UL1L2}
|U_{n}(\mathbf{b},l,r,k)|+|L_{n,1}(\mathbf{b},l,r,k)|+|L_{n,2}(\mathbf{b},l,r,k)| = k-1,
\end{equation}
so after using \eqref{eqn:pf_lem_lr_relations_with_c1_l}, \eqref{eqn:pf_lem_lr_relations_with_c1_l1nr1n}, and \eqref{eqn:pf_lem_lr_relations_with_c1_UL1L2}, the equation \eqref{eqn:pf_lem_lr_relations_with_c1_l1n} becomes
\begin{equation*}
-n+b_{k} = (-n+1-a)-(k-1)+2(k-1-|L_{n,1}(\mathbf{b},l,r,k)|)
\end{equation*}
or equivalently 
\begin{equation}
\label{eqn:pf_lem_lr_relations_with_c1_bk_v1}
b_{k} = k-a-2|L_{n,1}(\mathbf{b},l,r,k)|.
\end{equation}
Analogously, after using \eqref{eqn:pf_lem_lr_relations_with_c1_r} and \eqref{eqn:pf_lem_lr_relations_with_c1_l1nr1n}, the equation \eqref{eqn:pf_lem_lr_relations_with_c1_r1n} becomes 
\begin{equation*}
-n+b_{k}+1 = b-(k-1)+2|L_{n,2}(\mathbf{b},l,r,k)|)
\end{equation*}
or equivalently 
\begin{equation}
\label{eqn:pf_lem_lr_relations_with_c1_bk_v2}
b_{k} = b+n-k+2|L_{n,2}(\mathbf{b},l,r,k)|.    
\end{equation}

\medskip

For i), if $b_{k} < n$ then
\begin{equation*}
b_{k} = l_{k,n}(\mathbf{b})+n+(k-1)
\end{equation*}
by Lemma~\ref{lem:lr_relations}(i). Combining this with \eqref{eqn:pf_lem_lr_relations_with_c1_l} and \eqref{eqn:pf_lem_lr_relations_with_c1_bk_v1}, we see that
\begin{equation*}
(-n+1-a)+n+(k-1) = b_{k} = k-a-2|L_{n,1}(\mathbf{b},l,r,k)|
\end{equation*}
and consequently, $|L_{n,1}(\mathbf{b},l,r,k)| = 0$. If $b_{k} = n$ then by \eqref{eqn:pf_lem_lr_relations_with_c1_L1}, it suffices to show that $S = \emptyset$, where
\begin{equation*}
S = \{j \in L_{n}(\mathbf{b},l,r,k) \mid l_{j,n}(\mathbf{b}) < r_{j,n}(\mathbf{b}) < l\}.
\end{equation*}
Suppose that $S \neq \emptyset$ and let $j = \min S$. Since $b_{k} = n$, by Lemma~\ref{lem:condi_bk_equals_n}
\begin{equation*}
l_{k,n}(\mathbf{b})+r_{k,n}(\mathbf{b}) = (-n+1-a)+b \geq 1
\end{equation*}
or equivalently $b-a \geq n$. Let $j_{1}$ and $j_{2}$ be the maximal and the minimal numbers among all $i$ for which $(y_{i+1},y_{i})$ or $(y'_{i},y'_{i+1})$ is an arc of $C$ in $A_{1}$ and $A_{2}$ (see Definition~\ref{def:removable_arc} for $A_{1}$ and $A_{2}$), respectively. As one may check (see Figure~\ref{fig:pf_lem_lr_relations_with_c}(a)),
\begin{equation*}
\#(C \cap l^{h}_{j_{2}}) = (j_{2}-a)+(b-1-j_{2})+1 = b-a \geq n.
\end{equation*}
Since $C$ is realizable, it must be $b-a = n$ or equivalently 
\begin{equation*}
1-l_{k,n}(\mathbf{b}) = r_{k,n}(\mathbf{b}).
\end{equation*}
Let $c' = (p',q')$ be the arc of $C_{(j-1)}$ corresponding to $c$ in $C$ and let $c'' = (p'',q'')$ be the arc of $C_{(j-1)}$ with index $1$ relative to $\mathbf{b}_{(j)}$. Since $\mathbf{b}_{(j)}$ is the maximal sequence for $C_{(j-1)}$ by Lemma~\ref{lem:catalan_state_subseq_b}, no arc $(y'_{i},y'_{i+1})$ of $C_{(j-1)}$ with $i \leq 0$ has its ends to the right of $q''$. Therefore, neither $c'$ is a top return nor $c' = (y'_{0},y'_{1})$. Consequently, $c'$ joins top- and right-boundary points and 
\begin{equation*}
n_{l}+n_{r} \geq 1,   
\end{equation*}
where $n_{l}$ and $n_{r}$ are the number of top- and right-boundary points of $C_{(j-1)}$ between $p'$ and $q'$ in the region $A'$ shown in Figure~\ref{fig:pf_lem_lr_relations_with_c}(b). As one may check,
\begin{equation*}
n_{l} = 1-l_{k,n}(\mathbf{b})-(j-1) = r_{k,n}(\mathbf{b})-(j-1) = n_{r}.
\end{equation*}
However, this implies that there is an arc $(y'_{i},y'_{i+1})$ with $i \leq 0$ in $A'$ (which clearly has its ends to the right of $q''$), a contradiction.

\begin{figure}[htb]
\centering
\includegraphics[scale=1]{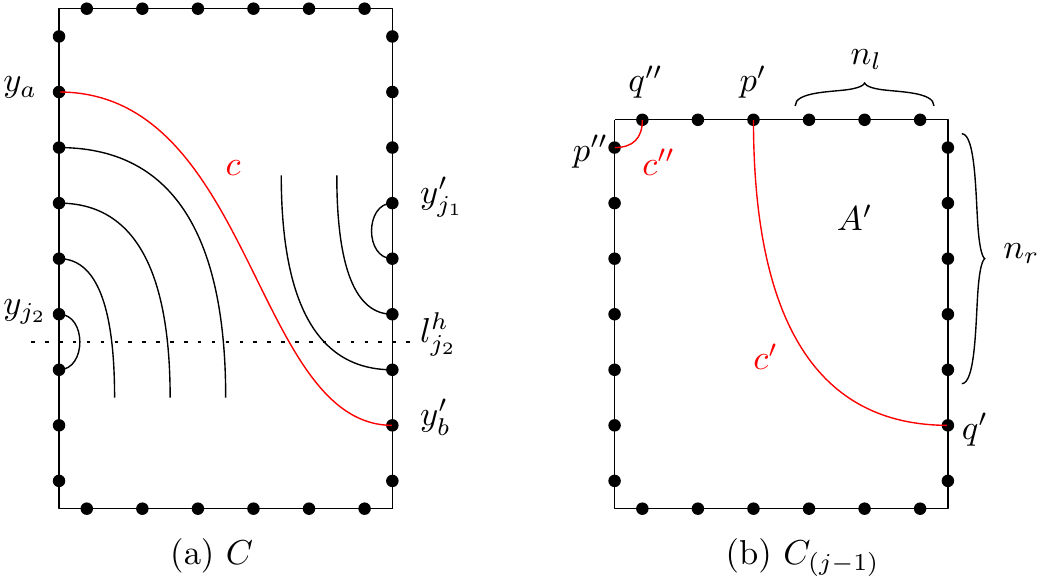}
\caption{Catalan states $C$ and $C_{(j-1)}$}
\label{fig:pf_lem_lr_relations_with_c}
\end{figure}

For ii), since $L_{n,1}(\mathbf{b},l,r,k) = \emptyset$ by i), after using \eqref{eqn:pf_lem_lr_relations_with_c1_bk_v1} and \eqref{eqn:pf_lem_lr_relations_with_c1_bk_v2} we see that,
\begin{equation*}
k-a = b_{k} = b+n-k+2|L_{n,2}(\mathbf{b},l,r,k)|.
\end{equation*}
Consequently, the formula for $b_{k}$ and $|L_{n,2}(\mathbf{b},l,r,k)|$ in ii) follows.
\end{proof}

\begin{lemma}
\label{lem:lr_relations_with_c2}
Assume that a Catalan state $C$ of $L(m,n)$ with no bottom returns has a removable arc $c = (y_{a},y'_{b})$, where $a,b \leq m$. Let $\mathbf{b}$ and $\mathbf{b}'$ be the maximal sequences for $C$ and $C \smallsetminus c$ and let $k$ be the index of $c$ relative to $\mathbf{b}$. Denote by $l = l_{k,n}(\mathbf{b})$ and $r = r_{k,n}(\mathbf{b})$. Then
\begin{enumerate}
\item[i)] $L_{n,1}(\mathbf{b}',l,r,k) = \emptyset$,
\item[ii)] $\{ l_{j,n}(\mathbf{b}') \mid j \in U_{n,1}(\mathbf{b}',l,r,k) \} = \{ l_{j,n}(\mathbf{b}) \mid j \in U_{n,1}(\mathbf{b},l,r,k) \}$,
\item[iii)] $\{ r_{j,n}(\mathbf{b}') \mid j \in U_{n,2}(\mathbf{b}',l,r,k) \} = \{ r_{j,n}(\mathbf{b}) \mid j \in U_{n,2}(\mathbf{b},l,r,k) \}$, and
\item[iv)] $\{ r_{j,n}(\mathbf{b}') \mid j \in L_{n,2}(\mathbf{b}',l,r,k) \} = \{ r_{j,n}(\mathbf{b})-1 \mid j \in L_{n,2}(\mathbf{b},l,r,k) \}$.
\end{enumerate}
\end{lemma}

\begin{proof}
Let $a_{j}$ be the arc of $C$ with index $j$ relative to $\mathbf{b}$, $1 \leq j \leq k$, and let $a'_{j}$ be the arc of $C \smallsetminus c$ with index $j$ relative to $\mathbf{b}'$, where $1 \leq j \leq k-1$. By Lemma~\ref{lem:catalan_state_subseq_b}, arcs $\mathcal{A}(\tau_{k}(C),k)$ correspond to arcs $\mathcal{X} = \{a_{1},a_{2},\ldots,a_{k}\}$ and arcs $\mathcal{A}(\tau_{k-1}(C \smallsetminus c),k-1)$ correspond to arcs $\mathcal{X}' = \{a'_{1},a'_{2},\ldots,a'_{k-1}\}$. Clearly,
\begin{equation*}
\mathcal{A}(\tau_{k}(C),k) = \mathcal{A}(\tau_{k-1}(C \smallsetminus c),k-1) \cup \{c\}  
\end{equation*}
and, by Lemma~\ref{lem:lr_relations_with_c1}(i), left ends of arcs of $\mathcal{A}(\tau_{k-1}(C \smallsetminus c),k-1)$ are to the right of the left end of $c$. There is a bijection $\sigma$ between $\mathcal{X} \setminus \{c\}$ and $\mathcal{X}'$ that maps an arc $a_{j} = (p_{j},q_{j})$ to the arc $a'_{j'} = (p'_{j'},q'_{j'})$ such that 
\begin{equation*}
\iota_{m-1,n}(p'_{j'}) = \iota_{m,n}(p_{j})\ \text{and}\ \iota_{m-1,n}(q'_{j'}) = \iota_{m,n}(q_{j}) \quad \text{if}\ j \in U_{n}(\mathbf{b},l,r,k)
\end{equation*}
and
\begin{equation*}
\iota_{m-1,n}(p'_{j'}) = \iota_{m,n}(p_{j})-1\ \text{and}\ \iota_{m-1,n}(q'_{j'}) = \iota_{m,n}(q_{j})-1 \quad \text{if}\ j \in L_{n}(\mathbf{b},l,r,k).
\end{equation*}
As one may check,
\begin{equation}
\label{eqn:pf_lr_relations_with_c_0U}
\{\sigma(a_{j}) \mid j \in U_{n}(\mathbf{b},l,r,k)\} = \{a'_{j'} \mid j' \in U_{n}(\mathbf{b}',l,r,k)\}
\end{equation}
and
\begin{equation}
\label{eqn:pf_lr_relations_with_c_0L}
\{\sigma(a_{j}) \mid j \in L_{n}(\mathbf{b},l,r,k)\} = \{a'_{j'} \mid j' \in L_{n}(\mathbf{b}',l,r,k)\}.
\end{equation}

\medskip

For i), suppose that $j' \in L_{n,1}(\mathbf{b}',l,r,k)$ and let $a_{j} = \sigma^{-1}(a'_{j'})$. Then 
\begin{equation*}
j \in L_{n}(\mathbf{b},l,r,k) = L_{n,2}(\mathbf{b},l,r,k)    
\end{equation*}
by Lemma~\ref{lem:lr_relations_with_c1}(i) and consequently
\begin{equation*}
r_{j,n}(\mathbf{b}) > l_{j,n}(\mathbf{b}) > r    
\end{equation*}
by \eqref{eqn:pf_lem_lr_relations_with_c1_L2}. Hence,
\begin{equation}
\label{eqn:pf_lem_lr_relations_with_c_3}
r_{j',n}(\mathbf{b}') > l_{j',n}(\mathbf{b}') = \iota_{m-1,n}(p'_{j'}) = \iota_{m,n}(p_{j}) - 1 = l_{j,n}(\mathbf{b})-1 \geq r.    
\end{equation}
Since $j' \in L_{n,1}(\mathbf{b}',l,r,k)$, $b'_{j'} < n$ and thus, by Lemma~\ref{lem:condi_bk_equals_n}
\begin{equation}
\label{eqn:pf_lem_lr_relations_with_c_4}
l_{j',n}(\mathbf{b}')+r_{j',n}(\mathbf{b}') < 1.    
\end{equation}
Using \eqref{eqn:pf_lem_lr_relations_with_c_3} and \eqref{eqn:pf_lem_lr_relations_with_c_4}, we see that $r < 0$ and consequently $y'_{b}$ is a top-boundary point in $C$, this implies that the left end $p'_{j'}$ of $a'_{j'}$ cannot be a left-boundary point. Moreover, $a'_{j'}$ is not a right return by \eqref{eqn:pf_lem_lr_relations_with_c_4}, so either $a'_{j'}$ is an arc joining top- and right-boundary points or it is a top return. In the former case, since 
\begin{equation*}
r_{j',n}(\mathbf{b}') < 1-l_{j',n}(\mathbf{b}')    
\end{equation*}
(i.e., the number of top-boundary points to the right of $p'_{j'}$ is greater than the number of right-boundary points to the left of $q'_{j'}$), $C \smallsetminus c$ has a top return $c'$ with both ends to the right of $p'_{j'}$. Thus, in both cases $C \smallsetminus c$ has a top return and consequently it must have a top return with consecutive top-boundary points $a'_{t'} = (y'_{i'},y'_{i'+1})$ for some $t' \in L_{n}(\mathbf{b}',l,r,k)$, where $r \leq i' < 0$. Since $\sigma$ is a bijection, there is $t \in L_{n}(\mathbf{b},l,r,k)$ such that
\begin{equation*}
a_{t} = \sigma^{-1}(a'_{t'}) = (y'_{i'+1},y'_{i'+2}).
\end{equation*}
It follows that $a_{t}$ is in the region $A_{2}$ (see Definition~\ref{def:removable_arc}) and either $a_{t} = (x_{n},y'_{1})$ or $a_{t}$ is a top return of $C$ that is not parallel to $c$. However, this is impossible by Remark~\ref{rem:removable_arc}, so $L_{n,1}(\mathbf{b}',l,r,k) = \emptyset$.

\medskip

For ii), we first show that
\begin{equation}
\label{eqn:pf_lr_relations_with_c_1}
\{\sigma(a_{j}) \mid j \in U_{n,1}(\mathbf{b},l,r,k)\} = \{a'_{j'} \mid j' \in U_{n,1}(\mathbf{b}',l,r,k)\}.
\end{equation}
If $a'_{j'} = \sigma(a_{j})$ for some $j \in U_{n,1}(\mathbf{b},l,r,k)$, then by Lemma~\ref{lem:condi_bk_equals_n},
\begin{equation*}
l_{j,n}(\mathbf{b})+r_{j,n}(\mathbf{b}) < 1.    
\end{equation*}
Hence,
\begin{equation*}
l_{j',n}(\mathbf{b}')+r_{j',n}(\mathbf{b}') = \iota_{m-1,n}(p'_{j'})+\iota_{m-1,n}(q'_{j'}) = \iota_{m,n}(p_{j})+\iota_{m,n}(q_{j}) = l_{j,n}(\mathbf{b})+r_{j,n}(\mathbf{b}) < 1    
\end{equation*}
by the definition of $\sigma$ and consequently $j' \in U_{n,1}(\mathbf{b}',l,r,k)$ by Lemma~\ref{lem:condi_bk_equals_n}, i.e.,
\begin{equation*}
\{\sigma(a_{j}) \mid j \in U_{n,1}(\mathbf{b},l,r,k)\} \subseteq \{a'_{j'} \mid j' \in U_{n,1}(\mathbf{b}',l,r,k)\}.
\end{equation*}
Since $\sigma$ is a bijection, inclusion in the opposite direction can be argued analogously. This completes our proof for \eqref{eqn:pf_lr_relations_with_c_1}. Consequently,
\begin{eqnarray*}
\{l_{j,n}(\mathbf{b}') \mid j \in U_{n,1}(\mathbf{b}',l,r,k)\} &=& \{\iota_{m-1,n}(p'_{j'}) \mid j' \in U_{n,1}(\mathbf{b}',l,r,k)\} = \{\iota_{m,n}(p_{j}) \mid j \in U_{n,1}(\mathbf{b},l,r,k)\} \\
&=& \{l_{j,n}(\mathbf{b}) \mid j \in U_{n,1}(\mathbf{b},l,r,k)\}.
\end{eqnarray*}

\medskip

For iii), we see that
\begin{eqnarray*}
\{a'_{j'} \mid j' \in U_{n,2}(\mathbf{b}',l,r,k)\} &=& \{a'_{j'} \mid j' \in U_{n}(\mathbf{b}',l,r,k)\} \setminus \{a'_{j'} \mid j' \in U_{n,1}(\mathbf{b}',l,r,k)\} \\
&=& \{\sigma(a_{j}) \mid j \in U_{n}(\mathbf{b},l,r,k)\} \setminus \{\sigma(a_{j}) \mid j \in U_{n,1}(\mathbf{b},l,r,k)\} \\
&=& \{\sigma(a_{j}) \mid j \in U_{n,2}(\mathbf{b},l,r,k)\},
\end{eqnarray*}
where the second equality follows by \eqref{eqn:pf_lr_relations_with_c_0U} and \eqref{eqn:pf_lr_relations_with_c_1}. Hence,
\begin{eqnarray*}
\{r_{j,n}(\mathbf{b}') \mid j \in U_{n,2}(\mathbf{b}',l,r,k)\} &=& \{\iota_{m-1,n}(q'_{j'}) \mid j' \in U_{n,2}(\mathbf{b}',l,r,k)\} = \{\iota_{m,n}(q_{j}) \mid j \in U_{n,2}(\mathbf{b},l,r,k)\} \\
&=& \{r_{j,n}(\mathbf{b}) \mid j \in U_{n,2}(\mathbf{b},l,r,k)\}.
\end{eqnarray*}

\medskip

For iv), we see that
\begin{eqnarray*}
\{a'_{j'} \mid j' \in L_{n,2}(\mathbf{b}',l,r,k)\} &=& \{a'_{j'} \mid j' \in L_{n}(\mathbf{b}',l,r,k)\} = \{\sigma(a_{j}) \mid j \in L_{n}(\mathbf{b},l,r,k)\} \\
&=& \{\sigma(a_{j}) \mid j \in L_{n,2}(\mathbf{b},l,r,k)\},
\end{eqnarray*}
where the first and the last equalities follow from i) and Lemma~\ref{lem:lr_relations_with_c1}(i), respectively, and the second one is by \eqref{eqn:pf_lr_relations_with_c_0L}. Hence,
\begin{eqnarray*}
\{r_{j,n}(\mathbf{b}') \mid j \in L_{n,2}(\mathbf{b}',l,r,k)\} &=& \{\iota_{m-1,n}(q'_{j'}) \mid j' \in L_{n,2}(\mathbf{b}',l,r,k)\} = \{\iota_{m,n}(q_{j})-1 \mid j \in L_{n,2}(\mathbf{b},l,r,k)\} \\
&=& \{r_{j,n}(\mathbf{b})-1 \mid j \in L_{n,2}(\mathbf{b},l,r,k)\}.
\end{eqnarray*}
\end{proof}

\begin{lemma}
\label{lem:removable_arc_beta}
Assume that a Catalan state $C$ of $L(m,n)$ with no bottom returns has a removable arc $c = (y_{a},y'_{b})$, where $a,b \leq m$. Then
\begin{equation*}
\beta(C) = \beta(C \smallsetminus c) + \frac{n+b-a}{2}.
\end{equation*}
\end{lemma}

\begin{proof}
Let $\mathbf{b} = (b_{1},b_{2},\ldots,b_{m})$ and $\mathbf{b}' = (b'_{1},b'_{2},\ldots,b'_{m-1})$ be maximal sequences for $C$ and $C' = C \smallsetminus c$, respectively. Denote by $k$ the index of $c$ relative to $\mathbf{b}$ and let $l = l_{k,n}(\mathbf{b})$ and $r = r_{k,n}(\mathbf{b})$. We show that
\begin{enumerate}
\item[i)] $b'_{1}+b'_{2}+\cdots+b'_{k-1} = (b_{1}+b_{2}+\cdots+b_{k-1})+|L_{n,2}(\mathbf{b},l,r,k)|$, and
\item[ii)] $b'_{j} = b_{j+1}$ for all $k \leq j \leq m-1$.
\end{enumerate}

\medskip

For i), denote by $U_{1}$, $U_{2}$, $L_{1}$, $L_{2}$, $U'_{1}$, $U'_{2}$, $L'_{1}$, and $L'_{2}$ the sets $U_{n,1}(\mathbf{b},l,r,k)$, $U_{n,2}(\mathbf{b},l,r,k)$, $L_{n,1}(\mathbf{b},l,r,k)$, $L_{n,2}(\mathbf{b},l,r,k)$, $U_{n,1}(\mathbf{b}',l,r,k)$, $U_{n,2}(\mathbf{b}',l,r,k)$, $L_{n,1}(\mathbf{b}',l,r,k)$, and $L_{n,2}(\mathbf{b}',l,r,k)$, respectively. By Lemma~\ref{lem:lr_relations},
\begin{equation*}
b_{j} = l_{j,n}(\mathbf{b})+n+(j-1)
\end{equation*}
for any $j \in U_{1} \cup L_{1}$ and
\begin{equation*}
b_{j} = n = n+j-r_{j,n}(\mathbf{b})
\end{equation*}
for any $j \in U_{2} \cup L_{2}$. Therefore,
\begin{equation*}
\sum_{j=1}^{k-1} b_{j} = \sum_{j \in U_{1} \cup L_{1}} (n+j+l_{j,n}(\mathbf{b})-1) + \sum_{j \in U_{2} \cup L_{2}} (n+j-r_{j,n}(\mathbf{b})).
\end{equation*}
Since $L_{1} = \emptyset$ by Lemma~\ref{lem:lr_relations_with_c1}(i), the above equation can be written as 
\begin{equation*}
\sum_{j=1}^{k-1} b_{j} = n(k-1) + \frac{1}{2}(k-1)k + \sum_{j \in U_{1}} (l_{j,n}(\mathbf{b})-1) - \sum_{j \in U_{2}} r_{j,n}(\mathbf{b}) - \sum_{j \in L_{2}} r_{j,n}(\mathbf{b}).
\end{equation*}
Analogously, as $L'_{1} = \emptyset$ by Lemma~\ref{lem:lr_relations_with_c2}(i), we see that
\begin{equation*}
\sum_{j=1}^{k-1} b'_{j} = n(k-1) + \frac{1}{2}(k-1)k + \sum_{j \in U'_{1}} (l_{j,n}(\mathbf{b}')-1) - \sum_{j \in U'_{2}} r_{j,n}(\mathbf{b}') - \sum_{j \in L'_{2}} r_{j,n}(\mathbf{b}').
\end{equation*}
Since by Lemma~\ref{lem:lr_relations_with_c2}(ii)-(iv),
\begin{equation*}
\sum_{j \in U'_{1}} (l_{j,n}(\mathbf{b}')-1) = \sum_{j \in U_{1}} (l_{j,n}(\mathbf{b})-1),
\end{equation*}
\begin{equation*}
\sum_{j \in U'_{2}} r_{j,n}(\mathbf{b}') = \sum_{j \in U_{2}} r_{j,n}(\mathbf{b}),
\end{equation*}
and
\begin{equation*}
\sum_{j \in L'_{2}} r_{j,n}(\mathbf{b}') = \sum_{j \in L_{2}} r_{j,n}(\mathbf{b})-|L_{2}|,
\end{equation*}
it follows that
\begin{equation*}
\sum_{j=1}^{k-1} b'_{j} - \sum_{j=1}^{k-1} b_{j} = |L_{2}|,
\end{equation*}
which proves i).

\medskip

To prove ii), it suffices to show that Catalan states with no bottom returns realized by $\mathbf{b}_{(k+1)} = (b_{k+1},b_{k+2},\ldots,b_{m})$ and $\mathbf{b}'_{(k)} = (b'_{k},b'_{k+1},\ldots,b'_{m-1})$ are same. However, by Lemma~\ref{lem:catalan_state_subseq_b}, this is equivalent to showing that $C_{(k)} = C'_{(k-1)}$. Since $c$ has index $k$ relative to $\mathbf{b}$, we see that 
\begin{equation*}
C_{(k)} = \tau_{k}(C) - \mathcal{A}(\tau_{k}(C),k) = \tau_{k-1}(C') - \mathcal{A}(\tau_{k-1}(C'),k-1) = C'_{(k-1)},
\end{equation*}
which completes the argument for ii).

\medskip

Since $b_{k} = k-a$ and $|L_{2}| = k-\frac{n+a+b}{2}$ by Lemma~\ref{lem:lr_relations_with_c1}(ii), it follows that
\begin{equation*}
\beta(C)-\beta(C \smallsetminus c) = b_{k}-|L_{2}| = \frac{n+b-a}{2}.
\end{equation*}
\end{proof}

\begin{lemma}
\label{lem:remove_an_arc_no_bot_rtn}
Assume that a Catalan state $C \in \mathrm{Cat}(m,n)$ with no bottom returns has a removable arc $c = (y_{a},y'_{b})$, where $a,b \leq m$. Then
\begin{equation*}
C(A) = A^{b-a} \, C'(A),
\end{equation*}
where $C' = C \smallsetminus c$.
\end{lemma}

\begin{proof}
By Lemma~\ref{lem:removable_arc_tree}, $C$ is realizable if and only if $C'$ is realizable. If $C$ is not realizable then the statement is true since 
\begin{equation*}
C(A) = 0 = A^{b-a} \, C'(A).    
\end{equation*}
If $C$ is realizable, let $\mathbf{b} = (b_{1},b_{2},\ldots,b_{m})$ and $\mathbf{b}' = (b'_{1},b'_{2},\ldots,b'_{m-1})$ be maximal sequences for $C$ and $C'$, and denote by $(T(C),v_{0},\alpha)$ and $(T(C'),v'_{0},\alpha')$ their plane rooted trees. Since by Lemma~\ref{lem:removable_arc_beta}
\begin{equation*}
2\beta(C)-mn = 2\beta(C')-(m-1)n+b-a,
\end{equation*}
it follows from Theorem~\ref{thm:coef_no_bot_rtn} and Lemma~\ref{lem:removable_arc_tree} that
\begin{equation*}
C(A) = A^{2\beta(C)-mn} \, Q_{A^{-4}}^{*}(T(C),v_{0},\alpha) = A^{2\beta(C')-(m-1)n+b-a} \, Q_{A^{-4}}^{*}(T(C'),v'_{0},\alpha') = A^{b-a} \, C'(A).
\end{equation*}
\end{proof}

In \cite{DW2022}, a partial order $\preceq$ on $\mathrm{Fin}(\mathbb{N})$ was defined as follows: For $I,J \in \mathrm{Fin}(\mathbb{N})$,
\begin{equation*}
J \preceq I \quad \text{if} \ (\varphi_{n_{I}}^{-1}(I))_{J} \neq K_{0}.
\end{equation*}

\begin{lemma}
\label{lem:removable_arc_phi}
Assume that a Catalan state $C \in \mathrm{Cat}(m,n)$ has a removable arc $c$ with none of its ends on the top boundary. Let $C = R * B$ and $C \smallsetminus c = R' * B'$, where $R,R'$ are roof states and $B,B'$ are bottom states. Then either $\varphi_{n}(B') = \varphi_{n}(B)$ or $\varphi_{n}(B') \prec \varphi_{n}(B)$, and in the latter case:
\begin{itemize}
\item[i)] For any $I \in \mathrm{Fin}(\mathbb{N})$, if $I \prec \varphi_{n}(B)$ then $I \preceq \varphi_{n}(B')$, and
\item[ii)] $\#(C \cap l^{h}_{m-1}) = n+2-2|\varphi_{n}(B)|$.
\end{itemize}
\end{lemma}

\begin{proof}
If $c$ has no end on the bottom boundary of $C$ then $B' = B$ and consequently 
\begin{equation*}
\varphi_{n}(B') = \varphi_{n}(B).    
\end{equation*}
Let $A_{1}$ and $A_{2}$ be the regions determined by $c$ (see Definition~\ref{def:removable_arc}) and assume that $c$ has exactly one of its ends on the bottom boundary. Since $c$ is removable, $C$ has neither innermost bottom corners nor bottom returns which are not parallel to $c$ in $A_{1}$ by Remark~\ref{rem:removable_arc}. Moreover, no bottom returns can be parallel to $c$ in $A_{1}$ and, as one may check, after $c$ is removed no new bottom return is created and indices of left ends of bottom returns in $A_{2}$ do not change. Therefore,
\begin{equation*}
\varphi_{n}(B') = \varphi_{n}(B).
\end{equation*}

If $c$ is a bottom return, we consider cases a) $C$ has an arc $c'$ with ends $y_{m},y'_{m}$ or $y_{m},x'_{k}$ or $y'_{m},x'_{k}$, or b) $C$ has no such an arc. In the case a), we notice that any bottom return of $C$ in $A_{1}$ must be parallel to $c$ by Remark~\ref{rem:removable_arc}. Moreover, as it can easily be seen, $c'$ is parallel to $c$ in $C$ and in $C' = C \smallsetminus c$ the arc corresponding to $c'$ is a bottom return with left end $x'_{i}$, where $i = \min \varphi_{n}(B)$. Furthermore, after $c$ is removed from $C$, indices of left ends of bottom returns in $A_{2}$ do not change in $C'$ while indices of left ends of bottom returns in $A_{1}$ increase by $1$ in $C'$. Therefore,
\begin{equation*}
\varphi_{n}(B') = \varphi_{n}(B). 
\end{equation*}

In the case b), as we argued above, all bottom returns of $C$ which are in $A_{1}$ must be parallel to $c$. Moreover, after $c$ is removed from $C$, we see that no new bottom return will appear in $C'$, indices of left ends of bottom returns of $C$ in $A_{2}$ remain the same in $C'$, and indices of left ends of bottom returns in $A_{1}$ increase by $1$ in $C'$. Hence, 
\begin{equation*}
\varphi_{n}(B') = \varphi_{n}(B) \setminus \{i\},
\end{equation*}
where $i = \min \varphi_{n}(B)$. Thus, by Proposition~3.2(i) of \cite{DW2022},
\begin{equation*}
\varphi_{n}(B) = \{i\} \cup \varphi_{n}(B') = \{i\} \oplus \varphi_{n}(B'),
\end{equation*}
and consequently, 
\begin{equation*}
\varphi_{n}(B') \prec \varphi_{n}(B).    
\end{equation*}

For i), let $c''$ be the bottom return with left end $x'_{i}$, where $i = \min \varphi_{n}(B)$. We see that bottom returns of $B$ other than $c''$ are in the region enclosed by $c''$ and the bottom boundary of $B$. Therefore, if $I \prec \varphi_{n}(B)$ then $i \notin I$. Let $J = \varphi_{n-2|I|}(B_{I})$, then $\varphi_{n}(B) = J \oplus I$ and clearly $i = \min J$, so
\begin{equation*}
J = \{i\} \cup J' = \{i\} \oplus J',
\end{equation*}
where $J' = J \setminus \{i\}$. Since all elements of $J' \oplus I$ are larger than $i$,
\begin{equation*}
\{i\} \cup \varphi_{n}(B') = \varphi_{n}(B) = J \oplus I = \{i\} \oplus J' \oplus I = \{i\} \cup (J' \oplus I).
\end{equation*}
Therefore, $\varphi_{n}(B') = J' \oplus I$ and consequently 
\begin{equation*}
I \preceq \varphi_{n}(B').    
\end{equation*}

For ii), since there are $n+2$ points below line $l^{h}_{m-1}$ and there are no other connections between those points except for $2|\varphi_{n}(B)|$ of them (which are the ends of bottom returns of $C$), it follows that
\begin{equation*}
\#(C \cap l^{h}_{m-1}) = n+2-2|\varphi_{n}(B)|.    
\end{equation*}
\end{proof}

Now we are ready to prove the main result in this section.

\begin{proof}[Proof of Theorem~\ref{thm:remove_an_arc}]
We may assume that $c$ is a removable arc with none of its ends on the top boundary since otherwise we consider a $\pi$-rotation $C^{*}$ of $C$ and its removable arc $c^{*}$ (the image of $c$) instead. In particular, if $y_{a^{*}}$ and $y'_{b^{*}}$ are ends of $c^{*}$ then 
\begin{equation*}
a^{*} = m+1-b \quad \text{and} \quad  b^{*} = m+1-a,   
\end{equation*}
so $b^{*}-a^{*} = b-a$. Therefore, if 
\begin{equation*}
C^{*}(A) = A^{b^{*}-a^{*}} (C^{*} \smallsetminus c^{*})(A),
\end{equation*}
then
\begin{equation*}
C(A) = C^{*}(A) = A^{b^{*}-a^{*}} \, (C^{*} \smallsetminus c^{*})(A) = A^{b-a} \, (C^{*} \smallsetminus c^{*})^{*}(A) = A^{b-a} \, C'(A).
\end{equation*}

Let $C = R * F$, where $R$ is a top state and $F$ is a floor state. Since $c$ has no ends on the top boundary, $c$ is an arc of $F$. Therefore, $F' = F \smallsetminus c$ is defined and $C' = R * F'$. By Theorem~\ref{thm:main}, there is a $\Theta_{A}$-state expansion for $(R,\emptyset)$
\begin{equation*}
\Theta_{A}(R,\emptyset;\cdot) = \sum_{(R',I') \in \mathcal{P}'} Q_{R',I'}(A) \, \Theta_{A}(R',I';\cdot),
\end{equation*}
so to prove that
\begin{equation*}
\Theta_{A}(R,\emptyset;F) = A^{b-a} \, \Theta_{A}(R,\emptyset;F')    
\end{equation*}
it suffices to show that 
\begin{equation}
\label{eqn:pf_remove_an_arc}
\Theta_{A}(R',I';F) = A^{b-a} \, \Theta_{A}(R',I';F')
\end{equation}
for all $(R',I') \in \mathcal{P}'$.

Let $F = M * B$ and $F' = M' * B'$, where $M$, $M'$ are middle states and $B$, $B'$ are bottom states. By Lemma~\ref{lem:removable_arc_phi}, for $B$ and $B'$ either 
\begin{enumerate}
\item[a)] $\varphi_{n}(B') \prec \varphi_{n}(B)$ or 
\item[b)] $\varphi_{n}(B') = \varphi_{n}(B)$.   
\end{enumerate}

In the case a), as in the proof of Lemma~\ref{lem:removable_arc_phi}, one argues that $c$ is a bottom return of $C$ and there is no arcs of $C$ joining $y_{m},y'_{m}$ or $y_{m},x'_{k}$ or $y'_{m},x'_{k}$. Consider the following cases
\begin{enumerate}
\item[i)] $I' \npreceq \varphi_{n}(B)$, 
\item[ii)] $I' \preceq \varphi_{n}(B)$ and $I' \npreceq \varphi_{n}(B')$, and 
\item[iii)] $I' \preceq \varphi_{n}(B')$.
\end{enumerate}

For i), since clearly $I' \npreceq \varphi_{n}(B')$,
\begin{equation*}
R' * F_{I'} = K_{0} = R' * (F')_{I'}. 
\end{equation*}
Thus, both sides of \eqref{eqn:pf_remove_an_arc} are zero. 

For ii), by Lemma~\ref{lem:removable_arc_phi}, we see that $I' = \varphi_{n}(B)$ and
\begin{equation*}
\#(F_{I'} \cap l^{h}_{m-1}) = \#(C \cap l^{h}_{m-1}) =  n+2-2|I'| > n-2|I'| = n_{t}(R').
\end{equation*}
Hence, by Theorem~\ref{thm:vh_line_condi}, $R' * F_{I'}$ is not a realizable Catalan state and consequently both sides of \eqref{eqn:pf_remove_an_arc} are zero. 

For iii), we see that both $C_{1} = R' * F_{I'}$ and $C_{2} = R' * (F')_{I'}$ are Catalan states with no top returns. Let $j$ be the index of the left end of $c$ and let $c'$ be the bottom return of $C$ whose left end has index $i = \min \varphi_{n}(B) \leq j$. As we argued in the proof of Lemma~\ref{lem:removable_arc_phi}, $c'$ is parallel to $c$ and hence $B' = B \smallsetminus c = B \smallsetminus c'$. Thus, 
\begin{equation*}
\varphi_{n}(B') = \varphi_{n}(B) \setminus \{i\},    
\end{equation*}
so, in particular, 
\begin{equation*}
\min \varphi_{n}(B') > i,    
\end{equation*}
consequently $i \notin I'$. Therefore, $F_{I'} = M*B_{I'}$ has a bottom return $c_{1}$ corresponding to $c'$ in $F = M*B$.

We show that $c_{1}$ is a removable arc of $C_{1} = R'*F_{I'}$. Indeed, since the removable arc $c$ is a bottom return, by Remark~\ref{rem:removable_arc}, $C$ has neither innermost bottom corners nor bottom returns in $A'_{1}$, where $A'_{1}$ is the region determined by $c'$ which touches the top boundary of $C$. It follows that $C_{1}$ has neither innermost bottom corners nor bottom returns in $A''_{1}$, where $A''_{1}$ is the region determined by $c_{1}$ which touches the top boundary of $C_{1}$. Consequently, as one may check, $c_{1}$ is a removable arc of $C_{1}$.

Clearly, $C_{1} \smallsetminus c_{1} = C_{2}$ and $C_{1}^{*} \smallsetminus c_{1}^{*} = C_{2}^{*}$, where $c_{1}^{*}$ is the arc of $C_{1}^{*}$ corresponding to $c_{1}$ in $C_{1}$. Let $y_{a_{1}^{*}}$ and $y'_{b_{1}^{*}}$ be the left and right ends of $c_{1}^{*}$, respectively. Since $c_{1}^{*}$ is a removable arc of $C_{1}^{*}$, by Lemma~\ref{lem:remove_an_arc_no_bot_rtn},
\begin{equation*}
C_{1}^{*}(A) = A^{b_{1}^{*}-a_{1}^{*}} \, (C_{1}^{*} \smallsetminus c_{1}^{*})(A) = A^{b_{1}^{*}-a_{1}^{*}} \, C_{2}^{*}(A)
\end{equation*}
and consequently,
\begin{equation*}
\Theta_{A}(R',I';F) = C_{1}(A) = A^{b_{1}^{*}-a_{1}^{*}} \, C_{2}(A) = A^{b_{1}^{*}-a_{1}^{*}} \, \Theta_{A}(R',I';F').
\end{equation*}
Therefore, to prove \eqref{eqn:pf_remove_an_arc}, it suffices to show that $b_{1}^{*}-a_{1}^{*} = b-a$. Indeed, let $x'_{j'}$ be the right end of $c$ and let $x'_{i'}$ be the right end of $c'$, then clearly 
\begin{equation*}
a = \mathrm{ht}(M)+j, \quad b = \mathrm{ht}(M)+n+1-j',
\end{equation*}
and $i'-j' = j-i$. Let $y_{a_{1}}$ and $y'_{b_{1}}$ be the left and right ends of $c_{1}$, respectively. As one may show,
\begin{equation*}
a_{1} = \mathrm{ht}(R')+\mathrm{ht}(M)+i \quad \text{and} \quad b_{1} = \mathrm{ht}(R')+\mathrm{ht}(M)+(n-2|I'|)+1-(i'-2|I'|).
\end{equation*}
Moreover, since 
\begin{equation*}
a_{1}^{*} = \mathrm{ht}(C_{1})+1-b_{1} \quad \text{and} \quad b_{1}^{*} = \mathrm{ht}(C_{1})+1-a_{1}, 
\end{equation*}
it follows that $b_{1}^{*}-a_{1}^{*} = b_{1}-a_{1} = b-a$, which finishes our proof for iii).

In the case b), if $I' \npreceq \varphi_{n}(B)$ then \eqref{eqn:pf_remove_an_arc} holds since its both sides are clearly zero. If $I' \preceq \varphi_{n}(B)$ then $C_{1} = R'*F_{I'}$ and $C_{2} = R'*(F')_{I'}$ are Catalan states with no top returns. As we argued in the proof of Lemma~\ref{lem:removable_arc_phi}, either $c$ is not a bottom return or $c$ is a bottom return in which case there is an arc $c'$ (parallel to $c$) joining $y_{m},y'_{m}$ or $y_{m},x'_{k}$ or $y'_{m},x'_{k}$. In the former case, $c_{1} = c$ is a removable arc of $C_{1}$ and in the latter case, $c_{1} = c'$ is a removable arc of $C_{1}$. Moreover, as one may check, $C_{1} \smallsetminus c_{1} = C_{2}$ in both cases. Therefore, \eqref{eqn:pf_remove_an_arc} follows by using arguments analogous to the ones in iii) since one can apply Lemma~\ref{lem:remove_an_arc_no_bot_rtn} to $C_{1}^{*}$ and its removable arc $c_{1}^{*}$.

Finally, by Theorem~\ref{thm:coef_nonzero}, it is clear that $C$ is realizable if and only if $C'$ is realizable.
\end{proof}

As an application of Theorem~\ref{thm:remove_an_arc}, we give closed-form formulas for coefficients of Catalan states of $L(m,3)$. This result was first mentioned (without a complete proof) in Proposition~5.4 of \cite{DP2019}. A Catalan state $C \in \mathrm{Cat}(m,n)$ is called \emph{vertically decomposable} if
\begin{equation*}
\#(C \cap l^{h}_{i}) = n
\end{equation*}
for some $0 \leq i \leq m$, and it is called \emph{vertically indecomposable} if no such $i$ exists.

\begin{corollary}
\label{cor:coefficients_n3}
Let $C \in \mathrm{Cat}(m,3)$ be a realizable Catalan state. Then 
\begin{equation}
\label{eqn:coefficients_n3_decom}
C(A) = A^{a} (A^{-2}+A^{2})^{b} (A^{-4}+1+A^{4})^{c}
\end{equation}
for some $a \in \mathbb{Z}$ and $b,c \in \mathbb{N} \cup \{0\}$ if $C$ is vertically decomposable, and
\begin{equation}
\label{eqn:coefficients_n3_indecom}
C(A) = A^{a} (A^{-2}+A^{2})^{b} \, \frac{(A^{-2}+A^{2})^{2c}-1}{A^{-4}+1+A^{4}}
\end{equation}
for some $a \in \mathbb{Z}$, $b \in \{0,1\}$, and $c \in \mathbb{N}$ if $C$ is vertically indecomposable.
\end{corollary}

\begin{proof}
If $m = 0$ then $C = L(0,3)$ is vertically decomposable and $C(A) = 1$ which agrees with \eqref{eqn:coefficients_n3_decom} when $a = b = c = 0$. If $m = 1$, there are eight realizable Catalan states and, as one may check, $C(A) = A^{\pm 3}$ or $A^{\pm 1}$ if $C$ is vertically decomposable and $C(A) = A^{\pm 1}$ if $C$ is vertically indecomposable. As it can easily be seen, with appropriate choices made for $a$, $b$, and $c$, either $C(A)$ is given by \eqref{eqn:coefficients_n3_decom} or \eqref{eqn:coefficients_n3_indecom}. Hence, we assume that $m \geq 2$. 

For a vertically decomposable Catalan state $C$ of $L(m,3)$, we find all integers $1 \leq i_{1} < i_{2} < \cdots < i_{k-1} \leq m-1$ such that, $\#(C \cap l^{h}_{i_{j}}) = 3$. Let $i_{0} = 0$ and $i_{k} = m$ and let $m_{j} = i_{j}-i_{j-1}$ for $j = 1,2,\ldots,k$. Then we can express $C$ as
\begin{equation}
\label{eqn:irred_decom}
C = C_{1} * C_{2} * \cdots * C_{k},    
\end{equation}
where each $C_{j}$ is a vertically decomposable Catalan state of $L(m_{j},3)$ satisfying $\#(C_{j} \cap l^{h}_{i}) < 3$ for $i = 1,2,\ldots,m_{j}-1$. By Theorem~5.1 of \cite{DP2019},
\begin{equation*}
C(A) = C_{1}(A) \cdot C_{2}(A) \cdot \cdots \cdot C_{k}(A).
\end{equation*}
Hence, to prove that $C(A)$ is given by \eqref{eqn:coefficients_n3_decom}, it suffices to show that each $C_{j}(A)$ is given by \eqref{eqn:coefficients_n3_decom}.

Consider a Catalan state $C$ of $L(m,3)$ such that $\#(C \cap l^{h}_{i}) < 3$ for $i = 1,2,\ldots,m-1$. For any boundary point $x$ of $C$, let $y = \pi(x)$ be the unique boundary point of $C$ connected to $x$ by an arc of $C$. Then
\begin{equation*}
\pi(y_{i}) \in \{y_{i-1},y_{i+1},y'_{i-1},y'_{i+1}\} \end{equation*}
for all $1 \leq i \leq m$, where $y_{0} = x_{1}$, $y_{m+1} = x'_{1}$, $y'_{0} = x_{3}$, and $y'_{m+1} = x'_{3}$. Notice that, an arc $c$ of $C$ with ends $y_{i-1},y'_{i}$ or $y_{i},y'_{i-1}$ for some $1 \leq i \leq m+1$ is removable. Therefore, by Theorem~\ref{thm:remove_an_arc}
\begin{equation*}
C(A) = A^{\pm(i-(i-1))} \, C'(A) = A^{\pm 1} \, C'(A),
\end{equation*}
where $C' = C \smallsetminus c$. Hence, we only need to consider cases of Catalan states $C$ which have no arcs with ends $y_{i-1},y'_{i}$ or $y_{i},y'_{i-1}$, $1 \leq i \leq m+1$. Moreover,
\begin{equation*}
\begin{array}{lll}
\pi(x_{1}) \in \{y_{1},x_{2}\}, & \pi(x_{2}) \in \{x_{1},x_{3},y_{2},y'_{2}\}, & \pi(x_{3}) \in \{x_{2},y'_{1}\}, \\
\pi(x'_{1}) \in \{y_{m},x'_{2}\}, & \pi(x'_{2}) \in \{x'_{1},x'_{3},y_{m-1},y'_{m-1}\}, & \pi(x'_{3}) \in \{x'_{2},y'_{m}\},
\end{array}
\end{equation*}
and, as it could easily be seen, each choice of $\pi(x_{2})$ and $\pi(x'_{2})$ completely determines $C$.

\begin{figure}[htb]
\centering
\includegraphics[scale=1]{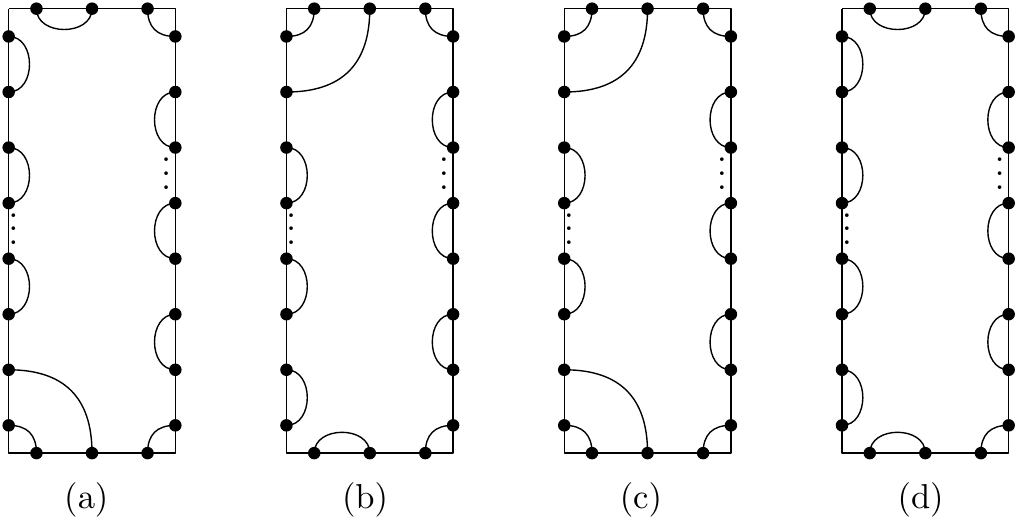}
\caption{Cases $(\pi(x_{2}),\pi(x'_{2})) = (x_{1},y_{m-1}),(y_{2},x'_{1}),(y_{2},y_{m-1}),(x_{1},x'_{1})$ when $m$ is even}
\label{fig:coefficients_n3}
\end{figure}

There are sixteen choices for pairs of points $\pi(x_{2})$ and $\pi(x'_{2})$ but only half of them are possible for a given $m \geq 2$. When $m$ is even, six out of the eight cases are vertically decomposable (three of them are shown in Figure~\ref{fig:coefficients_n3}(a)-(c) and the remaining three are obtained by reflecting them about a vertical line) and the remaining two cases are vertically indecomposable (one of them is shown in Figure~\ref{fig:coefficients_n3}(d) and the other one is obtained by reflecting it about a vertical line). Standard calculations (see Proposition~5.2 in \cite{DP2019}) show that $C(A)$ has a form given by \eqref{eqn:coefficients_n3_decom} or \eqref{eqn:coefficients_n3_indecom} depending on whether $C$ is vertically decomposable or indecomposable. Analogous argument applies when $m$ is odd.
\end{proof}

\begin{remark}
\label{rem:coefficients_n3}
The proof of Proposition~5.4 given in \cite{DP2019} was based on an observation (which authors promised to address in a follow-up paper) that coefficients of all ``irreducible'' factors in \eqref{eqn:irred_decom} can be determined (up to a power of $A$) using Proposition~5.2 of \cite{DP2019}. This observation is justified in our proof of Corollary~\ref{cor:coefficients_n3} by applying Theorem~\ref{thm:remove_an_arc}.
One can also show that for any $a \in \mathbb{Z}$ and $b,c \in \mathbb{N} \cup \{0\}$, there is a vertically decomposable Catalan state $C$ of $L(m,3)$ for some $m$, such that $C(A)$ is given by \eqref{eqn:coefficients_n3_decom}. Analogous result holds for \eqref{eqn:coefficients_n3_indecom}.
\end{remark}

\section{Vertical Factorization Theorem}
\label{s:vertical_factor_thm}

In this section, for a Catalan state $C$ admitting a family $\Lambda$ of arcs which ``vertically factorizes'' $C$ (see Definition~\ref{def:vertical_factor}), we use $\Theta_{A}$-state expansion to show that $C(A)$ factors into a product of $C_{T(\Lambda)}(A)$ and $C_{\Lambda}(A)$ (see Theorem~\ref{thm:vertical_factor_thm}). The Catalan states $C_{T(\Lambda)}$ and $C_{\Lambda}(A)$ in this factorization depend only on $\Lambda$ and $C \setminus \Lambda$, respectively, and their coefficients are easier to compute.

For a Catalan state $C \in \mathrm{Cat}(m,n)$, a non-empty subset $\Lambda$ of the set of its arcs is called a \emph{local family of arcs} if the set $E_{\Lambda}$ of ends of arcs in $\Lambda$ consists of consecutive boundary points of $C$ (see example in Figure~\ref{fig:C_Lambda}(a)). Given a local family of arcs $\Lambda$ of $C$, let $C_{\Lambda}$ be the Catalan state obtained from $C$ by replacing arcs in $\Lambda$ by a family of parallel arcs whose set of ends is $E_{\Lambda}$ (see Figure~\ref{fig:C_Lambda}(b)). Denote by $C \setminus \Lambda$ the complement of $\Lambda$ in set of arcs of $C$. For $C - (C \setminus \Lambda)$ obtained from $C$ by deleting arcs in $C \setminus \Lambda$, let $(T(\Lambda),v_{0})$ be its dual plane rooted tree (see Figure~\ref{fig:C_Lambda}(c)). Denote by $\lambda$ the number of elements of $\Lambda$ and let $C_{T(\Lambda)} \in \mathrm{Cat}(\lambda,2\lambda)$ be the unique Catalan state satisfying $T(C_{T(\Lambda)}) = T(\Lambda)$ and all of its bottom-boundary points are connected to either the left- or right-boundary points (see Figure~\ref{fig:C_Lambda}(d)).

\begin{figure}[htb]
\centering
\includegraphics[scale=1]{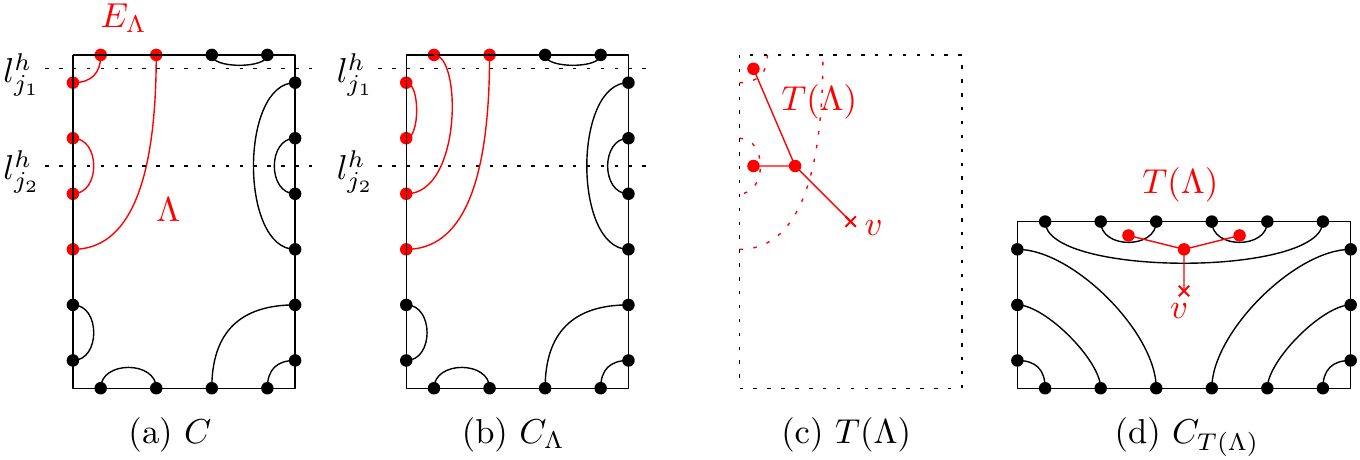}
\caption{Catalan states $C$, $C_{\Lambda}$, plane rooted tree $T(\Lambda)$ and Catalan state $C_{T(\Lambda)}$}
\label{fig:C_Lambda}
\end{figure}

\begin{definition}
\label{def:vertical_factor}
We say that a local family of arcs $\Lambda$ of $C \in \mathrm{Cat}(m,n)$ \emph{vertically factorizes} $C$ if
\begin{itemize}
\item[i)] either $E_{\Lambda} \cap \{y_{1},y_{2},\ldots,y_{m}\} = \emptyset$ or $E_{\Lambda} \cap \{y'_{1},y'_{2},\ldots,y'_{m}\} = \emptyset$, and
\item[ii)] there are integers $j_{1},j_{2}$ such that, $0 \leq j_{1} \leq j_{2} \leq m$ and for any arc $c$ of $C$ with ends $y_{j},y_{j+1}$ or $y'_{j},y'_{j+1}$, if $c \in \Lambda$ then $j_{1} \leq j \leq j_{2}$, otherwise $j \leq j_{1}$ or $j \geq j_{2}$.
\end{itemize}
\end{definition}

Condition ii) in Definition~\ref{def:vertical_factor} has a simple geometric interpretation, i.e., there are horizontal lines $l^{h}_{j_{1}}$ and $l^{h}_{j_{2}}$ with $j_{1} \leq j_{2}$ such that, the region between them does not include any returns in $C \setminus \Lambda$ while the regions above or below both lines do not include any returns in $\Lambda$ (see Figure~\ref{def:vertical_factor}(a)).

Using notations and terminology from above, we formulate the main result for this section.

\begin{theorem}[Vertical Factorization Theorem]
\label{thm:vertical_factor_thm}
Let $\Lambda$ be a local family of arcs of a Catalan state $C$ which vertically factorizes $C$. Then
\begin{equation}
\label{eqn:vertical_factor_thm}
C(A) = C_{T(\Lambda)}(A) \cdot C_{\Lambda}(A).
\end{equation}
In particular, $C$ is realizable if and only if $C_{\Lambda}$ is realizable.
\end{theorem}

Our proof of Theorem~\ref{thm:vertical_factor_thm} is based on Lemma~\ref{lem:plucking_poly_C_Lambda}, Lemma~\ref{lem:beta_relations_C_Lambda}, and Lemma~\ref{lem:theta_state_expn_new}. 

\begin{lemma}
\label{lem:plucking_poly_C_Lambda}
Let $\Lambda$ be a local family of arcs of $C \in \mathrm{Cat}(m,n)$, where $C$ has neither right nor bottom returns, such that $E_{\Lambda} \subseteq \{y_{-n+1},y_{-n+2},\ldots,y_{m}\}$. If $\Lambda$ vertically factorizes $C$, then
\begin{equation}
\label{eqn:lem_plucking_poly_C_Lambda}
Q_{q}^{*}(T(C),v_{0},\alpha) = Q_{q}^{*}(T(\Lambda),v,\alpha_{0}) \cdot Q_{q}^{*}(T(C_{\Lambda}),v_{0},\alpha_{\Lambda}),
\end{equation}
where $(T(C),v_{0},\alpha)$ and $(T(C_{\Lambda}),v_{0},\alpha_{\Lambda})$ are the plane rooted trees with delay corresponding to $C$ and $C_{\Lambda}$, respectively, and $(T(\Lambda),v)$ is the dual plane rooted tree for $C - (C \setminus \Lambda)$ with delay $\alpha_{0} \equiv 1$. In particular, $C$ is realizable if and only if $C_{\Lambda}$ is realizable.
\end{lemma}

\begin{proof}
Denote by $C'$ and $C''$ the Catalan states obtained from $C$ and $C_{\Lambda}$ by a clockwise $\frac{\pi}{2}$-rotation, respectively. Let $(T(C'),v',\alpha')$ and $(T(C''),v'',\alpha'')$ be the plane rooted trees with delay corresponding to $C'$ and $C''$, respectively. By our assumption on $\Lambda$, the tree $T(\Lambda)$ is a subtree of $T(C')$ and $\alpha'(u) = 1$ for all leaves $u$ of $T(\Lambda)$. Since $(T(\Lambda),v,\alpha_{0})$ is a splitting subtree of $(T(C'),v',\alpha')$ and, as one may check, $(T(C''),v'',\alpha'')$ is its complementary tree, it follows by Theorem~\ref{thm:prod_formula_plucking} that
\begin{equation}
\label{eqn:pf_lem_plucking_poly_C_Lambda_1}
Q_{q}(T(C'),v',\alpha') = Q_{q}(T(\Lambda),v,\alpha_{0}) \cdot Q_{q}(T(C''),v'',\alpha'').
\end{equation}
Moreover, as $C'$ is a clockwise $\frac{\pi}{2}$-rotation of $C$,
\begin{equation*}
C'(A) = C(A^{-1})    
\end{equation*}
and consequently $C'$ is realizable if and only if $C$ is realizable by Theorem~\ref{thm:coef_nonzero}. If both $C'$ and $C$ are non-realizable then, using Proposition~3.14 of \cite{DW2022}, we see that
\begin{equation*}
Q_{q}(T(C'),v',\alpha') = 0 = Q_{q}(T(C),v_{0},\alpha),
\end{equation*}
so consequently
\begin{equation}
\label{eqn:pf_lem_plucking_poly_C_Lambda_2}
Q_{q}^{*}(T(C'),v',\alpha') = Q_{q^{-1}}^{*}(T(C),v_{0},\alpha),
\end{equation}
where $Q_{q^{-1}}^{*}$ stands for $(Q_{q^{-1}})^{*}$. If both $C'$ and $C$ are realizable, since
\begin{equation*}
\min\deg_{q} Q_{q}^{*}(T(C'),v',\alpha') = 0 = \min\deg_{q} Q_{q^{-1}}^{*}(T(C),v_{0},\alpha),
\end{equation*} 
it follows from Theorem~\ref{thm:coef_no_bot_rtn} that \eqref{eqn:pf_lem_plucking_poly_C_Lambda_2} also holds. Analogous arguments used for $C''$ and $C_{\Lambda}$ yield
\begin{equation}
\label{eqn:pf_lem_plucking_poly_C_Lambda_3}
Q_{q}^{*}(T(C''),v'',\alpha'') = Q_{q^{-1}}^{*}(T(C_{\Lambda}),v_{0},\alpha_{\Lambda}).
\end{equation}
Finally, by Corollary~2.3(ii) of \cite{Prz2016}, $Q_{q}(T(\Lambda),v,\alpha_{0})$ is a quotient of products of $q$-integers\footnote{A $q$-integer (or $q$-analog of an integer) $[n]_{q}$ is defined by $[n]_{q} = 1 + q +\cdots + q^{n-1}$ for $n \geq 1$.}, so
\begin{equation}
\label{eqn:pf_lem_plucking_poly_C_Lambda_4}
Q_{q}^{*}(T(\Lambda),v,\alpha_{0}) = Q_{q^{-1}}^{*}(T(\Lambda),v,\alpha_{0}).
\end{equation}
Therefore, equation \eqref{eqn:lem_plucking_poly_C_Lambda} follows from \eqref{eqn:pf_lem_plucking_poly_C_Lambda_1}-\eqref{eqn:pf_lem_plucking_poly_C_Lambda_4}.

Since $Q_{q}^{*}(T(\Lambda),v,\alpha_{0}) \neq 0$, it follows that $Q_{q}(T(C),v_{0},\alpha) \neq 0$ if and only if $Q_{q}(T(C_{\Lambda}),v_{0},\alpha_{\Lambda}) \neq 0$. Consequently, by Proposition~3.14 of \cite{DW2022}, $C$ is realizable if and only if $C_{\Lambda}$ is realizable.
\end{proof}

\begin{lemma}
\label{lem:beta_relations_C_Lambda}
Let $C \in \mathrm{Cat}(m,n)$ be a realizable Catalan state with no bottom returns and let $\Lambda$ be a local family of arcs of $C$ such that $E_{\Lambda} \subseteq \{y_{-n+1},y_{-n+2},\ldots,y_{m}\}$. If $C_{\Lambda}$ is realizable, then
\begin{equation*}
\beta(C)-\beta(C_{\Lambda}) = \beta(C_{T(\Lambda)}) - \lambda^{2},
\end{equation*}
where $\lambda$ is the number of elements of $\Lambda$.
\end{lemma}

\begin{proof}
Let $\Lambda'$ be the local family of arcs of $C_{\Lambda}$ with $E_{\Lambda'} = E_{\Lambda}$. There is a unique $1-n \leq j \leq m-1$ such that $(y_{j+1},y_{j}) \in \Lambda'$. Let $p_{0} = y_{j+1}$, then by Corollary~\ref{cor:lr_relations},
\begin{equation*}
\beta(C)-\beta(C_{\Lambda}) = \sum_{(p,q) \in \Lambda} \iota_{m,n}(p) - \sum_{(p,q) \in \Lambda'} \iota_{m,n}(p) = \sum_{(p,q) \in \Lambda} \iota_{m,n}(p) - \frac{\lambda}{2}(2\iota_{m,n}(p_{0})-\lambda+1).
\end{equation*}
By Corollary~\ref{cor:lr_relations} and definitions of $C_{T(\Lambda)}$ and $C_{T(\Lambda')}$,
\begin{equation*}
\beta(C_{T(\Lambda)}) - \beta(C_{T(\Lambda')}) = \sum_{(p,q) \in \Lambda''} \iota_{\lambda,2\lambda}(p) - \sum_{(p,q) \in \Lambda'''} \iota_{\lambda,2\lambda}(p) = \sum_{(p,q) \in \Lambda''} \iota_{\lambda,2\lambda}(p) - \frac{\lambda}{2}(2\iota_{\lambda,2\lambda}(p'_{0})-\lambda+1),
\end{equation*}
where $\Lambda''$ and $\Lambda'''$ are the sets of top returns of $C_{T(\Lambda)}$ and $C_{T(\Lambda')}$, respectively, and $p'_{0} = x_{\lambda}$ in $C_{T(\Lambda')}$. Since
\begin{equation*}
\sum_{(p,q) \in \Lambda} (\iota_{m,n}(p)-\iota_{m,n}(p_{0})) = \sum_{(p,q) \in \Lambda''} (\iota_{\lambda,2\lambda}(p)-\iota_{\lambda,2\lambda}(p'_{0})),
\end{equation*}
it follows that
\begin{equation*}
\beta(C)-\beta(C_{\Lambda}) = \beta(C_{T(\Lambda)}) - \beta(C_{T(\Lambda')}) = \beta(C_{T(\Lambda)}) - \lambda^{2}.
\end{equation*}
\end{proof}

For non-negative integers $n,k$, let $\mathcal{R}_{n,k}$ be the family of all roof states $R$ with $n_{t}(R) = n$, $n_{b}(R) = n' \geq k$, and $\mathrm{ht}(R) = m \geq \max\{k-n,0\}$ such that, each pair of boundary points $x'_{n'-i}$ and $y'_{m-i}$ is joined by an arc for $i = 0,1,2,\ldots,k-1$. (See Figure~\ref{fig:R_nk})

\begin{figure}[ht] 
\centering
\includegraphics[scale=1]{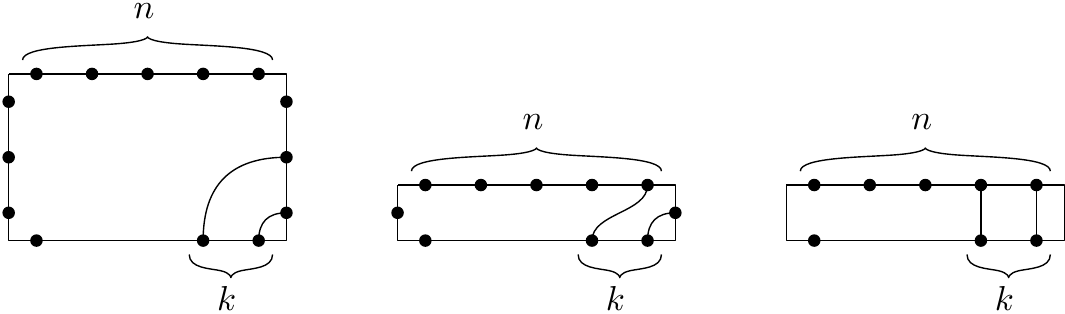}
\caption{Examples of roof states in $\mathcal{R}_{n,k}$}
\label{fig:R_nk}
\end{figure}

In Lemma~\ref{lem:theta_state_expn_new}, we prove that any pair $(R,I)$ with $R \in \mathcal{R}_{n,k}$ and $I \in \mathrm{Fin}(\mathbb{N})$ has a $\Theta_{A}$-state expansions satisfying additional properties. We note that this result is a stronger version of Theorem~\ref{thm:main}. In particular, a $\Theta_{A}$-state expansion resulting from the proof of Theorem~4.7 in \cite{DW2022} may not have such properties.

\begin{figure}[ht] 
\centering
\includegraphics[scale=1]{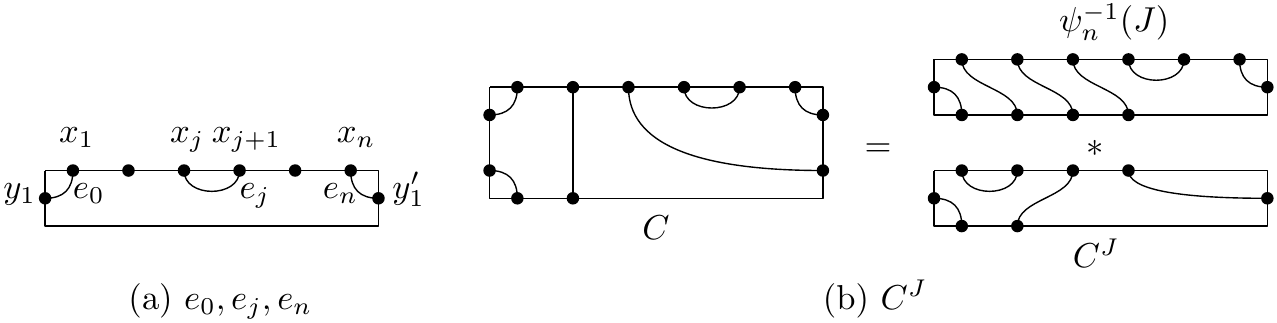}
\caption{Arcs $e_{0},e_{j},e_{n}$, roof state $R$, and $R^{J}$ for $J = \{4,6\}$}
\label{fig:R_J}
\end{figure}

Let $e_{j}$, $j = 0,1,\ldots,n$, be arcs shown in Figure~\ref{fig:R_J}(a) and, as in \cite{DW2022}, for a crossingless connection $C$ we define
\begin{equation*}
\mathcal{J}(C) = \{j \mid e_{j} \in C\}.
\end{equation*}
Let $R$ be a roof state for which there is a bottom state $B$, such that $R*B$ is a Catalan state of $L(1,n)$. Denote by $\mathcal{R}_{n}$ the set of all such roof states $R$. Define a map $\psi_{n}$ from $\mathcal{R}_{n}$ to the set of all subsets of $\{0,1,\ldots,n\}$ by 
\begin{equation*}
\psi_{n}(R) = \mathcal{J}(R).
\end{equation*}
By Lemma~3.8 of \cite{DW2022}, $\psi_{n}$ is a bijection between $\mathcal{R}_{n}$ and the set $\mathcal{U}_{n}$ consisting of all $J = \{j_{1},j_{2},\ldots,j_{t+1}\} \subseteq \{0,1,\ldots,n\}$ with $t \geq 0$, such that $j_{k+1}-j_{k} > 1$ for $k = 1,2,\ldots,t$. Given a crossingless connection $C$ with $n_{t}(C) = n$, $\mathrm{ht}(C) \geq 1$, and $J \in \mathcal{U}_{n}$, define $C^{J} = C'$ if there is a crossingless connection $C'$ such that
\begin{equation*}
C = \psi_{n}^{-1}(J) * C'    
\end{equation*}
and we set $C^{J} = K_{0}$ otherwise (see Figure~\ref{fig:R_J}(b)).

\begin{lemma}
\label{lem:theta_state_expn_new}
For each $(R,I) \in \mathcal{R}_{n,k} \times \mathrm{Fin}(\mathbb{N})$, there is a $\Theta_{A}$-state expansion 
\begin{equation}
\label{eqn:Theta_state_expansion}
\Theta_{A}(R,I;\cdot) = \sum_{(R',I') \in \mathcal{P}'} Q_{R',I'}(A) \, \Theta_{A}(R',I' \oplus I;\cdot)
\end{equation}
such that, each $(R',I') \in \mathcal{P}'$ satisfies the following additional properties
\begin{enumerate}
\item[i)] $R' \in \mathcal{R}_{n+2|I|-2|I'|,k}$,
\item[ii)] $\mathrm{ht}(R') \geq k$, and
\item[iii)] $R'$ has no left returns and no arcs joining left- and bottom-boundary points.
\end{enumerate}
\end{lemma}

\begin{proof}
Our proof is by induction on $n$ and on $k$. For convenience, we will say that $(R,I) \in \mathcal{R}_{n,k} \times \mathrm{Fin}(\mathbb{N})$ has property $\mathrm{P}(n,k)$ if there is a $\Theta_{A}$-state expansion for $(R,I)$ with each pair $(R',I') \in \mathcal{P}'$ having properties i)-iii).

For $n = 0,1$ and any $k \geq 0$, let $R \in \mathcal{R}_{n,k}$ and $I \in \mathrm{Fin}(\mathbb{N})$. If $n_{b}(R) > n$ then by Remark~4.9 of \cite{DW2022}, $\Theta_{A}(R,I;\cdot) \equiv 0$. Taking for $\mathcal{P}' = \emptyset$ in \eqref{eqn:Theta_state_expansion}, we see that $(R,I)$ has property $\mathrm{P}(n,k)$. For $n_{b}(R) \leq n$, since $n+n_{b}(R)$ is even, it must be $n_{b}(R) = n$, i.e., $R$ is a Catalan state. If $R$ is not realizable, then clearly $(R,I)$ has property $\mathrm{P}(n,k)$ as we may take $\mathcal{P'} = \emptyset$ in \eqref{eqn:Theta_state_expansion}. If $R$ is realizable, then $R(A) \neq 0$ by Theorem~\ref{thm:coef_nonzero}, so we may take in \eqref{eqn:Theta_state_expansion}
\begin{equation*}
\mathcal{P}' = \{(C_{0},I)\} \quad \text{and} \quad Q_{C_{0},I}(A) = \frac{R(A)}{C_{0}(A)},
\end{equation*}
where $C_{0} \in \mathrm{Cat}(n,n)$ is the unique Catalan state with coefficient $C_{0}(A) = A^{-n^{2}}$. Hence, $(R,I)$ also has property $\mathrm{P}(n,k)$ in this case.

Let $n_{0} > 1$ and assume that for all $n < n_{0}$ and $k \geq 0$ each pair $(R,I)$ with $R \in \mathcal{R}_{n,k}$ and $I \in \mathrm{Fin}(\mathbb{N})$ has property $\mathrm{P}(n,k)$.
When $n = n_{0}$, $k = 0$, let $R \in \mathcal{R}_{n,0}$. If $n_{b}(R) > n$ then, as we argued above, each pair $(R,I)$ has property $\mathrm{P}(n,0)$. So we assume that $n_{b}(R) \leq n$ and consider the following cases for $R$:
\begin{enumerate}
\item[a)] $R = L(0,n)$,
\item[b)] $\mathrm{ht}(R) = 0$ and $R \neq L(0,n)$, or
\item[c)] $\mathrm{ht}(R) \geq 1$.
\end{enumerate}

For a), one chooses in \eqref{eqn:Theta_state_expansion} \begin{equation*}
\mathcal{P}' = \{(C_{0},I)\} \quad \text{and} \quad Q_{C_{0},I}(A) = \frac{1}{C_{0}(A)},
\end{equation*}
where $C_{0} \in \mathrm{Cat}(n,n)$ is defined as above. Clearly $(R,I)$ has property $\mathrm{P}(n,0)$. 

For b), let $\overline{R}$ be the reflection of $R$ about a vertical line and define 
\begin{eqnarray*}
j(R) &=& \min \mathcal{J}(R),\\ 
u(R) &=& |\{i \in \varphi_{n}(\overline{R}^{*}) \mid i \leq j(R)\}|,\ \text{and}\\ 
v(R) &=& n - \min (\mathcal{J}(R) \cup \{n\} \setminus \{j(R)\}) \geq 0.
\end{eqnarray*}
The geometric interpretation of $j = j(R)$, $u = u(R)$, and $v = v(R)$ is given in Figure~\ref{fig:pf_theta_state_expn_new}. 

\begin{figure}[htb]
\centering
\includegraphics[scale=1]{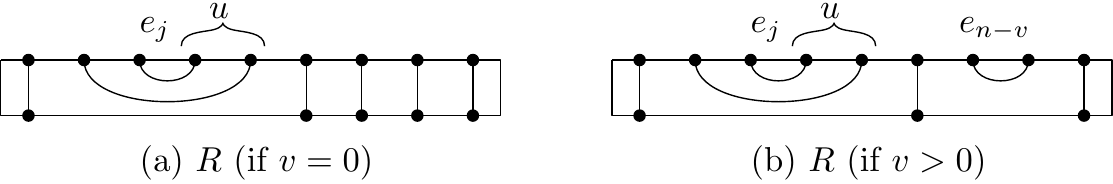}
\caption{$j = j(R)$, $u = u(R)$, and $v = v(R)$ for top state $R \neq L(0,n)$}
\label{fig:pf_theta_state_expn_new}
\end{figure}

We consider the following cases
\begin{enumerate}
\item[1)] $u = j$ and $v = 0$,
\item[2)] $u = j$ and $v > 0$, or
\item[3)] $u < j$ and $v \geq 0$.
\end{enumerate} 

For the case 1), we use Lemma~4.3 of \cite{DW2022} and Proposition~3.7(i) of \cite{DW2022} to see that
\begin{equation*}
\Theta_{A}(R,I;\cdot) = Q(A) \, \Theta_{A}(R_{n,j},I;\cdot) - \sum_{(R',I') \in \tilde{\mathcal{P}}'} Q_{R',I'}(A) \, \Theta_{A}(R',I' \oplus I;\cdot)
\end{equation*}
for some non-zero coefficients $Q(A),Q_{R',I'}(A) \in \mathbb{Q}(A)$, where 
\begin{equation*}
\tilde{\mathcal{P}}' = \{(R'_{n,j,|J|},J) \mid J \in \mathcal{L}_{n}, \ 0 < |J| \leq j\}
\end{equation*}
and roof states $R_{n,j},R'_{n,j,l}$ are shown in Figure~\ref{fig:pf_theta_state_expn_new_2}. Clearly, $(R_{n,j},I)$ satisfies properties i)-iii). Since $R' = R'_{n,j,|I'|} \in \mathcal{R}_{n-2|I'|,0}$ with $|I'| > 0$ in each pair $(R',I') \in \tilde{\mathcal{P}}'$, the induction hypothesis on $n$ applies to each pair $(R',I' \oplus I)$, i.e., $(R',I' \oplus I)$ has property $\mathrm{P}(n-2|I'|,0)$. Consequently, $(R,I)$ with $u = j$ and $v = 0$ has property $\mathrm{P}(n,0)$. 

\begin{figure}[htb]
\centering
\includegraphics[scale=1]{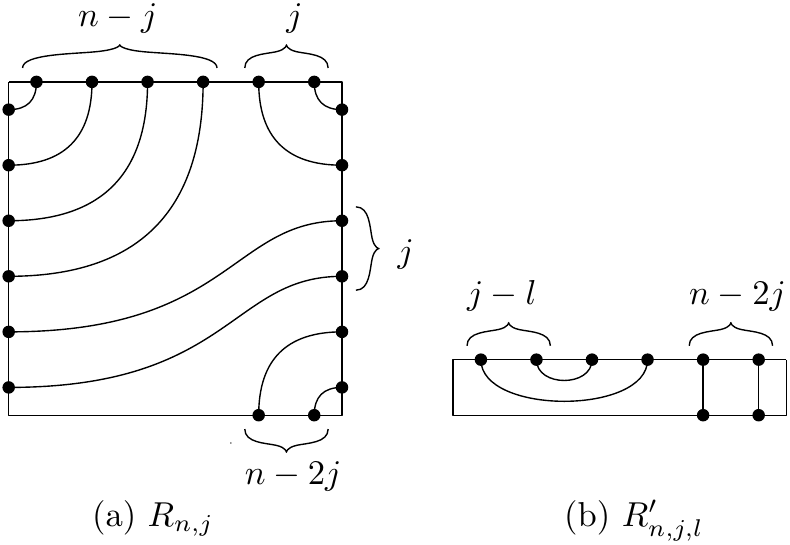}
\caption{Roof states $R_{n,j}$ and $R'_{n,j,l}$}
\label{fig:pf_theta_state_expn_new_2}
\end{figure}

For the case 2), let $\tilde{R} = \psi_{n}(\{j\}) * R$ and then, by Proposition~3.11 of \cite{DW2022},
\begin{equation*}
\Theta_{A}(R,I;\cdot) = A^{n-2j} \, \Theta_{A}(\tilde{R},I;\cdot) - \sum_{(R',I') \in \mathcal{P}'_{0}} Q_{R',I'}(A) \, \Theta_{A}(R',I' \oplus I;\cdot)
\end{equation*}
for some family $\mathcal{P}'_{0} \subset \mathcal{W}$. As one may check, $R'$ in each pair $(R',I') \in \mathcal{P}'_{0}$ is a top state with $R' \in \mathcal{R}_{n-2|I'|,0}$ and either $n_{t}(R') < n$ or $n_{t}(R') = n$ with $u(R') = j(R')$ and $v(R') < v$. Therefore, either induction on $n$ or induction on $v$ (for $R' \in \mathcal{R}_{n,0}$ with $\mathrm{ht}(R') = 0$, $u(R') = j(R')$, and using 1) as the base case for this induction) applies for each pair $(R',I' \oplus I)$, where $(R',I') \in \mathcal{P}'_{0}$. Consequently, each such a pair has property $\mathrm{P}(n-2|I'|,0)$. Moreover, if we express $\tilde{R}$ as a vertical product of a top state $\tilde{R}'$ and a middle state $\tilde{M}$, i.e., $\tilde{R} = \tilde{R}' * \tilde{M}$, then $\tilde{R}' \in \mathcal{R}_{n,0}$ and $\tilde{M} \in \mathcal{R}_{n',0}$, where $n' = n_{b}(\tilde{R}')$. Furthermore, as one may check,
\begin{equation*}
u(\tilde{R}') = j(\tilde{R}') = j \ \text{and} \ v(\tilde{R}') = v-1.
\end{equation*}
Using induction on $v$, the pair $(\tilde{R}',I)$ has property $\mathrm{P}(n,0)$, i.e., $(\tilde{R}',I)$ has a $\Theta_{A}$-state expansion with a family $\mathcal{P}'_{1}$ of pairs $(R',I')$ satisfying properties i)-iii). Consequently, by Proposition~3.7(ii) of \cite{DW2022}, we see that $(\tilde{R},I)$ has a $\Theta_{A}$-state expansion with the family
\begin{equation*}
\tilde{\mathcal{P}}'_{1} = \{(R' * \tilde{M},I') \mid (R',I') \in \mathcal{P}'_{1}\}    
\end{equation*}
for which, as one may verify, each pair $(R' * \tilde{M},I') \in \tilde{\mathcal{P}}'_{1}$ satisfies properties i)-iii), i.e., $(\tilde{R},I)$ has property $\mathrm{P}(n,0)$. Hence, we can conclude that $(R,I)$ with $u = j$ and $v > 0$ has property $\mathrm{P}(n,0)$.

For the case 3), let $\tilde{R}_{0} = R$ and
\begin{equation*}
\tilde{R}_{i+1} = \psi_{n}^{-1}(\{0\}) * \tilde{R}_{i}
\end{equation*}
for $i = 0,1,\ldots,n-1$. By Proposition~3.11 of \cite{DW2022}, there are corresponding families $\mathcal{P}'_{i} \subset \mathcal{W}$, $i = 0,1,\ldots,n-1$, such that
\begin{equation}
\label{eqn:pf_theta_state_expn_new}
\Theta_{A}(\tilde{R}_{i},I;\cdot) = A^{n} \, \Theta_{A}(\tilde{R}_{i+1},I;\cdot) - \sum_{(R',I') \in \mathcal{P}'_{i}} Q_{R',I'}(A) \, \Theta_{A}(R',I' \oplus I;\cdot)
\end{equation}
with $R' \in \mathcal{R}_{n-2|I'|,0}$ and $\mathrm{ht}(R') = i$ for each $R'$ in the corresponding pair $(R',I') \in \mathcal{P}'_{i}$. If $|I'| = 0$ in such a pair $(R',I')$ then, after representing $R' = \tilde{R}' * \tilde{M}$ as a vertical product of a top state $\tilde{R}'$ and a middle state $\tilde{M}$, we see that $\tilde{R}' \in \mathcal{R}_{n,0}$ and $\tilde{M} \in \mathcal{R}_{n',0}$, where $n' = n_{b}(\tilde{R}')$. Furthermore, as one may verify,
\begin{equation*}
u(\tilde{R}') = j(\tilde{R}') = 1 \ \text{and} \ v(\tilde{R}') \geq 0.
\end{equation*}
Hence, using either 1) or 2) for $\tilde{R}'$ and, similar arguments as the above for $\tilde{R}' * \tilde{M}$, one shows that each $(R',I' \oplus I)$ with $(R',I') \in \mathcal{P}'_{i}$ and $|I'| = 0$ has property $\mathrm{P}(n-2|I'|,0)$. If $|I'| > 0$ in a pair $(R',I') \in \mathcal{P}'_{i}$ then, using induction on $n$, we see that $(R',I' \oplus I)$ also has property $\mathrm{P}(n-2|I'|,0)$. Moreover, after back-substituting $\Theta_{A}(\tilde{R}_{i},I;\cdot)$ in \eqref{eqn:pf_theta_state_expn_new} for $i = 1,2,\ldots,n-1$, it could easily be seen that $\Theta_{A}(R,I;\cdot)$ is given by
\begin{equation*}
\Theta_{A}(R,I;\cdot) = A^{n^{2}} \, \Theta_{A}(\tilde{R}_{n},I;\cdot) - \sum_{(R',I') \in \tilde{\mathcal{P}}'} Q_{R',I'}(A) \, \Theta_{A}(R',I' \oplus I;\cdot),
\end{equation*}
where 
\begin{equation*}
\tilde{\mathcal{P}}' = \mathcal{P}'_{0} \cup \mathcal{P}'_{1} \cup \cdots \cup \mathcal{P}'_{n-1},    
\end{equation*}
and as one may verify, $(\tilde{R}_{n},I)$ satisfies properties i)-iii). Consequently, $(R,I)$ with $u < j$ and $v \geq 0$ has property $\mathrm{P}(n,0)$. This also concludes b).

For c), since $h = \mathrm{ht}(R) \geq 1$, by Proposition~3.11 of \cite{DW2022},
\begin{equation*}
\Theta_{A}(R,I;\cdot) = \sum_{(R',I') \in \tilde{\mathcal{P}}'} Q_{R',I'}(A) \, \Theta_{A}(R',I' \oplus I;\cdot),
\end{equation*}
for some $\tilde{\mathcal{P}}' \subset \mathcal{W}$, where $R' \in \mathcal{R}_{n-2|I'|,0}$ and $\mathrm{ht}(R') = h-1$ for $R'$ in each corresponding pair $(R',I') \in \tilde{\mathcal{P}}'$. Hence, by induction on $n$ if $|I'| > 0$ or by induction on $h$ if $|I'| = 0$ (for $R' \in \mathcal{R}_{n,0}$, and using a) and b) as the base case when $\mathrm{ht}(R') = 0$), we see that $(R',I' \oplus I)$ has property $\mathrm{P}(n-2|I'|,0)$ for any $(R',I') \in \tilde{\mathcal{P}}'$. It follows that $(R,I)$ with $\mathrm{ht}(R) \geq 1$ has property $\mathrm{P}(n,0)$. This concludes the case when $n = n_{0}$ and $k = 0$.

Assume that $k > 0$ and let $R \in \mathcal{R}_{n,k}$. If $n_{b}(R) > n$ then, as we argued above, each pair $(R,I)$ has property $\mathrm{P}(n,k)$. So we assume that $n_{b}(R) \leq n$ and consider the following cases for $R$:
\begin{enumerate}
\item[a\textquotesingle)] $R = L(0,n)$,
\item[b\textquotesingle)] $\mathrm{ht}(R) = 0$ and $R \neq L(0,n)$, or
\item[c\textquotesingle)] $\mathrm{ht}(R) \geq 1$.
\end{enumerate}
As one may see, a\textquotesingle) and c\textquotesingle) can be argued using the same arguments as cases a) and c). For b\textquotesingle), we consider the following cases:
\begin{enumerate}
\item[1\textquotesingle)] $u = j$ and $v = 0$,
\item[2\textquotesingle)] $u = j$ and $v > 0$, or
\item[3\textquotesingle)] $u < j$ and $v \geq 0$,
\end{enumerate} 
where $u,j$, and $v$ are defined as before.

Case 1\textquotesingle) can be argued as the case 1). For case 2\textquotesingle), we can modify the argument used in the case 2). After representing
\begin{equation*}
\tilde{R} = \psi_{n}(\{j\}) * R = \tilde{R}' * \tilde{M},    
\end{equation*}
as a vertical product of a top state $\tilde{R}'$ and a middle state $\tilde{M}$, we notice that $\tilde{R}' \in \mathcal{R}_{n,k-1}$ and $\tilde{M} \in \mathcal{R}_{n',1}$, where $n' = n_{b}(\tilde{R}')$. Hence, by induction on $k$, one can show that $(\tilde{R},I)$ has property $\mathrm{P}(n,k)$ and consequently, using the same arguments as in the case 2), $(R,I)$ with $u = j$ and $v > 0$ has property $\mathrm{P}(n,k)$. For case 3\textquotesingle), one modifies the argument from case 3) as follows. After representing
\begin{equation*}
R' = \tilde{R}' * \tilde{M}, 
\end{equation*}
as a vertical product of a top state $\tilde{R}'$ and a middle state $\tilde{M}$ in each corresponding pair $(R',I') \in \mathcal{P}'_{i}$ with $|I'| = 0$, we see that $\tilde{R}' \in \mathcal{R}_{n,k'}$ and $\tilde{M} \in \mathcal{R}_{n',k-k'}$, where $k' = \max\{k-i,0\}$ and $n' = n_{b}(\tilde{R}')$. Hence, by induction on $k$ when $i > 0$ or by either cases 1\textquotesingle) or 2\textquotesingle) when $i = 0$, one can show that each $(R',I' \oplus I)$ has property $\mathrm{P}(n-2|I'|,k)$ and consequently, using the same arguments as in the case 3), $(R,I)$ has property $\mathrm{P}(n,k)$. Therefore, we proved case b\textquotesingle). 
\end{proof}

For a roof state $R$ with $\mathrm{ht}(R) \geq 1$ and $n_{t}(R) = n$, let $\mathcal{D}_{n}$ be the set of all pairs $(J,I)$ of finite subsets of $\mathbb{Z}_{\geq 0}$ such that $J = \{j_{1},j_{2},\ldots,j_{t+1}\}$ and $I = \{i_{1},i_{2},\ldots,i_{t}\}$, where $t \geq 0$ and $0 \leq j_{k} < i_{k} < j_{k+1} \leq n$ for $k = 1,2,\ldots,t$. As in \cite{DW2022}, define
\begin{equation*}
\mathcal{H}(R) = \{(J,I) \in \mathcal{D}_{n} \mid J \subseteq \mathcal{J}(R)\}.
\end{equation*}

\begin{proof}[Proof of Theorem~\ref{thm:vertical_factor_thm}]
Let $\Lambda$ be a local family of arcs that vertically factorizes a Catalan state $C$. Among all pairs of integers $0 \leq j_{1} \leq j_{2} \leq m$ which satisfy condition ii) of Definition~\ref{def:vertical_factor}, we choose a pair with $j_{1}$ maximal and $j_{2}$ minimal and denote these numbers by $j_{1}(\Lambda, C)$ and $j_{2}(\Lambda, C)$, respectively. We prove \eqref{eqn:vertical_factor_thm} by induction on $j_{1}^{*} = j_{1}(\Lambda, C)$. 

For $j_{1}^{*} = 0$, we first consider the case when $E_{\Lambda} \cap \{y'_{1},y'_{2},\ldots,y'_{m}\} =  \emptyset$, where $m = \mathrm{ht}(C)$. Let $C = R * F$, where $R$ is a roof state with $\mathrm{ht}(R) = j_{2}^{*} = j_{2}(\Lambda, C)$ and $F$ is a floor state. As it could easily be seen, there is a roof state $R_{\Lambda}$ with $\mathrm{ht}(R_{\Lambda}) = j_{2}^{*}$ such that
\begin{equation*}
C_{\Lambda} = R_{\Lambda} * F.
\end{equation*}
We find the maximal $j$ such that $y_{j} \in E_{\Lambda}$ and then, as one may check,
\begin{equation*}
F^{*} \in \mathcal{R}_{n,j-j_{2}^{*}},
\end{equation*}
where $n = n_{t}(C)$. Therefore, by Lemma~\ref{lem:theta_state_expn_new}, $(F^{*},\emptyset)$ has a $\Theta_{A}$-state expansion 
\begin{equation*}
\Theta_{A}(F^{*},\emptyset;\cdot) = \sum_{(R',I') \in \mathcal{P}'} Q_{R',I'}(A) \, \Theta_{A}(R',I';\cdot)
\end{equation*}
with a family $\mathcal{P}'$ that satisfies properties i)-iii) of Lemma~\ref{lem:theta_state_expn_new}. Hence,
\begin{equation*}
C(A) = \Theta_{A}(R,\emptyset;F) = \Theta_{A}(F^{*},\emptyset;R^{*}) = \sum_{(R',I') \in \mathcal{P}'} Q_{R',I'}(A) \, \Theta_{A}(R',I';R^{*})
\end{equation*}
and 
\begin{equation*}
C_{\Lambda}(A) = \Theta_{A}(R_{\Lambda},\emptyset;F) = \Theta_{A}(F^{*},\emptyset;R_{\Lambda}^{*}) = \sum_{(R',I') \in \mathcal{P}'} Q_{R',I'}(A) \, \Theta_{A}(R',I';R_{\Lambda}^{*}).
\end{equation*}
For a pair $(R',I') \in \mathcal{P}'$, define 
\begin{equation*}
C' = (R' * (R^{*})_{I'})^{*} \quad \text{and} \quad C'' = (R' * (R_{\Lambda}^{*})_{I'})^{*}.
\end{equation*}
As it could easily be seen, $C'$ and $C''$ are either both $K_{0}$ or Catalan states with no bottom returns. In the former case, clearly
\begin{equation*}
\Theta_{A}(R',I';R^{*}) = 0 = C_{T(\Lambda)}(A) \cdot \Theta_{A}(R',I';R_{\Lambda}^{*}).
\end{equation*}
In the latter case, since no arc from $\Lambda$ is removed while $C'$ is constructed, there is a local family of arcs $\Lambda'$ in $C'$ corresponding to $\Lambda$. Clearly, $T(\Lambda') = T(\Lambda)$ and $\Lambda'$ vertically factorizes $C'$. As one may check $C'' = (C')_{\Lambda'}$ and, since $(R',I') \in \mathcal{P}'$, the local family of arcs $\Lambda'$ and Catalan state $C'$ satisfy assumptions of Lemma~\ref{lem:plucking_poly_C_Lambda}. Therefore, $C'$ is realizable if and only if $C''$ is realizable. If both $C'$ and $C''$ are not realizable, then clearly 
\begin{equation*}
\Theta_{A}(R',I';R^{*}) = C'(A) = 0 = C_{T(\Lambda)}(A) \cdot C''(A) = C_{T(\Lambda)}(A) \cdot \Theta_{A}(R',I';R_{\Lambda}^{*}). 
\end{equation*}
If both $C'$ and $C''$ are realizable then by Lemma~\ref{lem:beta_relations_C_Lambda},
\begin{equation*}
\beta(C')-\beta(C'') = \beta(C_{T(\Lambda')}) - \lambda^{2},
\end{equation*}
where $\lambda$ is the number of elements of $\Lambda'$ (hence also of $\Lambda$). Since $T(\Lambda') = T(\Lambda)$, 
after applying Theorem~\ref{thm:coef_no_bot_rtn} and Lemma~\ref{lem:plucking_poly_C_Lambda}, we see that
\begin{eqnarray*}
C'(A) &=& A^{2\beta(C')-m'n} \, Q_{A^{-4}}^{*}(T(C'),v',\alpha') \\
&=& A^{2(\beta(C'')+\beta(C_{T(\Lambda)})-\lambda^{2})-m'n} \, Q_{A^{-4}}^{*}(T(\Lambda),v,\alpha_{0}) \cdot Q_{A^{-4}}^{*}(T(C''),v'',\alpha'') \\
&=& C_{T(\Lambda)}(A) \cdot C''(A),
\end{eqnarray*}
where $m' = \mathrm{ht}(C') = \mathrm{ht}(C'')$, $(T(C'),v',\alpha')$ and $(T(C''),v'',\alpha'')$ are the plane rooted trees with delay corresponding to $C'$ and $C''$, respectively, and $\alpha_{0} \equiv 1$. Therefore, we showed that for any pair $(R',I') \in \mathcal{P}'$,
\begin{equation*}
\Theta_{A}(R',I';R^{*}) = C_{T(\Lambda)}(A) \cdot \Theta_{A}(R',I';R_{\Lambda}^{*}).
\end{equation*}
Consequently,
\begin{eqnarray*}
C(A) &=& \sum_{(R',I') \in \mathcal{P}'} Q_{R',I'}(A) \, \Theta_{A}(R',I';R^{*}) = \sum_{(R',I') \in \mathcal{P}'} Q_{R',I'}(A) \,  C_{T(\Lambda)}(A) \cdot \Theta_{A}(R',I';R_{\Lambda}^{*}) \\
&=& C_{T(\Lambda)}(A) \cdot C_{\Lambda}(A).
\end{eqnarray*}

For the case $E_{\Lambda} \cap \{y_{1},y_{2},\ldots,y_{m}\} =  \emptyset$. Let $\overline{C}$ and $\overline{\Lambda}$ denote reflections of $C$ and $\Lambda$ about a vertical line, respectively. Clearly, $\overline{\Lambda}$ vertically factorizes $\overline{C}$, $E_{\overline{\Lambda}} \cap \{y'_{1},y'_{2},\ldots,y'_{m}\} =  \emptyset$, and $j_{1}(\overline{\Lambda}, \overline{C}) = 0$ so, as it was shown above,
\begin{equation*}
\overline{C}(A) = C_{T(\overline{\Lambda})}(A) \cdot \overline{C}_{\overline{\Lambda}}(A).
\end{equation*}
Since 
\begin{equation*}
C(A) = \overline{C}(A^{-1}),\ C_{T(\Lambda)}(A) = C_{T(\overline{\Lambda})}(A^{-1}),\ \text{and}\ C_{\Lambda}(A) = \overline{C}_{\overline{\Lambda}}(A^{-1}),  
\end{equation*}
it follows that
\begin{equation*}
C(A) = \overline{C}(A^{-1}) = C_{T(\overline{\Lambda})}(A^{-1}) \cdot \overline{C}_{\overline{\Lambda}}(A^{-1}) = C_{T(\Lambda)}(A) \cdot C_{\Lambda}(A).
\end{equation*}
This concludes the base case for induction when $j_{1}^{*} = 0$.

Assume that $j_{1}^{*} > 0$ and \eqref{eqn:vertical_factor_thm} holds for any local family of arcs $\Lambda'$ of a Catalan state $C'$ that vertically factorizes $C'$ and for which $j_{1}(\Lambda',C') < j_{1}^{*}$. Let $\Lambda$ be a local family of arcs that vertically factorizes $C \in \mathrm{Cat}(m,n)$ and $j_{1}(\Lambda,C) = j_{1}^{*}$. We choose roof states $R$, $R_{\Lambda}$ and a floor state $F$ as before, so that $C = R*F$, $C_{\Lambda} = R_{\Lambda}*F$, and 
\begin{equation*}
\mathrm{ht}(R) = \mathrm{ht}(R_{\Lambda}) = j_{2}(\Lambda,C) \geq j_{1}^{*} \geq 1.
\end{equation*}
Since $\mathcal{J}(R) = \mathcal{J}(R_{\Lambda})$, we let
\begin{equation*}
\mathcal{H} = \mathcal{H}(R) = \mathcal{H}(R_\Lambda).
\end{equation*}
By Proposition~3.11 of \cite{DW2022},
\begin{equation*}
C(A) = \Theta_{A}(R,\emptyset;F) = \sum_{(J,I) \in \mathcal{H}} A^{-n+2(\Vert{J}\Vert-\Vert{I}\Vert)} \, \Theta_{A}(R^{J},I;F) 
\end{equation*}
and
\begin{equation*}
C_{\Lambda}(A) = \Theta_{A}(R_{\Lambda},\emptyset;F) = \sum_{(J,I) \in \mathcal{H}} A^{-n+2(\Vert{J}\Vert-\Vert{I}\Vert)} \, \Theta_{A}((R_\Lambda)^{J},I;F),
\end{equation*}
where $\Vert{K}\Vert$ equals the sum of all elements of $K$. For $(J,I) \in \mathcal{H}$, define
\begin{equation*}
C' = R^{J} * F_{I}\ \text{and}\ C'' = (R_\Lambda)^{J} * F_{I}
\end{equation*}
and notice that either $C'$ and $C''$ are $K_{0}$ or they are both Catalan states. In the former case, clearly
\begin{equation*}
\Theta_{A}(R^{J},I;F) = 0 = C_{T(\Lambda)}(A) \cdot \Theta_{A}((R_{\Lambda})^{J},I;F).
\end{equation*}
In the latter case, denote by $\Lambda'$ the local family of arcs of $C'$ corresponding to $\Lambda$. As it could easily be seen $\Lambda'$ vertically factorizes $C'$, $C'' = (C')_{\Lambda'}$, and $T(\Lambda') = T(\Lambda)$. Moreover, since $j_{1}(\Lambda';C') = j_{1}^{*}-1$, it follows by induction that
\begin{equation*}
\Theta_{A}(R^{J},I;F) = C'(A) = C_{T(\Lambda')}(A) \cdot C''(A) = C_{T(\Lambda)}(A) \cdot \Theta_{A}((R_{\Lambda})^{J},I;F).
\end{equation*}
Consequently, 
\begin{eqnarray*}
C(A) &=& \sum_{(J,I) \in \mathcal{H}} A^{-n+2(\Vert{J}\Vert-\Vert{I}\Vert)} \, \Theta_{A}(R^{J},I;F) \\
&=& \sum_{(J,I) \in \mathcal{H}} A^{-n+2(\Vert{J}\Vert-\Vert{I}\Vert)} \, C_{T(\Lambda)}(A) \cdot \Theta_{A}((R_\Lambda)^{J},I;F) \\
&=&  C_{T(\Lambda)}(A) \cdot C_{\Lambda}(A).
\end{eqnarray*}
Therefore, \eqref{eqn:vertical_factor_thm} follows by induction.

Finally, since $C_{T(\Lambda)}(A) \neq 0$, using \eqref{eqn:vertical_factor_thm} we see that, $C(A)\neq 0$ if and only if $C_{\Lambda}(A) \neq 0$. Therefore, by Theorem~\ref{thm:coef_nonzero}, $C$ is realizable if and only if $C_{\Lambda}$ is realizable.
\end{proof}

\begin{example}
\label{ex:vertical_factor_thm}
For the Catalan state $C$ with a local family of arcs $\Lambda$ of $C$ shown in Figure~\ref{fig:C_Lambda}(a), as one may check,
\begin{equation*}
C_{T(\Lambda)}(A) = A^{-2}+A^{2} \quad \text{and} \quad C_{\Lambda}(A) = A^{-12} + 2 A^{-8} + 3 A^{-4} + 2 + A^{4}.
\end{equation*}
Therefore, by Theorem~\ref{thm:vertical_factor_thm},
\begin{equation*}
C(A) = C_{T(\Lambda)}(A) \cdot C_{\Lambda}(A) = A^{-14} + 3 A^{-10} + 5 A^{-6} + 5 A^{-2} + 3 A^{2} + A^{6}.    
\end{equation*}
\end{example}

\section*{Acknowledgement} Authors would like to thank Professor J\'{o}zef H. Przytycki for all valuable discussions and suggestions.

\bibliography{mybibfile}

\end{document}